\documentclass{article}[10pt]
\usepackage{amsmath,amssymb,amsthm,txfonts}
 \usepackage{graphicx}
 \usepackage{color}

\usepackage{tikz}
\newcommand*{\bfrac}[2]{\genfrac{[}{]}{0pt}{}{#1}{#2}}
\theoremstyle{plain}%{\bf}{\it}
\newtheorem{Theorem}{Theorem} 

\newtheorem{Definition}[Theorem]{Definition}
\newtheorem{definition}[Theorem]{Definition}
\newtheorem{Lemma}[Theorem]{Lemma}
\newtheorem{Proposition}[Theorem]{Proposition}
\newtheorem{Corollary}[Theorem]{Corollary}

\newtheorem{Conjecture}[Theorem]{Conjecture}

\theoremstyle{definition}%{\bf}{\rm}

\theoremstyle{remark}%{\it}{\rm}
\newtheorem{examplex}[Theorem]{Example}
\newenvironment{example}
  {\pushQED{\qed}\examplex}
  {\popQED\endexamplex}

\newtheorem{Remark}[Theorem]{Remark}

\newcommand{\off}{\textrm{C}}
\newcommand{\Sm}{\textrm{Sm}}
\makeindex
\newcommand{\mult}{\textrm{mult}}

\newcommand{\kT}{\mathcal{T}}
\newcommand{\kN}{\mathcal{N}}
\newcommand{\kB}{\mathcal{B}}

\newcommand{\kA}{\mathcal{A}}
\newcommand{\kM}{\mathcal{M}}
\newcommand{\kF}{\mathcal{F}}
\newcommand{\kS}{\mathcal{S}}
\newcommand{\kP}{\mathcal{P}}

\newcommand{\NN}{\mathbb{N}}

\newcommand{\st}{ such that\;}
\newcommand{\cemu}{ Cerlienco-Mureddu \;}

\newcommand{\cG}{{\sf{G}}}
\newcommand{\cN}{{\sf{N}}}
\newcommand{\cT}{{\sf{T}}}
  
\newcommand{\ck}{{\bf{k}}}

\newcommand{\cB}{{\sf{B}}}
\newcommand{\cL}{{\sf{L}}}

\newcommand{\bM}{{\overline{M}}}
\title{Bar code for monomial ideals}
\author{Michela Ceria\\
\multicolumn{1}{p{.7\textwidth}}{\centering\emph{Department of Mathematics\\ University of Trento\\ Via Sommarive 14, 38123, Trento} {\small \texttt{michela.ceria@unitn.it}}}}
\date{} 
\begin{document}
\maketitle
\begin{abstract}
Aim of this paper is to count $0$-dimensional stable and strongly stable ideals 
in $2$ and $3$ variables, given their (constant) affine Hilbert polynomial.
 
To do so, we define the \emph{Bar Code}, a bidimensional structure representing 
any finite  set of terms $M$ and allowing to  desume many properties of the 
corresponding monomial ideal $I$, if $M$ is an order ideal.
Then, we use it to give a connection between (strongly) stable monomial ideals 
and integer partitions, thus allowing to count them via known determinantal 
formulas.
\end{abstract}

\section{Introduction}\label{Introduction}
% % % Strongly stable ideals play a special role in the study 
% % % of Hilbert scheme, introduced first by Grothendieck \cite{Gro}, since their 
% % %  escalier allows to study the Hilbert function of any ideal, exploiting 
% % %   the theory of Groebner bases,  as pointed out  by Bayer \cite{Bay} and Eisenbud \cite{Ei}.
% % % 
% % % The notions of generic initial ideal (originarily known as Grauert invariant) and strongly stable ideal  were introduced by Galligo \cite{GAL} who pointed out that the generic initial ideal of any homogeneous
% % % ideal is closed w.r.t the action of the Borel group.
% % % 
% % % Eisenbud and Peeva \cite{EI,PEEVA}, in their studies, focused on monomial ideals closed w.r.t the action of the Borel group, i.e. strongly stable ideals, labelling them 
% % % as $0$-Borel-fixed ideals.
% % % Later, Herzog \cite{Her} and Aramova-Herzog \cite{AH} renamed them strongly
% % %  stable ideals.

Strongly stable ideals play a special role in the study 
of Hilbert scheme, introduced first by Grothendieck \cite{Gro}, since their escalier allows to study the Hilbert function of any homogeneous ideal, exploiting the theory of Groebner bases,  as pointed out  by Bayer \cite{Bay} and Eisenbud \cite{Ei}.

The notion of generic initial ideal was introduced by Galligo \cite{GAL} with the name of \emph{Grauert invariant}.
Galligo proved that the generic initial ideal of any homogeneous ideal is closed w.r.t the action of the Borel group and gave a combinatorial characterization of such ideals, provided that they are defined on a field of characteristic zero.
Also Eisenbud and Peeva \cite{Ei,PEEVA}, focused on that monomial ideals, labelling them \emph{$0$-Borel-fixed ideals}.
Later,  Aramova-Herzog \cite{AH2, AH} renamed them \emph{strongly stable ideals}.

A combinatorial description of the ideals  closed w.r.t the action of the Borel group over a polynomial ring on a field of characteristic $p>0$ has been provided by Pardue in his Thesis \cite{Pardue} and Galligo's result has been extended to that setting by Bayer-Stillman \cite{BS}.

The notion of \emph{stable ideal} has been introduced by Eliahou-Kervaire \cite{EK}
 as a generalization of $0$-Borel-fixed ideals. They were able to give
  a minimal resolution for stable ideals.
  
Such minimal resolution was used by Bigatti \cite{Bi} and Hulett \cite{Hu}
 to extend Macaulay's result \cite{MaC};
 they proved that the  lex-segment ideal  has maximal Betti numbers, among all ideals sharing the same Hilbert function.

 In connection with the study of Hilbert schemes \cite{BLR, BCLR, Cioffi, LR, Moo, Ree} it has been considered relevant to list all the  stable ideals \cite{Bertone}  and strongly stable ideals \cite{Ciolemaro, Lella} with a fixed Hilbert polynomial.

Aim of this paper is to count zerodimensional stable and strongly stable ideals in $2$ and $3$ variables, given their (constant) affine Hilbert polynomial.
 
To do so, we first introduce a bidimensional structure, called \emph{Bar Code}
 which allows, a priori, to represent any (finite\footnote{There is also the possibility to have \emph{infinite} Bar Codes for infinite
  sets of terms, but it is out of the purpose of this paper, so we will only see an example for completeness' sake.}) set of terms $M$ and, 
if $M$ is an order ideal, to authomatically desume many properties of the corresponding monomial ideal $I$. For example, a Pommaret basis \cite{SeiB, CMR} of $I$ can be easily desumed.

The Bar Code is strictly connected to Felzeghy-Rath-Ronyay's Lex Trie \cite{Lex, Lund}, even if our goal and methods are completely different from theirs.

Using the Bar Code, we provide a connection between stable and strongly stable monomial ideals and integer partitions.

For the case of two variables, we see that there is a biunivocal correspondence between (strongly) stable ideals with affine Hilbert polynomial
 $p $ and partitions of $p$ with distinct parts.

 The case of three variables is more complicated and some more technology is required. Thanks to the Bar Code, we provide a bijection between (strongly) stable ideals  and some special plane partitions of their constant affine Hilbert polynomial $p$.

These plane partitions have been  studied by Krattenthaler \cite{Krat, Krat2}, who proved determinantal formulas to find their norm generating functions and - finally - to count them.

As an example, we consider the  stable monomial ideal $$I_1=(x_1^3, 
x_1x_2,x_2^2,x_1^2x_3,x_2x_3,x_3^2) \triangleleft \ck[x_1,x_2,x_3],$$ 
whose Groebner escalier is $\cN(I_1)=\{1,x_1,x_1^2,x_2,x_3,x_1x_3\}$.

It can be represented by the Bar Code below 
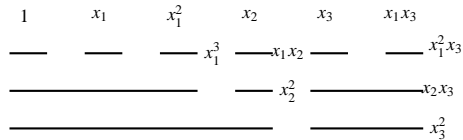
\begin{center}
\begin{tikzpicture}
\node at (4.7,0.5) [ ] {$\scriptstyle{1}$};
\node at (5.7,0.5) [ ] { $\scriptstyle{x_1}$};
\node at (6.7,0.5) [ ] { $\scriptstyle{x_1^2}$};
\node at (7.7,0.5) [ ] {$\scriptstyle{x_2}$};
\node at (8.7,0.5) [ ] { $\scriptstyle{x_3}$};
\node at (9.7,0.5) [] { $\scriptstyle{x_1x_3}$};
\draw [thick] (4.5,0)--(5,0);
\draw [thick] (5.5,0)--(6,0);
\draw [thick] (6.5,0)--(7,0);
\node at (7.2,0) [] {$\scriptstyle{x_1^3}$};
\draw [thick] (7.5,0)--(8,0);
\node at (8.2,0) [] {$\scriptstyle{x_1x_2}$};
\draw [thick] (8.5,0)--(9,0);
\draw [thick] (9.5,0)--(10,0);
\node at (10.3,0.1) [] {$\scriptstyle{x_1^2x_3}$};

\draw [thick] (4.5,-0.5)--(7,-0.5);
\draw [thick] (7.5,-0.5)--(8,-0.5);
\node at (8.2,-0.5) [] {$\scriptstyle{x_2^2}$};

\draw [thick] (8.5,-0.5)--(10,-0.5);
\node at (10.2,-0.5) [] {$\scriptstyle{x_2x_3}$};
\draw [thick] (4.5,-1)--(8,-1);
\draw [thick] (8.5,-1)--(10,-1);
\node at (10.2,-1) [] {$\scriptstyle{x_3^2}$};

\end{tikzpicture}
\end{center}
and it corresponds to the plane partition 
 \[\begin{array}{cc}
                                              \color{red}{3} & \color{blue}{1} \\ \; \color{green}{2}&
                                            \end{array}\]
% % \[\begin{array}{cc}
% %                                               \textcolor{red}{3} & \textcolor{blue}{1} \\ \; \textcolor{green}{2}& 
% %                                             \end{array}\]
The correspondence can be seen observing the rows of the Bar Code 
above: since the bottom row is composed by two segments, the plane partition has exactly two rows.
The number of entries in the $i$-th row of the partition, $i=1,2$ ({\em i.e.} 2 and 1 resp.), is given
 by the number of segments in the middle-row, lying over the 
$i$-th segment of the bottom row.
Finally, the entries are represented by the number of segments in the top row, lying over the segments representing the corresponding entry.

Exploiting this bijection and the determinantal formulas by Krattenthaler,
 we are finally able to count stable and strongly stable ideals in three variables.

Even if the Bar Code can easily represent finite sets of terms
 in any number of variables, the generalization of our results to the case of $4$ or more variables
  would require the introduction of $n$-dimensional partitions, for which, in my knowledge, it does not exist a complete study 
   from the point of view of counting them\footnote{In \cite{And}, Chapter 11, the author observes:
   \begin{quote}
                  \emph{Surprisingly, there is much 
                  of interest when the dimension is $1$ or $2$, and very little when the dimension
                  exceeds $2$.}                                            \end{quote}}, so, in this paper, we do not extensively deal with them.

% % % % A link between Hilbert schemes and the theory of Groebner bases has been explixity given by Bayer \cite{Bay} and Eisenbud.
% % % % 
% % % % 
% % % % 
% % % % A fundamental concept in algebraic geometry is the 
% % % %  Hilbert scheme, introduced first by Grothendieck \cite{Gro}
% % % %   and extensively studied across the centuries by many mathematicians.
% % % %  

% % % % \\ 
% % % % \medskip 
% % % % 
% % % % \textbf{Ricordare Ungheresi e Lundqvist
% % % % \\
% % % % CITARE SOFTWARE DI PAOLO E CRISTINA} 
\section{Some algebraic notation}\label{Notazioni}
Throughout this paper, in connection with monomial ideals, we mainly follow the notation of \cite{SPES}.
\\
We denote by $\mathcal{P}:=\mathbf{k}[x_1,...,x_n]$ the graded ring of polynomials in
$n$ variables with coefficients in the field $\ck$, assuming, once for all, that $char(\ck)=0$.\\
The \emph{semigroup of terms}, generated by the set $\{x_1,...,x_n\}$ is:
$$\mathcal{T}:=\{x^{\gamma}:=x_1^{\gamma_1}\cdots
x_n^{\gamma_n} \vert \,\gamma:=(\gamma_1,...,\gamma_n)\in \NN^n \}.$$
If $\tau=x_1^{\gamma_1}\cdots x_n^{\gamma_n}$, then $\deg(\tau)=\sum_{i=1}^n
\gamma_i$ is the \emph{degree} of $\tau$ and, for each $h\in \{1,...,n\}$
$\deg_h(\tau):=\gamma_h$ is the $h$-\emph{degree} of $\tau$.\\
\\
For each $d \in \NN$, $\mathcal{T}_d$ is the $d$-degree part of
$\mathcal{T}$, i.e.
$\mathcal{T}_d:=\{x^\gamma \in \kT \vert \, \deg(x^\gamma)=d\}$ 
and it is well known that $\vert \mathcal{T}_d \vert = {n+d-1 \choose d}$. For
each subset $M\subseteq \mathcal{T}$ we set $M_d=M\cap
\mathcal{T}_d$. The symbol $\mathcal{T}(d)$ denotes the degree $\leq d$ part of
$\mathcal{T}$, namely $\mathcal{T}(d)=\{x^\gamma \in \kT\vert \, \deg(x^\gamma)\leq d\}$.
Analogously, $\mathcal{P}(d)$  denotes the degree $\leq d$ part of
$\mathcal{P}$ and given an ideal $I$ of $\mathcal{P}$, $I(d)$ is its degree $\leq d$ part, i.e.  $I(d)=I\cap \mathcal
{P}(d)$.\\
We notice that $\mathcal{P}(d)$ is the vector space generated by
 $\mathcal{T}(d)$ and we observe that $I(d)$ is a vector subspace of  $\mathcal{P}(d)$.\\
  \smallskip
% %  
% % \noindent For each term $\tau \in \mathcal{T}$ and $x_j \vert  \tau$, the only $\upsilon
% % \in \mathcal{T}$ \st $\tau=x_j\upsilon$ is called $j$-th \emph{predecessor} of
% % $\tau$.\\
\noindent A \emph{semigroup ordering} $<$ on $\mathcal{T}$  is  a total ordering
\st
$ \tau_1<\tau_2 \Rightarrow \tau\tau_1<\tau\tau_2,\, \forall \tau,\tau_1,\tau_2
\in \mathcal{T}$. For each semigroup ordering $<$ on $\mathcal{T}$,  we can represent a polynomial
$f\in \mathcal{P}$ as a linear combination of terms arranged w.r.t. $<$, with
coefficients in the base field $\mathbf{k}$:
$$f=\sum_{\tau \in \mathcal{T}}c(f,\tau)\tau=\sum_{i=1}^s c(f,\tau_i)\tau_i:\,
c(f,\tau_i)\in
\mathbf{k}^*,\, \tau_i\in \mathcal{T},\, \tau_1>...>\tau_s,$$ with
$\cT(f):=\tau_1$   the 
\emph{leading term} of $f$, $Lc(f):=c(f,\tau_1)$ the  \emph{leading
coefficient} 
of $f$ and $tail(f):=f-c(f,\cT(f))\cT(f)$  the 
\emph{tail} of $f$.
\\
A \emph{term ordering} is a semigroup ordering \st $1$ is lower 
than every variable or, equivalently, it is a \emph{well ordering}.\\
Unless otherwise specified, we consider the \emph{lexicographical ordering} 
induced
by\\ $x_1<...<x_n$, i.e:
$$ x_1^{\gamma_1}\cdots x_n^{\gamma_n}<_{Lex} x_1^{\delta_1}\cdots
x_n^{\delta_n} \Leftrightarrow \exists j\, \vert  \,
\gamma_j<\delta_j,\,\gamma_i=\delta_i,\, \forall i>j, $$
which is a term ordering.

Since in all the paper we will consider the lexicographical
ordering,  no confusion may arise and so we drop the subscript and denote it by $<$ instead of $<_{Lex}$.\\
 
\noindent For each term $\tau \in \mathcal{T}$ and $x_j \vert  \tau$, the only $\upsilon
\in \mathcal{T}$ \st $\tau=x_j\upsilon$ is called $j$-th \emph{predecessor} of
$\tau$.\\
Given a term $\tau \in \kT$, we  denote by $\min(\tau)$ the smallest 
variable $x_i$, $i \in \{1,...,n\}$, s.t. $x_i\mid \tau$.\\
For $M \subset \mathcal{T}$, we  denote by $\bM$ the list 
obtained by ordering the elements of $M$ increasingly w.r.t. Lex. For example,
if $M=\{x_2,x_1^2\}\subset \ck[x_1,x_2],\, x_1<x_2$, $\bM=\{x_1^2,x_2\}$.
\\
\smallskip

A subset $J \subseteq \kT$ is a \emph{semigroup ideal} if  $\tau \in J 
\Rightarrow \sigma\tau \in J,\, \forall \sigma \in \mathcal{T}$; a subset ${\sf N}\subseteq \mathcal{T}$ is an \emph{order ideal} if
$\tau \in {\sf N} \Rightarrow \sigma \in {\sf N}\, \forall \sigma \vert \tau$.  We have that ${\sf N}\subseteq \mathcal{T}$ is an order ideal if and only if 
$\mathcal{T}\setminus {\sf N}=J$ is a semigroup ideal.
\\
\smallskip

Given a semigroup ideal $J\subset\mathcal{T}$  we define ${\sf 
N}(J):=\mathcal{T}\setminus J$. The minimal set of generators ${\sf G}(J)$ of $J$, called the \emph{monomial basis} of $J$, satisfies the conditions below
\begin{eqnarray*}
{\sf G}(J)&:=&\{\tau \in J\, \vert \, \textrm{ each predecessor of }\, \tau
\in \cN(J)\}\\
&=&\{\tau \in \mathcal{T}\, \vert \,\cN(J)\cup\{\tau\}\, \textrm{is an order ideal},
\, \tau \notin \cN(J)\}.
\end{eqnarray*}
% \begin{eqnarray*}
% &{\sf G}(J):=\{\tau \in J\, \vert \, \textrm{ each predecessor of }\, \tau 
% \in \cN(J)\}=&\\
% &=\{\tau \in \mathcal{T}\, \vert \,\cN(J)\cup\{\tau\}\, \textrm{is an order ideal}, 
% \, \tau \notin \cN(J)\}.
% \end{eqnarray*}
\noindent For all subsets $G \subset \mathcal{P}$,  $\cT\{G\}:=\{\cT(g),\, g \in  G\}$ and $\cT(G)$ is the semigroup ideal
of leading terms defined as $\cT(G):=\{\tau \cT(g),\, \tau \in \mathcal{T}, g \in G\}$. 
\\
Fixed a term order $<$, for any ideal  $I
\triangleleft \mathcal{P}$ the monomial basis of the semigroup ideal 
$\cT(I)=\cT\{I\}$ is called \emph{monomial basis}  of $I$ and denoted again by $\cG(I)$,
whereas the ideal 
$In(I):=(\cT(I))$ is called \emph{initial ideal} and the order ideal 
$\cN(I):=\kT \setminus \cT(I)$ is called \emph{Groebner escalier} of $I$. The 
\emph{border 
set} of $I$ is defined as:

\begin{eqnarray*}
{\sf B}(I) &:=& \{x_h\tau,\, 1 \leq h \leq n,\, \tau \in {\sf N}(I)\}\setminus {\sf N}(I)\\
&=&\cT(I)\cap (\{1\}\cup \{x_h\tau,\, 1 \leq h \leq n,\, \tau \in {\sf N}(I)\}).
\end{eqnarray*}
% % 
% % \begin{eqnarray*}
% % &{\sf B}(I):=\{x_h\tau,\, 1 \leq h \leq n,\, \tau \in {\sf N}(I)\}\setminus {\sf N}(I)=&\\
% % &=\cT(I)\cap (\{1\}\cup \{x_h\tau,\, 1 \leq h \leq n,\, \tau \in {\sf N}(I)\}).
% % \end{eqnarray*}
\smallskip

If $I \triangleleft \mathcal{P}$ is an ideal, we define its associated \emph{variety} as
$$V(I)=\{P \in \overline{\ck}^n,\, f(P)=0, \, \forall f \in \mathcal{I}\},$$
where $\overline{\mathbf{k}}$ is the algebraic closure of $\mathbf{k}$.
\begin{Definition}\label{HilbAff}
Let $I \triangleleft \mathcal{P}$ be an ideal. The \emph{affine Hilbert
function} of $I$ is the function
$$HF_I: \NN \rightarrow \NN $$
$$ d \mapsto dim(\mathcal{P}(d)/I(d)).$$
\end{Definition}
For $d$ sufficiently large, the affine Hilbert function of $I$ can be written
as:
$$HF_I(d)=\sum_{i=0}^l b_i {d\choose l-i}, $$
where $l$ is the Krull dimension of $V(I)$, $b_i$ 
are integers called \emph{Betti numbers} and $b_0$ is
positive.
\begin{Definition}
The polynomial which is equal to $HF_I(d)$, for $d$ sufficiently large, is
called the
\emph{affine Hilbert polynomial}
of $I$ and denoted $H_I(d)$.
\end{Definition}

\section{On the Integer Partitions}\label{Partizioni}
In this section, we give some definitions and theorems from the
theory of integer partitions that we will use as a tool for our study,
mainly following \cite{And, Krat, Krat2, Sta1}.
\\
Let us start giving the definition of \emph{integer partition}.
\begin{definition}[\cite{Sta1}]\label{IntPart}
An \emph{integer partition} of $p\in \NN$ is a $k$-tuple
$(\lambda_1,...,\lambda_k)\in \NN^k$
 such that $\sum_{i=1}^k \lambda_i =p$ and
$\lambda_1 \geq ...\geq \lambda_k$.
\end{definition}
We regard two partitions as identical if they only
differ in the number of terminal 
zeros. For example
$(3,2,1)=(3,2,1,0,0)$.\\
The nonzero terms are called \emph{parts} of
$\lambda$ and we say that $\lambda$ has $k$ parts
if
$k=\vert \{i ,\, \lambda_i>0\}\vert .$
\\
We will mainly deal with the special case
$\lambda_1> ...>\lambda_k>0$
i.e. with integer partitions of $p$ into $k$ non-zero
\emph{distinct parts}, denoting
 by $I_{(p,k)}$ the set containing them, i.e.
$$I_{(p,k)}:=\{(\lambda_1,...,\lambda_k)\in \NN^k,\,
\lambda_1>...>\lambda_k>0  \textrm{ and }
\sum_{j=1}^k \lambda_j=p \}.$$
% % % We will mainly deal with the special case
% % % $\lambda_1> ...>\lambda_k$ 
% % % i.e. with integer partitions of $p$ into $k$
% % % \emph{distinct parts}, denoting 
% % %  by $I_{(p,k)}$ the set containing them, i.e.
% % % $$I_{(p,k)}:=\{(\lambda_1,...,\lambda_k)\in \NN^k,\, 
% % % \lambda_1>...>\lambda_k  \textrm{ and }
% % % \sum_{j=1}^k \lambda_j=p \}.$$
% % % \\
The number $Q(p,i)$ of integer partitions of $p$
into $i$ distinct parts is well known  
in literature. For example, we can find in
\cite{Com} the formulas allowing to compute it:
$$\forall p,i \in \NN,\, i\neq 1,\,
Q(p,i)=P\left(p-{i\choose 2}, i\right),\, Q(p,1)=1
$$
where $P(n,k)$ denotes the number of integer partitions of $n$ 
with largest part equal to $k$:
$$\forall n,k \in
\NN,\,P(n,k)=P(n-1,k-1)+P(n-k,k), $$
with
\[\left\{
\begin{array}{ll}
P(n,k)=0 \, \textrm{ for }\, k>n\\
P(n,n)=1\\
P(n,0)=0
\end{array}
\right.\]
\medskip
We define now the notion of \emph{plane
partition}.

\begin{definition}[\cite{Krat}]\label{Plp}
A \emph{plane partition} $\pi$ of a positive
integer $p\in \NN$, is a partition of $p$ in which
the parts  have been arranged in a $2$-dimensional
array,  weakly decreasing  across rows
and down columns. If the inequality is strict across rows (resp. columns), we say that the partition
is \emph{row-strict} (resp \emph{column-strict}).
\\
Different
configurations are regarded as different plane
partitions.
\\
The \emph{norm} of $\pi$ is the sum 
$n(\pi):=\sum_{i,j}\pi_{i,j}$ of all its parts.
\end{definition}

We point out that an integer partition (see Definition \ref{IntPart}) is a simple and particular case of plane partition.

\begin{example}\label{Plpex}
An example of plane partition of $p=6$ is
$$\begin{matrix} 2 & 1 & 1\\ \;&  1  & 1
\end{matrix}$$
which is different from the plane partition
$$\begin{matrix} 2 & 1 & 1\\ \;& 1\\ & 1
\end{matrix}$$
\end{example}

In sections \ref{COUNTSTAB}, \ref{StStCount}, we will be interested in some
particular plane partitions, that we define in what follows.

\begin{definition}[\cite{Krat}]\label{PlanePartNoShift}
 Let $D_r$ denote the set of all $r$-tuples 
 $\lambda=(\lambda_1, ..., \lambda_r)$ of integers with 
 $\lambda_1 \geq ...\geq \lambda_r$. 
 \\
 For $\lambda, \mu \in D_r$, we write $\lambda \geq \mu$
 if $\lambda_i \geq \mu_i$ for all $i=1,2,...,r$. Let $c,d $ arbitrary 
 integers and $\lambda, \mu \in D_r$, with $\lambda \geq \mu$. We call an array 
 $\rho$ of integers of the form 
 $$\begin{matrix} 
& & &\rho_{1,\mu_1+1}& \rho_{1,\mu_1+2}&...& ...& ... & \rho_{1,\lambda_1}\\
& \rho_{2,\mu_2+1}& ...& ...& ...& ...& ...&  \rho_{2,\lambda_2} &  \\
 &  & & ...& ...&...& ...&  & \\
 \rho_{r,\mu_r+1}& ...& ...& \rho_{r,\lambda_r} &  &  \\
\end{matrix}$$
a \emph{$(c, d)$-plane partition} of shape $\lambda/\mu$ if 
$$\rho_{i,j}\geq \rho_{i,j+1}+c \textrm{ for
 } 1\leq i\leq r,\, \mu_i< j <\lambda_i,$$

$$\rho_{i,j}\geq \rho_{i+1,j}+d \textrm{ for
 } 1\leq i\leq r-1,\, \mu_i< j \leq
\lambda_{i+1}.$$
In the case $\mu=0$, we shortly say that $\rho$ is of shape $\lambda$.
% % % % The entries of $\rho$ are called \emph{parts} of
% % % % $\rho$, whereas its 
% % % % \emph{norm} is the sum 
% % % % $n(\rho):=\sum_{i,j}\rho_{i,j}$ of all parts.
\end{definition}
We denote by $\kP_{\lambda}(c, d)$ the set of  $(c,d)$-plane partitions of shape $\lambda$.

A $(1,1)$-plane partition containing only positive parts is a row and 
column-strict plane partition; 
these partitions will be useful while dealing with stable ideals (see section 
\ref{COUNTSTAB}).

\begin{definition}[\cite{Krat2}]\label{PlanePartShaped}
Let $c,d$ be arbitrary integers and
$\lambda$ be a partition 
with $\lambda_r \geq r$. We call
``\emph{shifted} $(c,d)$-\emph{plane
partition} of 
\emph{shape} $\lambda$'' an array $\pi$ of
integers of the form

$$\begin{matrix} 
\pi_{1,1}& \pi_{1,2}&...& ...& ...& ...& ... & ... & \pi_{1,\lambda_1}\\
& \pi_{2,2}& ...& ...& ...& ...& ...&  \pi_{2,\lambda_2} &  \\
 &  & ...& ...& ...&...& ...&  & \\
& & & \pi_{r,r}& ...& ...& \pi_{r,\lambda_r} &  &  \\
\end{matrix}$$

 and for which
$$\pi_{i,j}\geq \pi_{i,j+1}+c \textrm{ for
 } 1\leq i\leq r,\, i\leq j <\lambda_i,$$

$$\pi_{i,j}\geq \pi_{i+1,j}+d \textrm{ for
 } 1\leq i\leq r-1,\, i < j \leq
\lambda_{i+1}.$$
% % % % The entries of $\pi$ are called \emph{parts} of
% % % % $\pi$, whereas its 
% % % % \emph{norm} is the sum 
% % % % $n(\pi):=\sum_{i,j}\pi_{i,j}$ of all parts
\end{definition}

We  point out that, according to definition
  \ref{PlanePartShaped}, there are 
 $\lambda_i -i +1$ integers in the $i$-th row.

We denote by $\kS_{\lambda}(c, d)$ the set of  shifted $(c,d)$-plane partitions of shape $\lambda$.
These partitions will be useful in section \ref{StStCount},
where we will count   strongly 
stable ideals.

% \begin{oss}\label{NumeroElm}
% % Note that the plane partitions introduced in Definition \ref{PlanePartShaped}
% % are a particular case of $(c, d)$-plane partitions of shape $\lambda/\mu$,  
% % with $\mu=(\mu_1,...\mu_r)$, $\mu_i=i-1$,  $1 \leq i \leq r$.
% % \\
% We  point out that, according to definition
%  \ref{PlanePartShaped}, there are 
% $\lambda_i -i +1$ integers in the $i$-th row.
% \end{oss}
\begin{example}\label{StrShift}
The plane partition 
$$\begin{matrix} 5 & 4 & 3\\  4 & 1 &
\end{matrix}$$
is a $(1,1)$-plane partition with shape $\lambda=(3,2)$
and norm $17$.
\\
On the other hand, the plane partition 
$$\begin{matrix} 5 & 4 & 3\\ \;& 4 & 1
\end{matrix}$$
is a shifted $(1,0)$-plane partition of shape
$\lambda=(3,3)$ and norm $17$.
It contains $\lambda_1=3$ elements in the
first row and $\lambda_2-1=2$ elements in the
second row.
\end{example}
\vspace{0.8cm}
% \textbf{Differently from what said for integer
% partitions,  plane partitions
% of a given norm  in general cannot be counted via
% some kind of formulas. For this reason, it is
% fundamental to define the 
%  \emph{norm generating function}.\\ Cercare la referenza}
%  
 
We introduce now the notion of  \emph{norm generating function}, for counting plane partitions.
 
\begin{definition}[\cite{Krat}]\label{NormGenFunct}
The \emph{norm generating function} for a class
$C$ of $(c,d)$-plane partitions is $$\sum_{\pi \in C}x^{n(\pi)}.$$
\end{definition}
\medskip
If $x$ is an indeterminate, we introduce the
$x$-notations (see \cite{Krat}):
$$[n]=1-x^n $$
$$[n]!=[1][2]\cdots[n],\, [0]!=1 $$
$$\bfrac{n}{k}= \frac{[n]!}{[k]![n-k]!}, \textrm{ if } n\geq k\neq0.$$
If $k=0$, $\bfrac{n}{k}=1$; if $k\neq 0$ and $n<k$, then we set $\bfrac{n}{k}=
0$.
\\
\smallskip

% % In the theorem below, we see how to enumerate $(c,d)$-plane partitions
% % of a given shape $\lambda$.
Theorems \ref{ContoKratStab} and \ref{ContoKrat} give
a way to compute the norm generating function for plane partitions of the forms introduced in Definitions \ref{PlanePartNoShift} and \ref{PlanePartShaped},
under some hypotheses on the size of their parts.
\\
Let us start with the plane partitions of Definition
\ref{PlanePartNoShift}.
\begin{Theorem}[Krattenthaler,\cite{Krat}]
\label{ContoKratStab}
Let $c,d$ be arbitrary integers, $\lambda,\mu \in D_r$  and let $a,b$ be $r$-tuples of
integers satisfying
$$a_i-c(\mu_i-\mu_{i+1})+(1-d)\geq a_{i+1} $$
$$b_i+c(\lambda_i-\lambda_{i+1})+(1-d)\geq b_{i+1} $$
for $i=1,2,...,r-1$. 

Then, denoting $N_1(s,t)=b_s(\lambda_s-s-\mu_t+t)+(1-c-d)\left[{\mu_t+s-t \choose 2}- {\mu_t \choose 2}   \right]+c {\lambda_s -s-\mu_t +t \choose 2} $, the polynomial
$$det_{1 \leq s,t \leq r}\left(x^{N_1(s,t)}\bfrac{(1-c)(\lambda_s-\mu_t)-d(s-t)+a_t-b_s+c}{\lambda_s-s-\mu_t+t}\right), $$
% where $N_1(s,t)=b_s(\lambda_s-s-\mu_t+t)+(1-c-d)\left[{\mu_t+s-t \choose 2}- {\mu_t \choose 2}   \right]+c {\lambda_s -s-\mu_t +t \choose 2} $  
is the norm generating
function for 
$(c,d)$-plane partitions of shape
$\lambda/\mu$ in which the first part in row $i$ is
at most $a_i$ and the last part in row $i$ is at
least $b_i$.
\end{Theorem}

\begin{example}\label{NumPartNoShift}
  Let us consider the $(1,1)$-plane partitions of shape
$\lambda=(2,1)$ (so $\mu=0$), such that $a=(4,3)$ and $b=(1,1)$, i.e. row and column strict plane partitions of the form
  \[\left(              
\begin{array}{cc}                  
\rho_{1,1} &\rho_{1,2}\\
\rho_{2,1}&0
\end{array} 
\right)\]
with $\rho_{1,1} \leq 4$, $1 \leq \rho_{2,1} \leq 3$, $\rho_{1,2}\geq 1$,
With the notation introduced above, we have $r=2$.

Since
$$4=a_1-c(\mu_1-\mu_{2})+(1-d)\geq a_{2}=3   $$
$$2=b_1+c(\lambda_1-\lambda_{2})+(1-d)\geq b_{2}=1, $$
 we can apply the formula of Theorem \ref{ContoKratStab}, which, substituting our data, turns out to be significantly 
simplified: 
$$det_{1 \leq s,t \leq 2}\left(x^{N_1(s,t)}\bfrac{ -(s-t)+a_t-b_s+1}{\lambda_s-s+t}\right), $$
where $N_1(s,t)=b_s(\lambda_s-s+t)+(-1)\left[{s-t \choose 2} \right]+ {\lambda_s -s +t \choose 2}$.\\
Now, we have
$N(1,1)=(2-1+1)+{2 \choose 2}=2$; $N(1,2)=(2-1+2)+{3 \choose 2}=5$;
$N(2,1)=0$; $N(2,2)=(1-2+2)=1$, so we have to compute 
$det \left(\begin{array}{cc} x^3 \bfrac{4}{2} & x^6 \bfrac{4}{3}\\
\bfrac{3}{0} & x\bfrac{3}{1} \end{array}
                              \right)
= det \left(\begin{array}{cc} x^3(1+x^2)(1+x+x^2) & x^5 (1+x)(1+x^2)\\
1 & x(1+x+x^2) \end{array}
                              \right) = x^{10}+2x^9+3x^8+3x^7+3x^6+x^5+x^4$
For example, there are exactly $3$ partitions 
with norm $8$, namely

  \[              
 \left(              
\begin{array}{cc}                  
  \mathbf{4} &1\\
\mathbf{3}&0
\end{array} 
\right), \left(              
\begin{array}{cc}                  
\mathbf{4} &2\\
\mathbf{2}&0
\end{array} 
\right), \left(
\begin{array}{cc}                  
\mathbf{4} &3\\
\mathbf{1}&0
\end{array} 
\right)\]

\end{example}

We see now how to construct the norm generating function
for the partitions of Definition \ref{PlanePartShaped}.
% 
% 
% The following theorem gives a way to compute the norm
% generating functions of shifted
% $(c,d)$-plane partitions of a given shape
% $\lambda$.
\begin{Theorem}[Krattenthaler, \cite{Krat2}]
\label{ContoKrat}
Let $c,d$ be arbitrary integers, $\lambda$ a
partition with 
$\lambda_r\geq r$ and let $a,b$ be $r$-tuples of
integers satisfying
$$a_i-c-d\geq a_{i+1} $$
$$b_i+c(\lambda_i-\lambda_{i+1})+(1-d)\geq b_{i+1} $$
for $i=1,2,...,r-1$. Then, denoting $N_1=\sum_{i=1}^r(b_i(\lambda_i-i)+a_i+c {\lambda_i-i \choose 2})$,
 the polynomial 
$$x^{N_1}det_{1 \leq s,t \leq r}\left(\bfrac{(\lambda_s-s)(1-c)+(1-c-d)(s-t)+a_t-b_s}{\lambda_s-s}\right), $$
% % where $N_1=\sum_{i=1}^r(b_i(\lambda_i-i)+a_i+c {\lambda_i-i \choose 2})$ 
is the norm generating
function for 
shifted $(c,d)$-plane partitions of shape
$\lambda$ in which the first part in row $i$ is
equal to $a_i$ and the last part in row $i$ is at
least $b_i$.
\end{Theorem}
\begin{example}\label{Matriciona}
 Let us consider the shifted $(1,0)$-plane partitions of shape
$\lambda=(3,3,3)$, such that $a=(6,3,1)$ and $b=(1,1,1)$. 
By definition, they are  matrices
  \[\left(              
\begin{array}{ccc}                  
\pi_{1,1} &\pi_{1,2} & \pi_{1,3}\\
0&\pi_{2,2}&\pi_{2,3}\\
0&0&\pi_{3,3}
\end{array}       
\right)\]
with $\pi_{1,1}=6$, $\pi_{2,2}=3$, $\pi_{3,3}=1$.  Moreover,
$\pi_{1,3},\pi_{2,3}\geq 1$.
\\
We compute the norm generating function for these partitions, via Theorem
\ref{ContoKrat}.\\
First of all $N_1=\sum_{i=1}^r(b_i(\lambda_i-i)+a_i+c {\lambda_i-i \choose 2})=14.$
 \\
 Then we have to compute  each $m_{s,t}=\bfrac{(\lambda_s-s)(1-c)+(1-c-d)(s-t)+a_t-b_s}{\lambda_s-s}$, $1 \leq s,t \leq r$ and then the determinant of  
 the matrix $M=(m_{s,t})_{1 \leq s,t \leq r}$.\\
We have:
\\$m_{1,1}=\bfrac{5}{2}=\frac{
\prod_{i=1}^5(1-x^i)}{\prod_{i=1}^2(1-x^i)\cdot
\prod_{i=1}^3(1-x^i)}=(x^2+1)(x^4+x^3+x^2+x+1)$
\\
$m_{1,2}=\bfrac{2}{2}=1$\\
$m_{1,3}=\bfrac{0}{2}=0$\\
$m_{2,1}=\bfrac{5}{1}=\frac{
\prod_{i=1}^5(1-x^i)}{\prod_{i=1}^1(1-x^i)\cdot
\prod_{i=1}^4(1-x^i)}=x^4+x^3+x^2+x+1$\\
$m_{2,2}=\bfrac{2}{1}=\frac{
\prod_{i=1}^2(1-x^i)}{\prod_{i=1}^1(1-x^i)\cdot
\prod_{i=1}^1(1-x^i)}=x+1$\\
$m_{2,3}=\bfrac{0}{1}=0$\\
$m_{3,1}=m_{3,2}=m_{3,3}=1$.
\\
This way 
 \[M=\left(              
\begin{array}{ccc}                  
(x^2+1)(x^4+x^3+x^2+x+1) &1 & 0\\
x^4+x^3+x^2+x+1&x+1&0\\
1&1&1
\end{array}       
\right),\]
so $det(M)=x^7+2x^6+3x^5+3x^4+3x^3+2x^2+x$. The generating function is then
$x^{14}det(M)=x^{15}+2x^{16}+3x^{17}+3x^{18}+3x^{19}+2x^{20}+x^{21}.$\\
If we consider, for example, $n(\pi)=17$, the coefficient of $x^{17}$ in the above
polynomial is $3$, so it tells us that there are exactly three  shifted $(1,0)$-plane
partitions of shape $\lambda=(3,3,3)$, such that $a=(6,3,1)$ and $b=(1,1,1)$.\\
We can write them down for completeness'sake:
  \[\left(              
\begin{array}{ccc}                  
\mathbf{6}&5 & 1\\
0&\mathbf{3}&1\\
0&0&\mathbf{1}
\end{array}       
\right), \; \left(              
\begin{array}{ccc}                  
\mathbf{6}&4 & 2\\
0&\mathbf{3}&1\\
0&0&\mathbf{1}
\end{array}       
\right), \; \left(              
\begin{array}{ccc}                  
\mathbf{6}&3 & 2\\
0&\mathbf{3}&2\\
0&0&\mathbf{1}
\end{array}       
\right)\]
\end{example}

\section{Bar Code associated to a finite set of terms}\label{BarCode}

In this section, we provide a language in order to represent zerodimensional monomial ideals, which are characterized by having a constant affine
Hilbert polynomial. 
\\
In the case of two or three variables, this will allow us to establish a 
connection  between (strongly) stable ideals  $I \triangleleft \mathcal{P}$ with 
constant affine Hilbert polynomial 
$H_I(t)=p \in \NN$ and some particular plane partitions of the integer number 
$p$.
More precisely, we will  give a combinatorial representation for the  associated (finite) lexicographical Groebner escalier  $\cN(I)$.\\
First of all, we point out that, since $\mathcal{T}\cong \NN^n$, a term
$x^\gamma = x_1^{\gamma_1}\cdots x_n^{\gamma_n}$ can be regarded as the point $(\gamma_1,...,\gamma_n)$ in the $n$-dimensional  space.
\\
Using this convention, we can represent $\cN(I)$ with a $n$-dimensional picture, 
called \emph{tower structure} of $I$ (for more details see \cite{Ce0} \cite[II.33]{SPES}). 
\begin{example}\label{TowerStr}
 Consider the radical ideal $I=(x_1^2-x_1,x_1x_2,x_2^2-2x_2)\triangleleft 
\ck[x_1,x_2]$, defined by its lexicographical reduced Groebner basis.
 Since w.r.t. Lex\footnote{Since, in this paper, we are working with the 
lexicographical 
order, I precised here ``w.r.t.'' Lex. Anyway, it can be easily observed that 
$\cT(x_1^2-x_1)=x_1^2$, $\cT(x_1x_2)=x_1x_2$, $\cT(x_2^2-2x_2)=x_2^2$ trivially 
holds for each term order.}, we have  $\cT(x_1^2-x_1)=x_1^2$, 
$\cT(x_1x_2)=x_1x_2$, $\cT(x_2^2-2x_2)=x_2^2$, we can conclude that the 
lexicographical Groebner escalier of $I$ is $\cN(I)=\{1,x_1,x_2\}$, so 
it can be represented by the following picture:
 
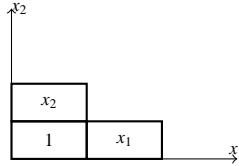
\begin{center}
 \begin{tikzpicture}
\draw [thick] (0,0) rectangle (1,.5);
\node at (.5,.25) {$\scriptstyle{1}$};
\draw [thick] (1,0) rectangle (2,.5);
\node at (1.5,.25) {$\scriptstyle{x_1}$};
\draw [thick] (0,.5) rectangle (1,1);
\node at (.5,.75) {$\scriptstyle{x_2}$};
 \node at (3,0.1) {$\scriptstyle{x_1}$};
\node at (0.1,2) {$\scriptstyle{x_2}$};
\draw[->] (0,0)--(0,2);
\draw[->] (0,0)--(3,0);
\end{tikzpicture}
\end{center}
 
\end{example}

\noindent For a radical ideal $I$, notice that if $\vert \cN(I) \vert < \infty$ also 
$\vert V(I)\vert < \infty$ (and, more precisely, it holds $\vert\cN(I)\vert= 
\vert V(I) \vert$), so the associated variety consists of a finite set of 
points.\\ 
It has been proved by \cemu (\cite{CeMu}) that, in this case,
any ordering on the points in $V(I)$ gives a precise one-to-one correspondence 
between the terms in $\cN(I) $ and the points in $V(I)$,
so it is also possible 
to label the points in the tower
structure with the corresponding point of the ordered $V(I)$.
\begin{example}\label{TowerStr2}
 Consider again the radical ideal $I=(x_1^2-x_1,x_1x_2,x_2^2-2x_2)\triangleleft 
\ck[x_1,x_2]$ of example \ref{TowerStr}.
 The corresponding variety can be easily computed and, actually, it is finite: 
$$V(I)=\{(0,0),(0,2),(1,0)\}.$$
 We can also note that, exactly as expected, $\vert\cN(I)\vert= \vert V(I) \vert =3.$ 
 The correspondence given by Cerlienco-Mureddu (see \cite{CeMu} for more details
 on how the correspondence is constructed) is displayed below; the corresponding reorderings of $V(I)$ 
 are indicated in square brackets: 
  \\
 \medskip

 \begin{minipage}{6cm}
\centering
 $\Phi_1: \cN(I) \rightarrow V(I) $\\
 $1\mapsto (0,0) $\\
 $x_2\mapsto (0,2) $\\
 $x_1\mapsto (1,0).$\\
\smallskip

 $[(0,0), (0,2), (1,0)];$\\
$[(0,0), (1,0), (0,2)]$.

 \end{minipage}
\hspace{0.2cm}
\begin{minipage}{7cm}
\centering
$\Phi_2: \cN(I) \rightarrow V(I) $\\
 $1\mapsto (1,0) $\\
 $x_2\mapsto (0,2) $\\
 $x_1\mapsto (0,0).$\\
\smallskip

$[(1,0), (0,0), (0,2)]$.
 \end{minipage}
\\
\smallskip

\begin{minipage}{6cm}
\centering
$\Phi_3: \cN(I) \rightarrow V(I) $\\
 $1\mapsto (1,0) $\\
 $x_2\mapsto (0,0) $\\
 $x_1\mapsto (0,2).$\\
\smallskip

$[(1,0), (0,2), (0,0)]$.
 \end{minipage}
\hspace{0.2cm}
\begin{minipage}{7cm}
\centering
$\Phi_4: \cN(I) \rightarrow V(I) $\\
 $1\mapsto (0,2) $\\
 $x_2\mapsto (0,0) $\\
 $x_1\mapsto (1,0).$\\
\smallskip 

 $[(0,2), (0,0), (1,0)];$\\
$[(0,2), (1,0), (0,0)]$.
 \end{minipage}
\\
 \smallskip
 
 Now, we can label the points in the tower structure with the corresponding 
point of $V(I)$, as it can be seen in the pictures below. \\
\smallskip

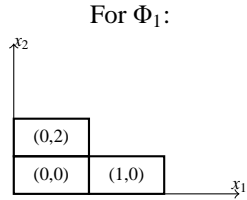
\begin{minipage}{6cm}
 \begin{center}
 For $\Phi_1$:\\
 \begin{tikzpicture}
\draw [thick] (0,0) rectangle (1,.5);
\node at (.5,.25) {$\scriptstyle{(0,0)}$};
\draw [thick] (1,0) rectangle (2,.5);
\node at (1.5,.25) {$\scriptstyle{(1,0)}$};
\draw [thick] (0,.5) rectangle (1,1);
\node at (.5,.75) {$\scriptstyle{(0,2)}$};
 \node at (3,0.1) {$\scriptstyle{x_1}$};
\node at (0.1,2) {$\scriptstyle{x_2}$};
\draw[->] (0,0)--(0,2);
\draw[->] (0,0)--(3,0);
\end{tikzpicture}
\end{center} 
\end{minipage}
\hspace{0.1 cm}
\begin{minipage}{6cm}
 \begin{center}
  For $\Phi_2$:\\
 
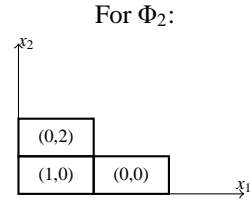
\begin{tikzpicture}
\draw [thick] (0,0) rectangle (1,.5);
\node at (.5,.25) {$\scriptstyle{(1,0)}$};
\draw [thick] (1,0) rectangle (2,.5);
\node at (1.5,.25) {$\scriptstyle{(0,0)}$};
\draw [thick] (0,.5) rectangle (1,1);
\node at (.5,.75) {$\scriptstyle{(0,2)}$};
 \node at (3,0.1) {$\scriptstyle{x_1}$};
\node at (0.1,2) {$\scriptstyle{x_2}$};
\draw[->] (0,0)--(0,2);
\draw[->] (0,0)--(3,0);
\end{tikzpicture}
\end{center} 
\end{minipage}
\\
\smallskip

\begin{minipage}{6cm}
 \begin{center}
  For $\Phi_3$:\\
 
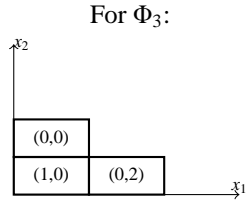
\begin{tikzpicture}
\draw [thick] (0,0) rectangle (1,.5);
\node at (.5,.25) {$\scriptstyle{(1,0)}$};
\draw [thick] (1,0) rectangle (2,.5);
\node at (1.5,.25) {$\scriptstyle{(0,2)}$};
\draw [thick] (0,.5) rectangle (1,1);
\node at (.5,.75) {$\scriptstyle{(0,0)}$};
 \node at (3,0.1) {$\scriptstyle{x_1}$};
\node at (0.1,2) {$\scriptstyle{x_2}$};
\draw[->] (0,0)--(0,2);
\draw[->] (0,0)--(3,0);
\end{tikzpicture}
\end{center} 
\end{minipage}
\hspace{0.1cm}
\begin{minipage}{6cm}
 \begin{center}
  For $\Phi_4$:\\
 
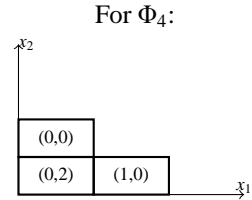
\begin{tikzpicture}
\draw [thick] (0,0) rectangle (1,.5);
\node at (.5,.25) {$\scriptstyle{(0,2)}$};
\draw [thick] (1,0) rectangle (2,.5);
\node at (1.5,.25) {$\scriptstyle{(1,0)}$};
\draw [thick] (0,.5) rectangle (1,1);
\node at (.5,.75) {$\scriptstyle{(0,0)}$};
 \node at (3,0.1) {$\scriptstyle{x_1}$};
\node at (0.1,2) {$\scriptstyle{x_2}$};
\draw[->] (0,0)--(0,2);
\draw[->] (0,0)--(3,0);
\end{tikzpicture}
\end{center} 
\end{minipage}
\end{example}

The construction of Examples \ref{TowerStr} and \ref{TowerStr2} is a sort of ``inverse'' of Macaulay's construction (see \cite{MaC} p.548) in which from a finite order ideal $\cN$, a finite set of point $\mathbf{X}$ and a Groebner basis of $I(\mathbf{X})$ are produced so that the lexicographical Groebner escalier $\cN(I(\mathbf{X}))$ is exactly $\cN$.

% % % % % Indeed, applying Cerlienco-Mureddu algorithm on the set $\mathbf{X}$ derived from
% % % % % Macaulay's construction, we get again $\cN$. On the other hands, given a 
% % % % % suitable set of points $\mathbf{X}$ and performing Cerlienco-Mureddu algorithm 
% % % % % on it, we get an order ideal $\cN$. Applying Macaulay's construction on $\cN$, 
% % % % % we can get again $\mathbf{X}$ (even if it is not true in general and this
% % % % %  is the reason for which we say a ``sort of inverse'').
% % % % % \textbf{Va spiegato meglio cosa intendo}

\begin{example}\label{EsXavier}
 For the case of two variables, the tower structure of a zerodimensional radical ideal $I$
 s.t. $V(I)=\{P_1,...,P_s\}$ is represented by $h$ towers, where $h$ is 
 the number of different values appearing as first coordinate of the points in 
$V(I)$, so that each tower corresponds to a ``first coordinate''. For each $1 
\leq i \leq h$, the $i$-th tower contains as many elements as the number of 
occurrences of the associated first coordinate. Displaying these towers in 
nonincreasing order by height, one obtains a tower structure for $I$ (see the 
one obtained in example \ref{TowerStr2} via the map $\Phi_1$).\\
 
This is not the case for three or more variables, since some shifts in the  towers' planes  are needed.
For example, given the zerodimensional radical ideal $I=(x_1^2-x_1,x_1x_2,x_2^2-x_2,x_1x_3-x_3,x_2x_3,x_3^2-x_3)
\triangleleft \ck[x_1,x_2,x_3]$, whose variety is $$V(I)=\{(0,0,0),(0,1,0),(1,0,0),(1,0,1)\},$$
 we have $\cN(I)=\{1,x_1,x_2,x_3\}$, which cannot be represented with a natural extension to three variables  of the procedure explained above.  In such an extension, the towers are in the $x(2)$ direction if the points have only the same first coordinate and in the $x(3)$ direction if both the first and the second coordinate are the same.
\end{example}

\begin{example}\label{TowerStr3D}
Let us consider the zerodimensional radical ideal $I=(x_1^3-3x_1^2+2x_1, x_1x_2,x_2^2-2x_2)\triangleleft \ck[x_1,x_2]$, defined by its lexicographical reduced Groebner basis.
Since, w.r.t. Lex,
$\cT(x_1^3-3x_1^2+2x_1)=x_1^3$, 
$\cT(x_1x_2)=x_1x_2$, $\cT(x_2^2-2x_2)=x_2^2$, we can 
conclude that the 
lexicographical Groebner escalier of $I$ is 
$\cN(I)=\{1,x_1,x_1^2,x_2\}$, so 
it can be represented with the following picture:
 
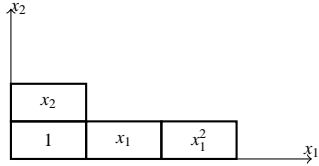
\begin{center}
 \begin{tikzpicture}
\draw [thick] (0,0) rectangle (1,.5);
\node at (.5,.25) {$\scriptstyle{1}$};
\draw [thick] (1,0) rectangle (2,.5);
\node at (1.5,.25) {$\scriptstyle{x_1}$};
\draw [thick] (2,0) rectangle (3,.5);
\node at (2.5,.25) {$\scriptstyle{x_1^2}$};
\draw [thick] (0,.5) rectangle (1,1);
\node at (.5,.75) {$\scriptstyle{x_2}$};
 \node at (4,0.1) {$\scriptstyle{x_1}$};
\node at (0.1,2) {$\scriptstyle{x_2}$};
\draw[->] (0,0)--(0,2);
\draw[->] (0,0)--(4,0);
\end{tikzpicture}
\end{center} 
\medskip

Consider now the zerodimensional radical ideal
$I'=(x_1^3-x_1,x_1x_2,x_2^2-2x_2,x_3+x_1^2-x_1) \triangleleft \ck[x_1,x_2,x_3]$, defined via its reduced lexicographical Groebner basis.
Since w.r.t. Lex, we have 
$\cT(x_1^3-x_1)=x_1^3$,
$\cT(x_1x_2)=x_1x_2$, $\cT(x_2^2-2x_2)=x_2^2$, $\cT(x_3+x_1^2-x_1)=x_3$, we can
conclude that the
lexicographical Groebner escalier of $I'$ is
$\cN(I')=\{1,x_1,x_1^2, x_2\}$, so
it can be represented with the  following picture:
% Consider now the zerodimensional radical ideal
% $I'=(x_1^3-3x_1^2+2x_1,x_1x_2-x_1^2,x_2^2-x_2-x_1^2+x_1,x_3+x_1^2-x_1) \triangleleft \ck[x_1,x_2,x_3]$, defined via its reduced lexicographical Groebner basis.
% Since w.r.t. Lex, we have 
% $\cT(x_1^3-3x_1^2+2x_1)=x_1^3$,
% $\cT(x_1x_2-x_1^2)=x_1x_2$, $\cT(x_2^2-x_2-x_1^2+x_1)=x_2^2$, $\cT(x_3+x_1^2-x_1)=x_3$, we can
% conclude that the
% lexicographical Groebner escalier of $I'$ is
% $\cN(I')=\{1,x_1,x_1^2, x_2\}$, so
% it can be represented with the following picture:
 
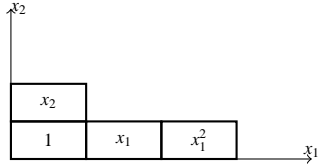
\begin{center}
 \begin{tikzpicture}
\draw [thick] (0,0) rectangle (1,.5);
\node at (.5,.25) {$\scriptstyle{1}$};
\draw [thick] (1,0) rectangle (2,.5);
\node at (1.5,.25) {$\scriptstyle{x_1}$};
\draw [thick] (2,0) rectangle (3,.5);
\node at (2.5,.25) {$\scriptstyle{x_1^2}$};
\draw [thick] (0,.5) rectangle (1,1);
\node at (.5,.75) {$\scriptstyle{x_2}$};
 \node at (4,0.1) {$\scriptstyle{x_1}$};
\node at (0.1,2) {$\scriptstyle{x_2}$};
\draw[->] (0,0)--(0,2);
\draw[->] (0,0)--(4,0);
\end{tikzpicture}
\end{center}
We point out that the tower structure above is exactly the same as for $I$, even if  $I'\triangleleft
\mathcal{P}=\ck[x_1,x_2,x_3]$
and  $I\triangleleft\ck[x_1,x_2]$.

The reason of this fact is that
$x_3 \notin \cN(I')$; indeed, $x_3$
is the leading term of $x_3+x_1^2-x_1$. In general,
the reason is that there is a polynomial
$(x_3-\sum_{t\in\cN(I')}c_t t)\in I'$.

In a slightly different situation
(i.e. in solving equations) the ability of detecting linear relations $\mod I'$ among the elements of $\{1,x_1,x_2,x_3\}$ and, equivalently,
producing a basis of the vector space generated by $\{1,x_1,x_2,x_3\}$,
${\bf Span}(1,x_1,x_2,x_3) \, \mod I'$, is crucial (see \cite{AuzStet,Lundqwist}).

This is the case, for instance of
% $I''=(x_1^3-1, x_1x_2, x_2^2-2x_2,x_3-x_1) \triangleleft
% \ck[x_1,x_2,x_3]$, where
$I''=(x_1^3-x_1 , x_1x_2, x_2^2-2x_2,x_3-x_1) \triangleleft
 \ck[x_1,x_2,x_3]$, where
${\bf Span}(1,x_1,x_2,x_3) = {\bf Span}(1,x_1,x_2)\mod I''$

\end{example}

Unfortunately, as one can easily understand, the
tower structure becomes rather complicated when
we have an high number of terms in $\cN(I)$ 
and/or  of linearly \emph{independent} variables in $\mathcal{P}$, i.e. when we
deal with a large number of points, and/or we have really to draw the structure
for high-dimensional spaces\footnote{Actually, in this context, ``high-dimensional'' means ``of dimension greater than or equal to'' $4$.}.
\\
Moreover, as shown in example \ref{TowerStr3D}, from the tower structure
 it is impossible to understand the ring in which the Groebner escalier 
 has been computed, since linearly dependent variables are discarded (see \cite{Lundqwist}).
\\
For these reasons, we introduce now the \emph{Bar Code diagram},
namely a (rather compact) \emph{bidimensional picture} which keeps track of
all the information contained in the tower structure, making them simple to be extracted.
\\
We define now, in general, what is a Bar Code. After that, we see how to 
associate to a finite set of terms a Bar Code and, 
vice versa, how to associate a finite set of terms to a given Bar Code.
\begin{definition}\label{BCdef1}
A Bar Code $\cB$ is a picture composed by segments, called \emph{bars}, superimposed in horizontal rows, which satisfies conditions $a.,b.$ below.
Denote by 
\begin{itemize}
 \item $\cB_j^{(i)}$ the $j$-th bar (from left to right) of the $i$-th row (from top to bottom), i.e. the \emph{$j$-th $i$-bar};
 \item $\mu(i)$ the number of bars of the $i$-th row
 \item $l_1(\cB_j^{(1)}):=1$, $\forall j \in \{1,2,...,\mu(1)\}$ the $(1-)$\emph{length} of the $1$-bars;
 \item $l_i(\cB_j^{(k)})$, $2\leq k \leq n$, $1 \leq i \leq k-1$, $1\leq j \leq \mu(k)$ the $i$-\emph{length} of $\cB_j^{(k)}$, i.e. the number of $i$-bars lying over $\cB_j^{(k)}$
\end{itemize}
\begin{itemize}
 \item[a.] $\forall i,j$, $1 \leq i \leq n-1$, $1\leq j \leq \mu(i)$, $\exists ! \overline{j}\in \{1,...,\mu(i+1)\}$ s.t. $\cB_{\overline{j}}^{(i+1)}$ lies  under  $\cB_j^{(i)}$ 
 \item[b.] $\forall i_1,\,i_2 \in \{1,...,n\}$, $\sum_{j_1=1}^{\mu(i_1)} l_1(\cB_{j_1}^{(i_1)})= \sum_{j_2=1}^{\mu(i_2)} l_1(\cB_{j_2}^{(i_2)})$; we will then say that  \emph{all the rows have the same length}.
\end{itemize}
\end{definition}
We denote by $\mathcal{B}_n$ the set of all Bar Codes composed by $n$ rows.

Note that if  $1 \leq i_1 <i_2 \leq n$, $1 \leq j_1 \leq \mu(i_1)$, 
$1 \leq j_2 \leq \mu(i_2)$ and $\cB_{j_2}^{(i_2)}$ lies below 
$\cB_{j_1}^{(i_1)}$, then $l_1(\cB_{j_2}^{(i_2)}) \geq l_1(\cB_{j_1}^{(i_1)})$.

\begin{Definition}\label{BarList}
 We call \emph{bar list} of a Bar Code $\sf B$, composed by $n$ rows, 
%(also called $n$-Bar Code), 
the list $${\sf 
L_B}:=(\mu(1),...,\mu(n)).$$
\end{Definition}

\begin{example}\label{BC}
 An example of Bar Code $\cB$ is 
 
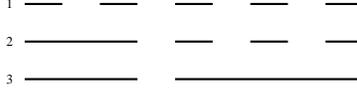
\begin{center}
\begin{tikzpicture}
\node at (3.8,-0.5) [] {${\scriptscriptstyle 1}$};
\node at (3.8,-1) [] {${\scriptscriptstyle 2}$};
\node at (3.8,-1.5) [] {${\scriptscriptstyle 3}$};

\draw [thick] (4,-0.5) --(4.5,-0.5);
\draw [thick] (5,-0.5) --(5.5,-0.5);
\draw [thick] (6,-0.5) --(6.5,-0.5);
\draw [thick] (7,-0.5) --(7.5,-0.5);
\draw [thick] (8,-0.5) --(8.5,-0.5);
\draw [thick] (4,-1)--(5.5,-1);
\draw [thick] (6,-1) --(6.5,-1);
\draw [thick] (7,-1) --(7.5,-1);
\draw [thick] (8,-1) --(8.5,-1);
\draw [thick] (4,-1.5)--(5.5,-1.5);
\draw [thick] (6,-1.5) --(8.5,-1.5);
\end{tikzpicture}
\end{center}
The $1$-bars have length $1$.
% 
% IO
% As regards the other rows, $l_1(\cB_1^{(2)})=2$, $l_1(\cB_2^{(2)})=l_1(\cB_3^{(2)})= l_1(\cB_4^{(2)})=1$,  $l_1(\cB_1^{(3)})=2$ and 
%  $l_1(\cB_2^{(3)})=3$, so 
%  $$\sum_{j_1=1}^{\mu(1)} l_1(\cB_{j_1}^{(1)})= \sum_{j_2=1}^{\mu(2)} l_1(\cB_{j_2}^{(2)})= \sum_{j_3=1}^{\mu(3)} l_1(\cB_{j_3}^{(3)})=5.$$
% TEO
%  
%  
  As regards the other rows, $l_1(\cB_1^{(2)})=2$,
$l_1(\cB_2^{(2)})=l_1(\cB_3^{(2)})= l_1(\cB_4^{(2)})=1$,
$l_2(\cB_1^{(3)})=1$,$l_1(\cB_1^{(3)})=2$ and
 $l_2(\cB_2^{(3)})=l_1(\cB_2^{(3)})=3$, so
 $$\sum_{j_1=1}^{\mu(1)} l_1(\cB_{j_1}^{(1)})= \sum_{j_2=1}^{\mu(2)}
l_1(\cB_{j_2}^{(2)})= \sum_{j_3=1}^{\mu(3)} l_1(\cB_{j_3}^{(3)})=5.$$

 The bar list is  ${\sf L_B}:=(5, 4, 2)$.

\end{example}

% % % % 
% % % % \begin{definition}\label{SubBC}
% % % % A Bar Code $\cB'$, obtained by extracting 
% % % % some (even non-consecutive) rows from a Bar Code
% % % % $\sf B$, but including row $1$, is called \emph{sub-Bar Code} of $\sf B$.
% % % % \end{definition}
\begin{definition}\label{Block}
 Given a Bar Code $\cB$, for each $1 \leq l \leq n$, $l\leq i\leq n$, $1\leq j \leq \mu(i)$, 
an $l$-\emph{block} associated to a bar $B_j^{(i)}$ of 
$\sf B$ is the set containing  $B_j^{(i)}$ itself and all the bars
of the 
$(l-1)$ 
rows lying immediately  above $B_j^{(i)}$.
\end{definition}

\begin{example}\label{SubBCBlock}
 Take again the Bar Code $\cB$ of example \ref{BC}
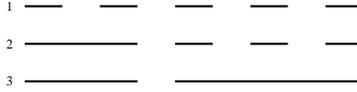
 \begin{center}
\begin{tikzpicture}
\node at (3.8,-0.5) [] {${\scriptscriptstyle 1}$};
\node at (3.8,-1) [] {${\scriptscriptstyle 2}$};
\node at (3.8,-1.5) [] {${\scriptscriptstyle 3}$};

\draw [thick] (4,-0.5) --(4.5,-0.5);
\draw [thick] (5,-0.5) --(5.5,-0.5);
\draw [thick] (6,-0.5) --(6.5,-0.5);
\draw [thick] (7,-0.5) --(7.5,-0.5);
\draw [thick] (8,-0.5) --(8.5,-0.5);
\draw [thick] (4,-1)--(5.5,-1);
\draw [thick] (6,-1) --(6.5,-1);
\draw [thick] (7,-1) --(7.5,-1);
\draw [thick] (8,-1) --(8.5,-1);
\draw [thick] (4,-1.5)--(5.5,-1.5);
\draw [thick] (6,-1.5) --(8.5,-1.5);
\end{tikzpicture}
\end{center}
% % 
% % If we extract the rows 1 and 3 we get the sub-Bar Code $\cB'$:
% % 
% %  \begin{center}
% % \begin{tikzpicture}
% % \node at (3.8,-0.5) [] {${\scriptscriptstyle 1}$};
% %  \node at (3.8,-1) [] {${\scriptscriptstyle 3}$};
% % %\node at (3.8,-1.5) [] {${\scriptscriptstyle 3}$};
% % 
% % \draw [thick] (4,-0.5) --(4.5,-0.5);
% % \draw [thick] (5,-0.5) --(5.5,-0.5);
% % \draw [thick] (6,-0.5) --(6.5,-0.5);
% % \draw [thick] (7,-0.5) --(7.5,-0.5);
% % \draw [thick] (8,-0.5) --(8.5,-0.5);
% % % \draw [thick] (4,-1)--(5.5,-1);
% % % \draw [thick] (6,-1) --(6.5,-1);
% % % \draw [thick] (7,-1) --(7.5,-1);
% % % \draw [thick] (8,-1) --(8.5,-1);
% % \draw [thick] (4,-1)--(5.5,-1);
% % \draw [thick] (6,-1) --(8.5,-1);
% % \end{tikzpicture}
% % \end{center}
Consider the bar $B^{(3)}_2$ (so $i=n=3$, $j=2=\mu(3)$) and set $l=2$.
The 
$2$-block associated
to $B^{(3)}_2$  consists of $B^{(3)}_2$ itself and of the bars  
$B^{(2)}_2,B^{(2)}_3, B^{(2)}_4$, as shown by the
thick blue lines in the picture 
below:
 
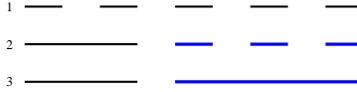
\begin{center}
\begin{tikzpicture}
\node at (3.8,-0.5) [] {${\scriptscriptstyle 1}$};
\node at (3.8,-1) [] {${\scriptscriptstyle 2}$};
\node at (3.8,-1.5) [] {${\scriptscriptstyle 3}$};

\draw [thick] (4,-0.5) --(4.5,-0.5);
\draw [thick] (5,-0.5) --(5.5,-0.5);
\draw [thick] (6,-0.5) --(6.5,-0.5);
\draw [thick] (7,-0.5) --(7.5,-0.5);
\draw [thick] (8,-0.5) --(8.5,-0.5);
\draw [thick] (4,-1)--(5.5,-1);
\draw [line width=1.2pt, color=blue] (6,-1) --(6.5,-1);
\draw [line width=1.2pt, color=blue] (7,-1) --(7.5,-1);
\draw [line width=1.2pt, color=blue] (8,-1) --(8.5,-1);
\draw [thick] (4,-1.5)--(5.5,-1.5);
\draw [line width=1.2pt, color=blue] (6,-1.5) --(8.5,-1.5);
\end{tikzpicture}
\end{center}

\end{example}

We outline now the construction of the Bar Code associated to a finite set 
of terms. In order to do it, we need to introduce the operators $P_{x_i},\, 
i=1,...,n$ on 
the terms.
\smallskip

First of all, we associate to each term  $\tau=x_1^{\gamma_1}\cdots 
x_n^{\gamma_n} \in \mathcal{T}\subset \ck[x_1,...,x_n]$, 
%a list of 
$n$
terms (one for each variable in $\mathcal{P}$). More precisely, for each $i \in \{1,...,n\}$, we let 
$$P_{x_i}(\tau):=x_i^{\gamma_i}\cdots x_n^{\gamma_n} \in \mathcal{T}, \textrm{ i.e. } P_{x_i}(\tau)=\frac{\tau}{x_1^{\gamma_1}\cdots 
x_{i-1}^{\gamma_{i-1}}}.$$
We can extend this procedure to a finite set of terms $M\subset \mathcal{T}$, 
defining, for each $ i \in \{1,...,n\}$, 
$$M^{[i]}:=P_{x_i}(M):=\{\sigma\in \mathcal{T},\,\vert \, \exists \tau \in M, P_{x_i}(\tau)=\sigma\}.$$
The terms in $M^{[i]}$ will play a fundamental role for the construction of the  Bar Code diagram.
\\
\smallskip

Here we list some features of  the operators $P_{x_i}$, that will be useful in 
 what follows.
\begin{enumerate}
\item For each $\tau\in \mathcal{T}$, $ P_{x_1}(\tau)=\tau.$
\item If $\tau=x_1^{\gamma_1}\cdots
x_n^{\gamma_n}$, $\gamma_i=deg_{i}(\tau)=0$
 then $P_{x_i}(\tau)=x_{i+1}^{\gamma_{i+1}}\cdots
x_n^{\gamma_n}=P_{x_{i+1}}(\tau)$.
\item It holds
$$ \tau<_{Lex}\sigma \Rightarrow P_{x_i}(\tau)\leq_{Lex} P_{x_i}(\sigma),\,\forall i \in \{1,...,n\}.$$
\item  For each term $\tau$ and for any pair of
indices $i,j$, say $1 \leq i <j \leq n$, we have that, since $x_i< x_j$, 
 $$ P_{x_j}(P_{x_i}(\tau))= P_{x_i}(P_{x_j}(\tau)) =P_{x_j}(\tau). $$
 \item For each $\sigma, \tau \in \kT$, $\forall 1 \leq i <n$, it  holds
 $$P_{x_i}(\tau)=P_{x_i}(\sigma) \Rightarrow  P_{x_{i+1}}(\tau)=P_{x_{i+1}}(\sigma).$$
\end{enumerate}

\begin{example}\label{PxiPropsEx}
Consider the term $\tau=x_1x_2^3x_3^4 \in \ck[x_1,x_2,x_3]$.\\ Clearly $P_{x_1}(\tau)=x_1x_2^3x_3^4$, while $P_{x_2}(\tau)=x_2^3x_3^4$ and $P_{x_3}(\tau)=x_3^4$. For $\sigma_1:=x_2x_3^5 >_{Lex} \tau$, $P_{x_2}(\tau)=x_2^3x_3^4<_{Lex}P_{x_2}(\sigma_1)=x_2x_3^5$ and $P_{x_3}(\tau)=x_3^4 <_{Lex}P_{x_3}(\sigma_1)=x_3^5$; for $\sigma_2:=x_1^5x_2^3x_3^4>_{Lex}\tau$, $P_{x_2}(\tau)=x_2^3x_3^4=P_{x_2}(\sigma_2)$ and $P_{x_3}(\tau)=P_{x_3}(\sigma_2)=x_3^4$. Moreover, $P_{x_3}(P_{x_2}(\tau))=P_{x_3}(x_2^3x_3^4)=x_3^4=P_{x_2}(P_{x_3}(\tau))$.
\end{example}
Now we take $M\subseteq \mathcal{T}$, with $\vert M\vert =m < \infty$ and we order its 
elements increasingly  w.r.t. Lex, getting the list  
$\bM=[\tau_1,...,\tau_m]$. Then, we construct the sets $M^{[i]}$, and 
the corresponding lexicographically ordered lists $\bM^{[i]}$, for $i=1,...,n$.
We notice that $\bM$ cannot contain repeated terms, while the $\bM^{[i]}$, 
for $1<i \leq n$, can. In case some repeated terms occur in $\bM^{[i]}$, $1<i 
\leq n$, they clearly have to be adjacent in the list, due to the 
lexicographical ordering.
\\
We can now define the $n\times m $ matrix of terms $\kM$ as the matrix s.t. 
its $i$-th row is $\bM^{[i]}$, $i=1,...,n$, i.e.
\[\kM:= \left(\begin{array}{cccc}
%\tau_1&... & \tau_n\\
P_{x_1}(\tau_1)&... & P_{x_1}(\tau_m)\\
P_{x_2}(\tau_1)&... & P_{x_2}(\tau_m)\\
\vdots & \quad &\vdots\\
P_{x_n}(\tau_1)& ... & P_{x_n}(\tau_m)
\end{array}\right)\]
\begin{definition}\label{BarCodeDiag}
 The \emph{Bar Code diagram} $\cB$ associated to $M$ (or, equivalently, to 
$\bM$) is a 
$n\times m $ diagram, made by segments s.t. the $i$-th row of $\cB$, $1\leq 
i\leq n$  is constructed as follows:
       \begin{enumerate}
        \item take the $i$-th row of $\kM$, i.e. $\bM^{[i]}$
        \item consider all the sublists of repeated terms, i.e. $[P_{x_{i}}(\tau_{j_1}),P_{x_{i}}(\tau_{j_1 +1}),...,P_{x_{i}}(\tau_{j_1 
+h})]$ s.t. 
        $P_{x_{i}}(\tau_{j_1})= P_{x_{i}}(\tau_{j_1 
+1})=...=P_{x_{i}}(\tau_{j_1 +h})$, noticing that\footnote{Clearly if a term 
$P_{x_{i}}(\tau_{\overline{j}})$ is not 
        repeated in $\bM^{[i]}$, the sublist containing it will be only  
$[P_{x_{i}}(\tau_{\overline{j}})]$, i.e. $h=0$.} $0 \leq h<m$ 
        \item underline each sublist with a segment
        \item delete the terms of $\bM^{[i]}$, leaving only the segments (i.e. 
%the \emph{$x_i$-bars} or, simply, 
the \emph{$i$-bars}).
       \end{enumerate}
 We usually label each $1$-bar $\cB_j^{(1)}$, $j \in \{1,...,\mu(1)\}$ with the 
term $\tau_j \in \bM$.
\end{definition}

By property 5. of the operators $P_{x_i}$ and, since for each $1 \leq i \leq n$,
$\vert \bM^{[i]}\vert =\sum_{j=1}^{\mu(i)} l_1(\cB^{(i)}_j)$, a Bar Code diagram is a 
Bar Code in the sense of Definition \ref{BCdef1}.

\begin{example}\label{BarCodeNoOrdId}
Given  $M=\{x_1,x_1^2,x_2x_3,x_1x_2^2x_3,x_2^3x_3\}\subset
\mathbf{k}[x_1,x_2,x_3]$, we have:\\
%$\bM^{[0]}=
$\bM^{[1]}=[x_1,x_1^2,x_2x_3,x_1x_2^2x_3, x_2^3x_3]$\\
$\bM^{[2]}=[1,1,x_2x_3,x_2^2x_3,x_2^3x_3]$\\
$\bM^{[3]}=[1, 1,x_3,x_3,x_3]$,\\ leading to the $3 \times 5 $ table on the 
left and then to the 
Bar Code on the right:
\\
\begin{minipage}[b]{0.5\linewidth}
\begin{center}
\begin{tikzpicture}
% \node at (4.2,0) [] {${\small x_1}$};
% \node at (5.2,0) [] {${\small x_1^2}$};
% \node at (6.2,0) [] {${\small x_2x_3}$};
% \node at (7.2,0) [] {${\small x_1x_2^2x_3}$};
% \node at (8.2,0) [] {${\small x_2^3x_3}$};

\node at (4.2,-0.5) [] {${\small x_1}$};
\node at (5.2,-0.5) [] {${\small x_1^2}$};
\node at (6.2,-0.5) [] {${\small x_2x_3}$};
\node at (7.2,-0.5) [] {${\small x_1x_2^2x_3}$};
\node at (8.2,-0.5) [] {${\small x_2^3x_3}$};

\node at (4.2,-1) [] {${\small 1}$};
\node at (5.2,-1) [] {${\small 1}$};
\node at (6.2,-1) [] {${\small x_2x_3}$};
\node at (7.2,-1) [] {${\small x_2^2x_3}$};
\node at (8.2,-1) [] {${\small x_2^3x_3}$};

\node at (4.2,-1.5) [] {${\small 1}$};
\node at (5.2,-1.5) [] {${\small1}$};
\node at (6.2,-1.5) [] {${\small x_3}$};
\node at (7.2,-1.5) [] {${\small x_3}$};
\node at (8.2,-1.5) [] {${\small x_3}$};
\end{tikzpicture}
\end{center}
\end{minipage}
\hspace{0.45cm}
\begin{minipage}[b]{0.5\linewidth}
\begin{center}
\begin{tikzpicture}
\node at (4.2,0) [] {${\small x_1}$};
\node at (5.2,0) [] {${\small x_1^2}$};
\node at (6.2,0) [] {${\small x_2x_3}$};
\node at (7.2,0) [] {${\small x_1x_2^2x_3}$};
\node at (8.2,0) [] {${\small x_2^3x_3}$};

% \node at (3.8,0) [] {${\scriptscriptstyle 1}$};
\node at (3.8,-0.5) [] {${\scriptscriptstyle 1}$};
\node at (3.8,-1) [] {${\scriptscriptstyle 2}$};
\node at (3.8,-1.5) [] {${\scriptscriptstyle 3}$};

\draw [thick] (4,-0.5) --(4.5,-0.5);
\draw [thick] (5,-0.5) --(5.5,-0.5);
\draw [thick] (6,-0.5) --(6.5,-0.5);
\draw [thick] (7,-0.5) --(7.5,-0.5);
\draw [thick] (8,-0.5) --(8.5,-0.5);
\draw [thick] (4,-1)--(5.5,-1);
\draw [thick] (6,-1) --(6.5,-1);
\draw [thick] (7,-1) --(7.5,-1);
\draw [thick] (8,-1) --(8.5,-1);
\draw [thick] (4,-1.5)--(5.5,-1.5);
\draw [thick] (6,-1.5) --(8.5,-1.5);
\end{tikzpicture}
\end{center}
\end{minipage}
\end{example} 
\begin{Remark}\label{NoinJ}
We can easily observe that Bar Codes associated to different sets 
of terms, \emph{need not} to be different.\\
 For example, if $M:=\{1,x_1\}, M':=\{x_1,x_1^2\}\subset \ck[x_1,x_2]$, both 
the Bar Code $\cB$ associated to $M$ 
and the Bar Code $\cB'$ associated to $M'$ 
are
 
 \begin{minipage}[b]{0.5\linewidth}
  \begin{center}
\begin{tikzpicture}
\node at (4.2,0) [] {${\small 1}$};
\node at (5.2,0) [] {${\small x_1}$};
 
% \node at (3.8,0) [] {${\scriptscriptstyle 1}$};
\node at (3.8,-0.5) [] {${\scriptscriptstyle 1}$};
\node at (3.8,-1) [] {${\scriptscriptstyle 2}$};

\draw [thick] (4,-0.5) --(4.5,-0.5);
\draw [thick] (5,-0.5) --(5.5,-0.5);
 
\draw [thick] (4,-1)--(5.5,-1);

\end{tikzpicture}
\end{center}
 \end{minipage}
 \hspace{0.5cm}
 \begin{minipage}[b]{0.5\linewidth}
  \begin{center}
\begin{tikzpicture}
\node at (4.2,0) [] {${\small x_1}$};
\node at (5.2,0) [] {${\small x_1^2}$};
 
% \node at (3.8,0) [] {${\scriptscriptstyle 1}$};
\node at (3.8,-0.5) [] {${\scriptscriptstyle 1}$};
\node at (3.8,-1) [] {${\scriptscriptstyle 2}$};

\draw [thick] (4,-0.5) --(4.5,-0.5);
\draw [thick] (5,-0.5) --(5.5,-0.5);
 
\draw [thick] (4,-1)--(5.5,-1);

\end{tikzpicture}
\end{center}
 \end{minipage}
 We will see soon that this cannot happen for order ideals.
% \textbf{ FRASE PROBABILMENTE DA TOGLIERE QUANDO DECIDERO' DI DEFINIRE LA MAPPA
% If restricted to order ideals the equality holds, 
% as we will also see soon.}
\end{Remark}

Now we explain how to associate a finite set of terms $M_\cB$ to a given Bar 
Code $\cB$.
In order to do it, we have to follow the steps below:
\begin{itemize}
 \item[BC1] consider the $n$-th row, composed by the bars 
$B^{(n)}_1,..., B^{(n)}_{\mu(n)}$. Let $l_1(B^{(n)}_j)=\ell^{(n)}_j$,
for 
$j\in\{1,...,\mu(n)\}$ and $a_1,...,a_{\mu(n)} \in \NN$, 
s.t. $a_k < a_h$ if 
$k<h$. Label each bar $B^{(n)}_j$ with $\ell^{(n)}_j$ 
copies of $x_n^{a_j}$.
 \item[BC2] For each $i=1,...,n-1$, $1 \leq j \leq \mu(n-i+1)$ 
 consider the bar $B^{(n-i+1)}_j$ and suppose that it has been 
 labelled by 
$\ell^{(n-i+1)}_j$ copies of a term $\tau$. Construct the 
 $2$-block associated to  $B^{(n-i+1)}_j$ which, by definition, 
 is composed by 
$B^{(n-i+1)}_j$ and by all the $(n-i)$-bars 
$B^{(n-i)}_{\overline{j}},...,B^{(n-i)}_{\overline{j}+h}$, 
 lying immediately  above  $B^{(n-i+1)}_j$; note that $h$ satisfies $0 \leq 
h \leq \mu(n-i)-\overline{j}$.  
\\ 
 Denote the  1-lenghts of $B^{(n-i)}_{\overline{j}}$... 
$B^{(n-i)}_{\overline{j}+h}$ by  
$l_1(B^{(n-i)}_{\overline{j}})=\ell^{(n-i)}_{\overline{j}}$,...,
 $l_1(B^{(n-i)}_{\overline{j}+h})=\ell^{(n-i)}_{\overline{j}+h}$ 
 and fix $h+1$ 
 natural numbers  $a_{\overline{j}}<a_{\overline{j}+1}<...<a_{\overline{j}+h}$. 
 For each $0\leq k\leq h$, label  $ B^{(n-i)}_{\overline{j}+k}$ with 
$\ell^{(n-i)}_{\overline{j}+k}$ copies
of $\tau x_{n-i}^{a_{\overline{j}+k}}$. 
 \end{itemize}
Clearly, if, given a Bar Code $\cB$, we apply BC1 and BC2 to get a set
$M\subset \kT$, and then we construct the Bar Code associated to $M$, we get 
back $\cB$. Indeed, BC1 and BC2 exactly construct the elements of the ordered 
lists $\bM^{[i]}$, $i=1,...,n$.
\\
\smallskip

Given a Bar Code $\cB$, applying 
steps BC1 and BC2, we can generate an \emph{infinite} number of sets 
$M\subset \kT$. 
\\
We modify the steps BC1 and BC2 getting BbC1 and BbC2 so that, for each  Bar 
Code $\cB$, the set of terms generated by applying them turns out to be 
\emph{unique}:
\begin{itemize}
 \item[BbC1] consider the $n$-th row, composed by the bars 
$B^{(n)}_1,...,B^{(n)}_{\mu(n)}$. Let $l_1(B^{(n)}_j)=\ell^{(n)}_j$, 
for 
$j\in\{1,...,\mu(n)\}$. Label each bar 
$B^{(n)}_j$ with $\ell^{(n)}_j$ copies 
of $x_n^{j-1}$.
 \item[BbC2] For each $i=1,...,n-1$, $1 \leq j \leq \mu(n-i+1)$ 
 consider the bar $B^{(n-i+1)}_j$ and suppose that it has been 
 labelled by 
$\ell^{(n-i+1)}_j$ copies of a term $\tau$. Construct the 
 $2$-block associated to  $B^{(n-i+1)}_j$ which, by definition, 
 is composed by $B^{(n-i+1)}_j$ and by all the $(n-i)$-bars 
$B^{(n-i)}_{\overline{j}},...,B^{(n-i)}_{\overline{j}+h}$ 
  lying immediately  above  $ B^{(n-i+1)}_j$; note that $h$ satisfies 
$0\leq h\leq \mu(n-i)-\overline{j}$. 
 Denote the 1-lenghts of 
$B^{(n-i)}_{\overline{j}},...,B^{(n-i)}_{\overline{j}+h}$  by  
$l_1(B^{(n-i)}_{\overline{j}})=\ell^{(n-i)}_{\overline{j}}$,...,
 $l_1(B^{(n-i)}_{\overline{j}+h})=\ell^{(n-i)}_{\overline{j}+h}$. 
 For each $0\leq k\leq h$, label  $ B^{(n-i)}_{\overline{j}+k}$ with 
$\ell^{(n-i)}_{\overline{j}+k}$ copies of $\tau x_{n-i}^{k}$. 
 \end{itemize} 
It is important to notice that not all Bar Codes can be associated to order ideals, as easily shown by the example below.
\begin{example}\label{star1}
Consider the Bar Code $\cB$
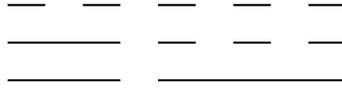
\begin{center}
\begin{tikzpicture}
\draw [thick] (4,0) --(4.5,0);
\draw [thick] (5,0) --(5.5,0);
\draw [thick] (6,0) --(6.5,0);
\draw [thick] (7,0) --(7.5,0);
\draw [thick] (8,0) --(8.5,0);
\draw [thick] (4,-0.5)--(5.5,-0.5);
\draw [thick] (6,-0.5) --(6.5,-0.5);
\draw [thick] (7,-0.5) --(7.5,-0.5);
\draw [thick] (8,-0.5) --(8.5,-0.5);
\draw [thick] (4,-1)--(5.5,-1);
\draw [thick] (6,-1) --(8.5,-1);
\end{tikzpicture}
\end{center}
We cannot associate any order ideal to it.\\
Indeed, using either BC1, BC2 or BbC1,BbC2, we obtain terms of the form
 $$\begin{array}{ccccc}
x_1^{\alpha_1}x_2^{\beta_1}x_3^{\gamma_1}  & x_1^{\alpha_2}x_2^{\beta_1}x_3^{\gamma_1} &x_1^{\alpha_3}x_2^{\delta_1}x_3^{\gamma_2}&x_1^{\alpha_4}x_2^{\delta_2}x_3^{\gamma_2}&x_1^{\alpha_5}x_2^{\delta_3}x_3^{\gamma_2} \\
x_2^{\beta_1}x_3^{\gamma_1} & x_2^{\beta_1}x_3^{\gamma_1}& x_2^{\delta_1}x_3^{\gamma_2}&x_2^{\delta_2}x_3^{\gamma_2}&x_2^{\delta_3}x_3^{\gamma_2}\\
x_3^{\gamma_1} & x_3^{\gamma_1}&x_3^{\gamma_2}&x_3^{\gamma_2}&x_3^{\gamma_2}\\
\end{array},$$
with $\gamma_1<\gamma_2$, $\delta_1<\delta_2<\delta_3$, $\alpha_1<\alpha_2$ and so the associated set of terms $M$ turns out to be 
$$M=\{x_1^{\alpha_1}x_2^{\beta_1}x_3^{\gamma_1}, x_1^{\alpha_2}x_2^{\beta_1}x_3^{\gamma_1},x_1^{\alpha_3}x_2^{\delta_1}x_3^{\gamma_2},x_1^{\alpha_4}x_2^{\delta_2}x_3^{\gamma_2},x_1^{\alpha_5}x_2^{\delta_3}x_3^{\gamma_2}\}.$$
\\
To be an order ideal, $M$ must contain all the divisors of its elements:
$$\forall \tau \in M \textrm{, if } \sigma \mid \tau \textrm{ then } \sigma \in 
M,$$ so we have to lay down some conditions on the exponents.
\\
Let us start examining $x_1^{\alpha_1}x_2^{\beta_1}x_3^{\gamma_1}$ and 
$x_1^{\alpha_2}x_2^{\beta_1}x_3^{\gamma_1}$.
Knowing that $\alpha_1<\alpha_2$, we need to take $\alpha_1=0$ and $\alpha_2=1$.
Indeed, otherwise, $M$ should contain at least another term of the form  
$x_1^{\alpha_0}x_2^{\beta_1}x_3^{\gamma_1}$, $\alpha_0 \neq \alpha_1, \alpha_2$ 
and  $\alpha_0<\max(\alpha_1,\alpha_2)$.
The exponent $\beta_1$ must be equal to zero, otherwise at least $x_1^{\alpha_1}x_2^{\beta_1-1}x_3^{\gamma_1}$ and $x_1^{\alpha_2}x_2^{\beta_1-1}x_3^{\gamma_1}$ would belong to $M$. For analogous reasons, we have to choose  $\gamma_1=0,\gamma_2=1$ and $\alpha_3=\alpha_4=\alpha_5=0$. We get 

$$M=\{1, x_1, x_2^{\delta_1}x_3,x_2^{\delta_2}x_3,x_2^{\delta_3}x_3\}.$$

But let us examine $\delta_1<\delta_2<\delta_3$. Similarly to what said for the other exponents, we have only one possible choice for them, i.e. $\delta_1=0, \ \delta_2=1\, \delta_3=2$\footnote{Notice that these assignments are those given by BbC1 and BbC2.}, but then also $x_2$ and $x_2^2$ should belong to $M$, and this is impossible: there is only one possible power of $x_2$ for $\gamma_1=0$ and this contradiction proves that   $\sf B$ cannot be associated to any order ideal.
\end{example}
Inspired by example \ref{star1}, we define \emph{admissible Bar Codes} as follows:
\begin{Definition}\label{Admiss}
A Bar Code $\cB$ is \emph{admissible} if the set $M$ obtained by applying 
$BbC1$ and $BbC2$ to $\cB$  is an order ideal.
\end{Definition}
\begin{Remark}\label{BBC12}
By definition of order 
ideal, using BbC1 and BbC2 is the only way an order 
ideal can be associated to an admissible Bar Code. 
Indeed, if we label two consecutive bars with
two terms $\tau x_i^{a_i}$, $\tau x_i^{a_i +h}$, $h>1$, then also the terms 
$\sigma$ with $P_{x_i}(\sigma)=\tau x_i^{a_i +1}$ would belong to $M$ and it would have to label a bar between 
those labelled by $\tau x_i^{a_i}$ and $\tau x_i^{a_i +h}$, giving a 
contradiction.
\end{Remark}
We need now an \emph{admissibility criterion} for Bar Codes.
In order to be able to state it, we start with the following trivial lemma.
\begin{Lemma}\label{Predec}
Given a set $M \subset \kT$, the following conditions are equivalent
\begin{enumerate}
\item $M$ is an order ideal.
 \item $\forall \tau \in M$, if $\sigma \mid \tau$, then $\sigma \in M$.
 \item $\forall \tau\in M$ each predecessor of $\tau$ belongs to $M$.
\end{enumerate}
\end{Lemma}
We give then the definition of \emph{e-list}, associated to each 
$1$-bar of a given Bar Code.
\begin{Definition}\label{elist}
 Given a Bar Code $\cB$, 
 let us consider a $1$-bar $B_{j_1}^{(1)}$, with $j_1 
\in \{1,...,\mu(1)\}$.
 The \emph{e-list} associated to $B_{j_1}^{(1)}$ is the $n$-tuple 
$e(B_{j_1}^{(1)}):=(b_{j_1,1},....,b_{j_1,n})$, defined as follows:
 \begin{itemize}
  \item consider the $n$-bar  $B_{j_n}^{(n)}$, lying under 
  $B_{j_1}^{(1)}$. 
The number of $n$-bars on the left of $B_{j_n}^{(n)}$ is  $b_{j_1,n}.$
  \item for each $i=1,...,n-1$, let  $B_{j_{n-i+1}}^{(n-i+1)}$ and 
$B_{j_{n-i}}^{(n-i)}$ be 
  the $(n-i+1)$-bar and the $(n-i)$-bar 
lying under $B_{j_1}^{(1)}$. Consider the $(n-i+1)$-block associated to 
$B_{j_{n-i+1}}^{(n-i+1)}$. The number of $(n-i)$-bars of 
the block, which lie on  the 
left of $B_{j_{n-i}}^{(n-i)}$ is $b_{j_1,n-i}.$
    \end{itemize}
\end{Definition}

\begin{example}\label{elistEs}
For  the Bar Code $\cB$ 
 \begin{center}
\begin{tikzpicture}[scale=0.4]
\node at (-0.5,4) [] {${\scriptscriptstyle 0}$};
\node at (-0.5,0) [] {${\scriptscriptstyle 3}$};
\node at (-0.5,1.5) [] {${\scriptscriptstyle 2}$};
\node at (-0.5,3) [] {${\scriptscriptstyle 1}$};
 \draw [thick] (0,0) -- (7.9,0);
 \draw [thick] (9,0) -- (10.9,0);
 \node at (11.5,0) [] {${\scriptscriptstyle
x_3^2}$};
 \draw [thick] (0,1.5) -- (4.9,1.5);
 \draw [thick] (6,1.5) -- (7.9,1.5);
 \node at (8.5,1.5) [] {${\scriptscriptstyle
x_2^2}$};
 \draw [thick] (9,1.5) -- (10.9,1.5);
 \node at (11.5,1.5) [] {${\scriptscriptstyle
x_2x_3}$};
 \draw [thick] (0,3.0) -- (1.9,3.0);
 \draw [thick] (3,3.0) -- (4.9,3.0);
 \node at (5.5,3.0) [] {${\scriptscriptstyle
x_1^2}$};
 \draw [thick] (6,3.0) -- (7.9,3.0);
 \node at (8.5,3.0) [] {${\scriptscriptstyle
x_1x_2}$};
 \draw [thick] (9,3.0) -- (10.9,3.0);
 \node at (11.5,3.0) [] {${\scriptscriptstyle
x_1x_3}$};
 \node at (1,4.0) [] {\small $1$};
 \node at (4,4.0) [] {\small $x_1$};
 \node at (7,4.0) [] {\small $x_2$};
 \node at (10,4.0) [] {\small $x_3$};
\end{tikzpicture}
\end{center}
the e-lists are $e(B_{1}^{(1)}):=(0,0,0)$; $e(B_{2}^{(1)}):=(1,0,0)$; 
$e(B_{3}^{(1)}):=(0,1,0)$ and \\$e(B_{4}^{(1)}):=(0,0,1)$.
 \end{example}

\begin{Remark}\label{ElistExp}
 Given a Bar Code $\cB$, fix a $1$-bar  $B_{j}^{(1)}$, with $j \in 
\{1,...,\mu(1)\}$.\\
 Comparing definition \ref{elist} and the steps BbC1 and 
 BbC2 described above, we can observe that the values of the e-list 
$e(B_j^{(1)}):=(b_{j,1},....,b_{j,n})$ are exactly the
exponents of the term 
labelling $B_{j}^{(1)}$, obtained applying BbC1 and BbC2 to
$\cB$.
\end{Remark}
% 
% Poi osservo prima che i valori della e-list sono esattamente gli esponenti delle varie variabili. Infatti, siccome parto da 0 e metto esponenti consecutivi sulle varie barre, la n-barra j avrà $x_n^{j-1}$.
% Poi, la costruzione in BbC2 e' di isolare i 2 blocchi e semplicemente contare le barre sempre d'accapo da zero.
\begin{Proposition}[Admissibility criterion]\label{AdmCrit}
 A Bar Code $\cB$ is admissible if and only if, for each 
 $1$-bar $\cB_{j}^{(1)}$, $j \in \{1,...,\mu(1)\}$, the e-list 
$e(\cB_j^{(1)})=(b_{j,1},....,b_{j,n})$ satisfies the following condition: 
$\forall k \in \{1,...,n\} \textrm{ s.t. } b_{j,k}>0,\, \exists \overline{j} 
 \in \{1,...,\mu(1)\}\setminus \{j\} \textrm{ s.t. } $ $$
e(\cB_{\overline{j}}^{(1)})= (b_{j,1},...,b_{j,k-1}, (b_{j,k})-1, 
b_{j,k+1},...,b_{j,n}). $$

\end{Proposition}
\begin{proof}
% Let $\cB$ be an admissible Bar Code and let $\cN$ be the order ideal associated to $\cB$, which is obtained by applying
% BbC1 and BbC2 to $\cB$ (see remark \ref{BBC12}). By remark \ref{ElistExp}, the e-lists associated to the $1$-bars of $\cB$ are the exponents' lists
%  of their terms. For each $1$-bar $\cB_{j_1}^{(1)}$, $j_1 \in \{1,...,\mu(1)\}$ the e-list $e(\cB_j^{(1)})=(b_{j,1},....,b_{j,n})$ satisfies the condition
%   of the statement, since all the predecessors of $x_1^{b_{j,1}}\cdots x_n^{b_{j,n}}$ are in $\cN$ and the lists of the form $(b_{j,1},...,b_{j,k-1}, (b_{j,k})-1, b_{j,k+1},...,b_{j,n}) $ are exactly their exponents' lists, so they are e-lists of $1$-bars of $\cB$.
%  \\
% i Conversely, suppose the condition on the e-lists hold true for each $1$-bar. Applying BbC1 and BbC2 to $\cB$ we get that the e-lists are the exponents' lists of the terms in the set associated to $\cB$. Since the condition holds, all predecessors of each term belong to $M$, so $M$ is an order ideal by lemma \ref{Predec}. 
It is a trivial consequence of Lemma \ref{Predec}
%, Definition \ref{elist}
 and Remark \ref{ElistExp}.
\end{proof}
% 
% \textbf{SECONDO ME QUI DEVO DAVVERO DEFINIRE LA FUNZIONE: non so fino a che 
% punto la cosa sia ovvia }

Consider the following sets
$$\kA_n:=\{\cB \in \mathcal{B}_n \textrm{ s.t. } \cB  \textrm{ admissible}\} $$
$$\kN_n:=\{\cN \subset \kT,\, \vert \cN\vert < \infty \textrm{ s.t. } \cN 
\textrm{ order ideal}\}.$$
We can define the map 
$$\eta: \kA_n \rightarrow \kN_n $$
$$ \cB \mapsto \cN,$$
where $\cN$ is the order ideal obtained applying BbC1 and BbC2 to $\cB$.\\
By BbC1 and BbC2,  
%remark \ref{BBC12},
$\eta$ is a function; it is trivially surjective.
Moreover, it is injective since, if $\cB, \cB' \in \kA_n$
 and $\cB \neq \cB'$ they have at least one pair of indices $i,j$ s.t.
 $l_1(\cB_{j}^{(i)})\neq l_1({\cB'}_{j}^{(i)} )$ and this changes the result of 
  the application of BbC1/BbC2.
\\
From the arguments above, we can then deduce that there is a biunivocal 
correspondence between admissible $n$-Bar Codes and finite order ideals of 
$\kT\subset \ck[x_1,...,x_n]$.
\\
% % % % The lemma below will be useful to identify the Bar Codes of the monomial ideals we 
% % % % aim to count.

In the Lemma below we state some properties of admissible 
Bar Codes related to lengths.

\begin{Lemma}\label{Admdecr}
If $\cB$ is an admissible Bar Code, the following two conditions hold:
 \begin{itemize}
  \item[a)] $l_{n-1}(\cB^{(n)}_1)\geq...\geq l_{n-1}(\cB^{(n)}_{\mu(n)})$
  \item[b)] $\forall 1 \leq i \leq n-2$, $\forall 1 \leq j \leq \mu(i+2)$
take the $(i+2)$-bar $\cB^{(i+2)}_j$ and let $\cB^{(i+1)}_{j_1},...,\cB^{(i+1)}_{j_1+h}$ (where
$h$ satisfies $h \in \{0,...,\mu(i+1)-j_1\}$) be the $(i+1)$-bars lying over
 $\cB^{(i+2)}_j$. \\ Then 
  $l_{i}(\cB^{(i+1)}_{j_1})\geq...\geq l_{i}(\cB^{(i+1)}_{j_1+h}).$
  \end{itemize}

\end{Lemma}
\begin{proof}
Let us start proving a).  If for some $1\leq l\leq \mu(n)-1$ it holds 
 $l_{n-1}(\cB^{(n)}_l)<l_{n-1}(\cB^{(n)}_{l+1})$ the Bar Code would be not 
admissible. Indeed, let $\cB^{(1)}_k$ be  the rightmost $1$-bar over 
$\cB^{(n)}_{l+1}$ and $e(\cB^{(1)}_k)=(b_{k,1},...,b_{k,n})$ be its e-list.
 By construction (see Definition \ref{elist}), $b_{k,n-1}=l_{n-1}(\cB^{(n)}_{l+1})-1$. Now, this proves that 
 there cannot exist a $1$-bar labelling $(b_{k,1},...,b_{k,n-1},b_{k,n}-1)$, since 
  $l_{n-1}(\cB^{(n)}_l)<l_{n-1}(\cB^{(n)}_{l+1})$ and so the $1$-bars 
$\cB^{(1)}_{\overline{k}}$ over 
 $\cB^{(n)}_l$ have $b_{\overline{k},n-1}\leq 
l_{n-1}(\cB^{(n)}_{l})-1<l_{n-1}(\cB^{(n)}_{l+1})-1=b_{k,n-1}$, contradicting the assumption
of admissibility (see Proposition  \ref{AdmCrit}).
\\
\smallskip

An analogous 
argument proves that if for some 
 $\forall 1 \leq i \leq n-2$, $\forall 1 \leq j \leq \mu(i+2)$ we
take the $(i+2)$-bar $\cB^{(i+2)}_j$ and  $\cB^{(i+2)}_{j_1+h}$  s.t.
$h$ satisfies $h \in \{0,...,\mu(i+1)-j_1\}$ is the $(i+1)$-bars lying over
 $\cB^{(i+2)}_j$, it happens that for a fixed  $l \in \{1,...,\mu(i+1)-1-j_1\}$ 
  $l_{i}(\cB^{(i+1)}_{j_1+l})<l_{i}(\cB^{(i+1)}_{j_1+l+1})$, $\cB$ is not 
admissible and so also b) is true.
\end{proof}
In what follows, unless differently specified, we always consider admissible 
Bar Codes, so, in general, we will omit the word ``admissible''.

\begin{Remark}\label{InfBC}
 In principle, it is possible to represent with a Bar Code also
  infinite order ideals, by means of a simple modification, i.e.
 the introduction
  of the symbol ``$\rightarrow$'' immediately after a $l$-bar for some $1 
\leq l\leq n$, meaning that
there should actually be infinitely many $l$-blocks equal to that containing 
that bar. 
% %   
% %   
% %   the introduction
% %   of the symbol ``$\rightarrow$'' between two consecutive bars, meaning that 
% % there should actually be infinite bars between them.

For example, the Bar Code 
of $I=(x_1^2x_2^2)\triangleleft \ck[x_1,x_2]$, whose lexicographical Groebner 
escalier is    $\cN(I)=\{x_1^{h_1}x_2^{h_2}, x_1^{h_3}x_2^{h_4}, \, h_1,h_4 \in
\NN, h_2,h_3\in\{0,1\} \}$,   turns out to be
  
  \begin{center}
\begin{tikzpicture}
\node at (4.2,0.5) [] {${\scriptstyle 1}$};

\node at (4.7,0) [] {${\scriptscriptstyle \rightarrow}$};

\node at (5.2,0.5) [] {${\small x_2}$};

\node at (5.7,0) [] {${\scriptscriptstyle \rightarrow}$};
\node at (6.2,0.5) [] {${\small x_2^2}$};
\node at (7.2,0.5) [] {${\small x_1x_2^2}$};
\node at (7.7,-0.5) [] {${\scriptscriptstyle \rightarrow}$};

\draw [thick] (4,0) --(4.5,0);
\draw [thick] (4,-0.5) --(4.5,-0.5);

\draw [thick] (5,0) --(5.5,0);
\draw [thick] (5,-0.5) --(5.5,-0.5);
\draw [thick] (6,0) --(6.5,0);
\draw [thick] (7,0) --(7.5,0);
\draw [thick] (6,-0.5) --(7.5,-0.5);
\end{tikzpicture}
\end{center}

In particular, the arrow on the right of $1$ represents the terms
of the form $x_1^{h_1}$, $h_1 \in \NN\setminus \{0\}$, 
the one on the right of $x_2$  represents the terms
of the form $x_1^{h_1}x_2$, $h_1 \in \NN\setminus \{0\}$;
 finally the bottom arrow represents the terms 
 of the form $x_2^{h_4}, x_1x_2^{h_4}$, $h_4 \in \NN$, $h_4>2$.
\\
Since infinite
 Bar Codes are out of the topics of this paper, we will not treat them in detail.

\end{Remark}

\section{The star set}
Up to this point, we have discussed the link between Bar Codes and order ideals,
 i.e. we focused on the link between Bar Codes and Groebner escaliers of 
monomial ideals. 

In this section, we show that, given a Bar Code 
$\cB$ and the order ideal  $\cN =\eta(\cB)$
it is possible to deduce a very specific generating set 
for the monomial ideal $I$ s.t. $\cN(I)=\cN$.

    \begin{Definition}\label{StarSet}
 The \emph{star set} of an order ideal $\cN$ and 
 of its associated Bar Code $\cB=\eta^{-1}(\cN)$ is a set $\kF_\cN$ constructed as 
follows:
 \begin{itemize}
  \item[a)] $\forall 1 \leq i\leq n$, let $\tau_i$ be a term 
  which labels a $1$-bar lying over $\cB^{(i)}_{\mu(i)}$, 
  then 
  $x_iP_{x_i}(\tau_i)\in \kF_\cN$;
  \item[b)] $\forall 1 \leq i\leq n-1$, 
  $\forall 1 \leq j \leq \mu(i)-1$ let 
  $\cB^{(i)}_j$ and $\cB^{(i)}_{j+1}$ be two 
  consecutive bars not lying over the 
same $(i+1)$-bar and let $\tau^{(i)}_j$ be a term
which labels a $1$-bar lying 
over   $\cB^{(i)}_j$, then 
  $x_iP_{x_i}(\tau^{(i)}_j)\in \kF_\cN$.
 \end{itemize}
\end{Definition}
We usually represent $\kF_\cN$ within  the associated Bar Code $\cB$, inserting
each $\tau \in \kF_\cN$ on the right of the bar from which it is deduced.
Reading the terms from left to right and from the top to the bottom, $\kF_\cN$ 
is ordered w.r.t. Lex. 

\begin{example}\label{BCP}
For ${\sf
N}=\{1,x_1,x_2,x_3\}\subset
\mathbf{k}[x_1,x_2,x_3]$, associated to the Bar Code of example \ref{elistEs},  
we have $\kF_\cN=\{x_1^2,x_1x_2,x_2^2,x_1x_3,x_2x_3,x_3^2\}$; looking at 
Definition \ref{StarSet}, we can see that  the terms $x_1x_3,x_2x_3,x_3^2$ come 
from a), whereas the terms  
$x_1^2,x_1x_2,x_2^2$ come from b).

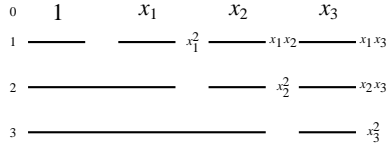
 \begin{center}
\begin{tikzpicture}[scale=0.4]
\node at (-0.5,4) [] {${\scriptscriptstyle 0}$};
\node at (-0.5,0) [] {${\scriptscriptstyle 3}$};
\node at (-0.5,1.5) [] {${\scriptscriptstyle 2}$};
\node at (-0.5,3) [] {${\scriptscriptstyle 1}$};
 \draw [thick] (0,0) -- (7.9,0);
 \draw [thick] (9,0) -- (10.9,0);
 \node at (11.5,0) [] {${\scriptscriptstyle
x_3^2}$};
 \draw [thick] (0,1.5) -- (4.9,1.5);
 \draw [thick] (6,1.5) -- (7.9,1.5);
 \node at (8.5,1.5) [] {${\scriptscriptstyle
x_2^2}$};
 \draw [thick] (9,1.5) -- (10.9,1.5);
 \node at (11.5,1.5) [] {${\scriptscriptstyle
x_2x_3}$};
 \draw [thick] (0,3.0) -- (1.9,3.0);
 \draw [thick] (3,3.0) -- (4.9,3.0);
 \node at (5.5,3.0) [] {${\scriptscriptstyle
x_1^2}$};
 \draw [thick] (6,3.0) -- (7.9,3.0);
 \node at (8.5,3.0) [] {${\scriptscriptstyle
x_1x_2}$};
 \draw [thick] (9,3.0) -- (10.9,3.0);
 \node at (11.5,3.0) [] {${\scriptscriptstyle
x_1x_3}$};
 \node at (1,4.0) [] {\small $1$};
 \node at (4,4.0) [] {\small $x_1$};
 \node at (7,4.0) [] {\small $x_2$};
 \node at (10,4.0) [] {\small $x_3$};
\end{tikzpicture}
\end{center}
\end{example}

In \cite{CMR}, given a monomial ideal $I$, the authors define
 the following set, calling it \emph{star set}:
% $$\mathcal{F}(I)=\{x^{\gamma} \in \mathcal{T}\setminus {\sf N}(I) \, \vert \, 
% \frac{x^{\gamma}}{\min(x^{\gamma})} \in {\sf N}(I) \}.$$

$$\mathcal{F}(I)=\left\{x^{\gamma} \in \mathcal{T}\setminus {\sf N}(I) \,
\left\vert \,
\frac{x^{\gamma}}{\min(x^{\gamma})} \right. \in {\sf N}(I) \right\}.$$

We can prove the following proposition, which connects the definition above to 
our construction.
\begin{Proposition}\label{DefSt}
With the above notation $\mathcal{F}_{\sf N}=\mathcal{F}(I)$.
\end{Proposition}
\begin{proof}
We start proving $\kF_\cN \subseteq \kF(I)$.\\
Consider $\sigma \in \kF_\cN$; by definition of $\kF_\cN$
there are two possibilities
 \begin{itemize}
  \item[a)] $\sigma =x_iP_{x_i}(\tau_i)$, with $1 \leq i\leq n$
  and  $\tau_i$ a  term   which labels a $1$-bar lying over 
$\cB^{(i)}_{\mu(i)}$;
  \item[b)] $\sigma = x_iP_{x_i}(\tau^{(i)}_j)$, with 
    $ 1 \leq i\leq n-1$,   $ 1 \leq j \leq \mu(i)-1$
    $\tau^{(i)}_j$  a term
which labels a $1$-bar lying 
over   $\cB^{(i)}_j$, under the condition that 
  $\cB^{(i)}_j$  $\cB^{(i)}_{j+1}$ do not lie over the 
same $(i+1)$-bar.
 \end{itemize} 
 Let us examine a) and b) separately.
 \begin{itemize}
  \item[a)] By definition, $\sigma >\tau_i$; indeed 
  $\deg_h(\sigma) = \deg_h(\tau_i)$ for $i+1\leq h\leq n$ and 
  $\deg_i(\sigma) > \deg_i(\tau_i)$. Clearly, 
  $\sigma \notin \cN$, because if it was in the Groebner escalier,
  applying the steps described in Definition  \ref{BarCodeDiag}, 
  $P_{x_i}(\sigma)=\sigma= x_iP_{x_i}(\tau_i)$ would be put in a list
   that is subsequent to the one containing $P_{x_i}(\tau_i)$, but, in this 
   case, there would be $\mu(i)+1$ $i$-bars instead of $\mu(i)$,
   contradicting the definition of $\mu(i).$ Since $\min(\sigma)=x_i$,
   $\frac{\sigma}{\min(\sigma)}=P_{x_i}(\tau_i)\mid \tau_i$, so 
$\frac{\sigma}{\min(\sigma)} \in \cN$ and $\sigma \in \kF(I)$.
  \item[b)] Analogously to case a), $\sigma >\tau^{(i)}_j$.
  Let us prove that $\sigma \notin \cN$. 
  If $\sigma \in \cN$
 then $\sigma $ would label a $1$-bar 
 lying over $\cB_{j+1}^{(i)}$ but, since 
$P_{x_{i+1}}(\sigma)=P_{x_{i+1}}(\tau^{(i)}_j)$, 
$\cB^{(i)}_j$  $\cB^{(i)}_{j+1}$ would lie over the 
same $(i+1)$-bar, contradicting the hypothesis.
As above, since $\min(\sigma)=x_i$,
   $\frac{\sigma}{\min(\sigma)}=P_{x_i}(\tau^{(i)}_j)\mid \tau^{(i)}_j$, so 
$\frac{\sigma}{\min(\sigma)} \in \cN$ and $\sigma \in \kF(I)$.
 \end{itemize}
We prove now that $\kF_\cN \supseteq \kF(I)$.\\
Let us consider $\sigma \in \kF(I)$ and let $\min(\sigma)=x_i$, $1\leq i\leq n$.
By definition of $\kF(I)$, $\sigma \notin \cN$ and 
$\widetilde{\sigma}:=\frac{\sigma}{x_i}\in\cN$, so it labels a $1$-bar lying 
over some $i$-bar $\cB^{(i)}_j$. Denote by 
$\cB^{(1)}_{\overline{j}},...,\cB^{(1)}_{\overline{j}+h}$ (where $h$ satisfies 
$0\leq h \leq \mu(i)-\overline{j}$) the $1$-bars lying over $\cB^{(i)}_j$.
Two possibilities may occur:
% $$\mathcal{F}(I)=\{x^{\gamma} \in \mathcal{T}\setminus {\sf N}(I) \, \vert \, 
% \frac{x^{\gamma}}{\min(x^{\gamma})} \in {\sf N}(I) \}.$$

\begin{itemize}
 \item[a)] $\overline{j}+h=\mu(i)$; in this case 
$x_iP_{x_i}(\widetilde{\sigma})=\sigma \in \kF_\cN$ by  Definition  
\ref{StarSet}.
\item[b)] otherwise consider the term $\tau_{\overline{j}+h}$, which labels 
$\cB^{(1)}_{\overline{j}+h}$, and the subsequent term 
$\tau_{\overline{j}+h+1}$, labelling $\cB^{(1)}_{\overline{j}+h+1}$. Notice 
that 
$P_{x_i}(\tau_{\overline{j}+h})=P_{x_i}(\widetilde{\sigma})$.
By Definition \ref{BarCodeDiag}, 
$\tau_{\overline{j}+h}<_{Lex}\tau_{\overline{j}+h+1}$.
If $P_{x_i}(\tau_{\overline{j}+h})=P_{x_i}(\tau_{\overline{j}+h+1})$ this would 
contradict the maximality of $h$, so, by property 3. of the operators 
$P_{x_i}$, 
it must 
be $P_{x_i}(\tau_{\overline{j}+h})<_{Lex}P_{x_i}(\tau_{\overline{j}+h+1})$.
But, if  $P_{x_{i+1} 
}(\tau_{\overline{j}+h})=P_{x_{i+1}}(\tau_{\overline{j}+h+1})$, then $\sigma 
\mid \tau_{\overline{j}+h+1}$ and so $\sigma \in \cN$, that is impossible since 
$\sigma \in \kF(I)$.
This means then that  $P_{x_{i+1} 
}(\tau_{\overline{j}+h})<_{Lex}P_{x_{i+1}}(\tau_{\overline{j}+h+1})$, so we can 
deduce that $\cB^{(1)}_{\overline{j}+h}$ and  $\cB^{(1)}_{\overline{j}+h+1}$ 
lie over two consecutive $i$-bars not lying over the same $(i+1)$-bar, so 
$\sigma=x_iP_{x_i}(\widetilde{\sigma})=x_iP_{x_i}(\tau_{\overline{j}+h}) \in 
\kF_\cN$. 
\end{itemize}
\end{proof}

\begin{Remark}\label{FI Border}
By Proposition \ref{DefSt}, being $\kF_\cN= \kF(I)$, it trivially holds ${\sf 
G}(I) \subseteq \mathcal{F}_{\sf N}\subseteq {\sf B}(I)$. In general, the 
inclusions may be strict; if $\mathcal{F}_{\sf N}=\cG(I)$, we say that ${\sf
B_N}:= \eta^{-1}(\cN)$ is a \emph{full} Bar Code.
\end{Remark}

The star set $\kF(I)$ of a monomial ideal $I$ is strongly connected to Janet's 
theory \cite{J1,J2,J3,J4} and to the notion of Pommaret basis \cite{Pom, 
PomAk, SeiB}, as explicitly pointed out in \cite{CMR}. For completeness sake, we 
recall it below.
\begin{definition}\cite[ppg.75-9]{J1}\label{multiplicative}
Let  $M\subset \mathcal{T}$ be a set of terms
 and  $\tau=x_1^{\gamma_1}\cdots x_n^{\gamma_n} $
be an element of $M$.
A variable $x_j$ is called \emph{multiplicative}
for $\tau$ with respect to $M$ if there is no term in
$M$ of the form
$\tau'=x_1^{\delta_1}\cdots x_j^{\delta_j}x_{j+1}^{\gamma_{j+1}} \cdots  
x_n^{\gamma_n}$
with $\delta_j>\gamma_j$.
 We will denote by $\mult_M(\tau)$ the set of
multiplicative variables for $\tau$ with respect to $M$.
\end{definition}
\begin{definition}
With the previous notation, the \emph{cone} of  $\tau $ with respect to $M$ 
  is the set
$$\off_M (\tau):=\{\tau x_1^{\lambda_1} \cdots x_n^{\lambda_n} \,\vert \, 
\textrm{where } \lambda_j\neq 0 \textrm{ only if } x_j \textrm{ is 
multiplicative for }
\tau \textrm{ w.r.t. } M\}.$$
\end{definition}
\begin{definition}\cite[ppg.75-9]{J1}\label{Complete}
A set of terms $M\subset \mathcal{T}$ is called
\emph{complete} if for every $ \tau \in M$ and
 $ x_j\notin mult_M(\tau)$, there exists $ \tau' \in M$
such that $x_j \tau \in \off_M (\tau')$.

 Moreover,  $M$ is \emph{stably complete} \cite{SeiB, CMR}   if it
 is complete and for every $\tau\in M$ it holds
  $\mult_M(\tau)=\{ x_i \  \vert \  x_i\leq \min(\tau) \}$.
\\
If a set $M$ is stably complete and finite, then it is the \emph{Pommaret basis} 
of $I=(M)$.
\end{definition}

\begin{Theorem}\label{moltiplicative} For every monomial ideal $I$, 
the star set 
$\mathcal{F}(I)$ is the unique stably complete system    
of generators of $I$. Hence, if $M$ is stably complete, $M=\mathcal{F}((M))$.
\end{Theorem}

 By Proposition 
\ref{DefSt}, the Bar Code gives a simple way to deduce 
the star set from the 
Groebner escalier of a zerodimensional monomial ideal.
\section{Counting stable ideals}\label{COUNTSTAB}
In this section, we connect the Bar Code associated to the Groebner escalier 
of a  stable monomial ideal to the theory of integer
and  plane partitions, in order to find the number of stable ideals in 
two or three variables with constant affine Hilbert polynomial
$H_{\_}(t)=p \in \NN$.
\\
 \medskip
 \\
We start recalling some definitions and known facts about stable and strongly 
stable ideals.
\begin{definition}{\em(\cite{J2}[pg.41], \cite{J4}) ( c.f.\cite{SPES}[IV.pg.673,679] )} \label{Stab}
A monomial ideal $J\triangleleft \mathcal{P}=\mathbf{k}[x_1,...,x_n]$ is called
 \emph{stable} \emph{\cite{EK}} if it holds
$$\tau \in J, \ x_j >\min(\tau) \Longrightarrow \frac{x_j\tau}{\min(\tau)}\in 
J.$$
\end{definition}

\begin{definition}[\cite{Rob1, Rob2, Gun1,Gun2, GAL, PEEVA}]\label{StronglyStab}
A monomial ideal $I\triangleleft \mathcal{P}=\mathbf{k}[x_1,...,x_n]$ is called 
\emph{strongly stable} \emph{\cite{AH, AH2}}
if, for every term $\tau\in I$ and pair of variables
$x_i,\ x_j$ such that $x_i\vert \tau$ and $x_i<x_j$,
then also $
\frac{\tau x_j}{x_i} $ belongs to $I$ or, equivalently,
for every $\sigma \in {\sf N}(I)$, and pair of
variables $x_i,\ x_j$ such that $x_i\vert \sigma$
and $x_i>x_j$,
then also $\frac{\sigma x_j}{x_i} $ belongs to ${\sf N}(I)$.
\end{definition}

It is well known 
%(\textbf{Citazione? Chiedi Paolo Lella} NO Galligo, americani per il nome ma 
%Delassus Gunther Robinson e Janet + SPES ) 
 that, in order to verify the 
(strong) 
stability of a monomial ideal, we can verify the conditions above for the terms 
in $\cG(I)$.

\begin{example}[\cite{CMR}]\label{StSrtSt}
 In $k[x_1,x_2,x_3]$ with $x_1<x_2<x_3$:
\begin{itemize}
\item the ideal $I_1=(x_1^3, x_1x_2,x_2^2,x_1^2x_3,x_2x_3,x_3^2) $ is stable.\\
Indeed, we have:\\ $\frac{(x_1^3)x_2}{x_1}=x_1^2x_2 \in I_1$,\\ 
$\frac{(x_1^3)x_3}{x_1}=x_1^2x_3 \in I_1$,\\$\frac{(x_1x_2)x_2}{x_1}=x_2^2 \in 
I_1$,\\$\frac{(x_1x_2)x_3}{x_1}=x_2x_3 \in 
I_1$,\\$\frac{(x_2)^2x_3}{x_2}=x_2x_3 \in 
I_1$,\\$\frac{(x_1^2x_3)x_2}{x_1}=x_1x_2x_3 \in 
I_1$, \\$\frac{(x_1^2x_3)x_3}{x_1}=x_1x_3^2 \in 
I_1$,\\ and $\frac{(x_2x_3)x_3}{x_2}=x_2x_3^2 \in 
I_1$.\\
Anyway, it is not strongly stable, since $x_1x_2 \in I_1$,
but $\frac{(x_1x_2)x_3}{x_2}=x_1x_3\notin I_1$;
\item the ideal $I_2=(x_1^2,x_1x_2,x_2^2,x_3)$ is strongly stable, since\\ 
$\frac{(x_1^2)x_2}{x_1}=x_1x_2 \in 
I_2$,\\
$\frac{(x_1^2)x_3}{x_1}=x_1x_3 \in 
I_2$,\\
$\frac{(x_1x_2)x_2}{x_1}=x_1x_2^2 \in 
I_2$,\\
$\frac{(x_1x_2)x_3}{x_1}=x_2x_3 \in 
I_2$,\\
$\frac{(x_1x_2)x_3}{x_2}=x_1x_3 \in 
I_2$,\\
$\frac{(x_2^2)x_3}{x_2}=x_2x_3 \in 
I_2$
\end{itemize}
\end{example}

\begin{Proposition}[\cite{CMR}]\label{QstabFugualG}
Let  $J$ be a monomial ideal.
Then TFAE:
\begin{itemize}
 \item[i)]  $J$   is stable
 \item[ii)] $\mathcal{F}(J)=\cG  (J)$
\end{itemize}

\end{Proposition}
A simple property, useful for what follows, and trivially following from
Remark \ref{FI 
Border} and Proposition \ref{QstabFugualG}, is that Bar Codes of (strongly) 
stable ideals are \emph{full}.

\begin{example}\label{BcStStrSt}
 In $\ck[x_1,x_2,x_3]$ with $x_1<x_2<x_3$, consider again the ideals $I_1,I_2$ of 
example \ref{StSrtSt}:
\begin{itemize}
 \item  the Bar Code $\cB_1$ associated to $I_1=(x_1^3, 
x_1x_2,x_2^2,x_1^2x_3,x_2x_3,x_3^2) $  is
\begin{center}
\begin{tikzpicture}
\node at (3.8,0.5) [] {${\scriptscriptstyle 0}$};
\node at (3.8,0) [] {${\scriptscriptstyle 1}$};
\node at (3.8,-0.5) [] {${\scriptscriptstyle 2}$};
\node at (3.8,-1) [] {${\scriptscriptstyle 3}$};
\node at (4.7,0.5) [] {\small $1$};
\node at (5.7,0.5) [] {\small $x_1$};
\node at (6.7,0.5) [] {\small $x_1^2$};
\node at (7.7,0.5) [] {\small $x_2$};
\node at (8.7,0.5) [] {\small $x_3$};
\node at (9.7,0.5) [] {\small $x_1x_3$};
\draw [thick] (4.5,0)--(5,0);
\draw [thick] (5.5,0)--(6,0);
\draw [thick] (6.5,0)--(7,0);
\node at (7.2,0) [] {$\scriptstyle{x_1^3}$};
\draw [thick] (7.5,0)--(8,0);
\node at (8.2,0) [] {$\scriptstyle{x_1x_2}$};
\draw [thick] (8.5,0)--(9,0);
\draw [thick] (9.5,0)--(10,0);
\node at (10.3,0.1) [] {$\scriptstyle{x_1^2x_3}$};

\draw [thick] (4.5,-0.5)--(7,-0.5);
\draw [thick] (7.5,-0.5)--(8,-0.5);
\node at (8.2,-0.5) [] {$\scriptstyle{x_2^2}$};

\draw [thick] (8.5,-0.5)--(10,-0.5);
\node at (10.2,-0.5) [] {$\scriptstyle{x_2x_3}$};
\draw [thick] (4.5,-1)--(8,-1);
\draw [thick] (8.5,-1)--(10,-1);
\node at (10.2,-1) [] {$\scriptstyle{x_3^2}$};

\end{tikzpicture}
\end{center}
and we have $\mathcal{F}(I_1)=\cG(I_1)=\{x_1^3, 
x_1x_2,x_2^2,x_1^2x_3,x_2x_3,x_3^2\}$
 \item the Bar Code $\cB_2$ associated to  $I_2=(x_1^2,x_1x_2,x_2^2,x_3)$
 is 
 
 \begin{center}
\begin{tikzpicture}
% \node at (3.8,0.5) [] {${\scriptscriptstyle 0}$};
 \node at (3.8,0) [] {${\scriptscriptstyle 1}$};
 \node at (3.8,-0.5) [] {${\scriptscriptstyle 2}$};
 \node at (3.8,-1) [] {${\scriptscriptstyle 3}$};
\node at (4.7,0.5) [] {\small $1$};
\node at (5.7,0.5) [] {\small $x_1$};
\node at (6.7,0.5) [] {\small $x_2$};
% \node at (7.7,0.5) [] {\small $x_2$};
% \node at (8.7,0.5) [] {\small $x_3$};
% \node at (9.7,0.5) [] {\small $x_1x_3$};
\draw [thick] (4.5,0)--(5,0);
\draw [thick] (5.5,0)--(6,0);
\draw [thick] (6.5,0)--(7,0);
 \node at (6.2,0) [] {$\scriptstyle{x_1^2}$};
  \node at (7.3,0) [] {$\scriptstyle{x_1x_2}$};
    \node at (7.3,-0.5) [] {$\scriptstyle{x_2^2}$};
  \node at (7.3,-1) [] {$\scriptstyle{x_3}$};

% \draw [thick] (7.5,0)--(8,0);
% % \node at (8.2,0) [] {$\scriptstyle{x_1x_2}$};
% \draw [thick] (8.5,0)--(9,0);
% \draw [thick] (9.5,0)--(10,0);
% \node at (10.3,0.1) [] {$\scriptstyle{x_1^2x_3}$};

\draw [thick] (4.5,-0.5)--(6,-0.5);
\draw [thick] (6.5,-0.5)--(7,-0.5);
% \node at (8.2,-0.5) [] {$\scriptstyle{x_2^2}$};

% \node at (10.2,-0.5) [] {$\scriptstyle{x_2x_3}$};
\draw [thick] (4.5,-1)--(7,-1);
% \draw [thick] (8.5,-1)--(10,-1);
% % \node at (10.2,-1) [] {$\scriptstyle{x_3^2}$};

\end{tikzpicture}
\end{center}
and we have $\mathcal{F}(I_2)=\cG(I_2)=\{x_1^2,x_1x_2,x_2^2,x_3\}$
\end{itemize}
We see that, as expected, both their Bar Codes are full.
 
\end{example}

\begin{Proposition}\label{barredecr}
Let $I \triangleleft \ck[x_1,...,x_n] $ be a stable zerodimensional monomial 
ideal and let $\cB$ be its Bar Code. Then the following two conditions hold:
 \begin{itemize}
  \item[a)] $l_{n-1}(\cB^{(n)}_1)>...>l_{n-1}(\cB^{(n)}_{\mu(n)})$
  \item[b)] $\forall 1 \leq i \leq n-2$, $\forall 1 \leq j \leq \mu(i+2)$
take the $(i+2)$-bar $\cB^{(i+2)}_j$ and let $\cB^{(i+1)}_{j_1},...,\cB^{(i+1)}_{j_1+h}$, s.t.
$h$ satisfies $h \in \{0,...,\mu(i+1)-j_1\}$ be the $(i+1)$-bars lying over
 $\cB^{(i+2)}_j$. \\ Then 
  $l_{i}(\cB^{(i+1)}_{j_1})>...>l_{i}(\cB^{(i+1)}_{j_1+h}).$
  \end{itemize}
\end{Proposition}
\begin{proof}
By lemma \ref{Admdecr} the case $<$ cannot occur.\\
Suppose now that for some $1\leq l\leq \mu(n)-1$ it holds 
 $l_{n-1}(\cB^{(n)}_l) = l_{n-1}(\cB^{(n)}_{l+1})$, let $\cB^{(1)}_k$ 
be  the rightmost $1$-bar over  $\cB^{(n)}_{l}$ and call $\tau_k$ the term 
labelling $\cB^{(1)}_k$. By definition of star set $x_{n-1}P_{x_{n-1}}(\tau_k) 
\in \kF(I) \subset I$; moreover, clearly we know that $P_{x_{n-1}}(\tau_k) \in \cN(I)$. But if 
 $l_{n-1}(\cB^{(n)}_l) = l_{n-1}(\cB^{(n)}_{l+1})$, then 
 $x_nP_{x_{n-1}}(\tau_k)= \frac{x_{n-1}P_{x_{n-1}}(\tau_k)}{x_{n-1}}x_n    \notin 
I $ and this contradicts the stability of $I$.

If for some 
 $1 \leq i \leq n-2$, $\forall 1 \leq j \leq \mu(i+2)$ we
take the $(i+2)$-bar $\cB^{(i+2)}_j$ and  $\cB^{(i+1)}_{j_1}...,\cB^{(i+i)}_{j_1+h}$ (where
$h$ satisfies $h \in \{0,...,\mu(i+1)-j_1\}$) are the $(i+1)$-bars lying over
 $\cB^{(i+2)}_j$, it happens that for a fixed  $l \in \{1,...,\mu(i+1)-1-j_1\}$ 
  $l_{i}(\cB^{(i+1)}_{j_1+l})=l_{i}(\cB^{(i+1)}_{j_1+l+1})$, an analogous 
argument proves that $I$ cannot be stable.
\end{proof}
In the example below, we show that there are also non-stable
ideals satisfying conditions a) and b).
\begin{example}\label{NotStable}
For the ideal 
$I=(x_1^2,x_1x_2,x_2^2,x_1x_3,x_2x_3,x_3^2,x_2x_4,
x_3x_4,x_4^2)\triangleleft \ck[x_1,x_2,x_3,x_4]$, we have
 $\cN(I)=\{1,x_1,x_2,x_3,x_4,x_1x_4\}$ and the associated Bar Code $\cB$
  is

\begin{center}
\begin{tikzpicture}
% \node at (3.8,0.5) [] {${\scriptscriptstyle 0}$};
 \node at (3.8,0) [] {${\scriptscriptstyle 1}$};
 \node at (3.8,-0.5) [] {${\scriptscriptstyle 2}$};
 \node at (3.8,-1) [] {${\scriptscriptstyle 3}$};
  \node at (3.8,-1.5) [] {${\scriptscriptstyle 4}$};
\node at (4.7,0.5) [] {\small $1$};
\node at (5.7,0.5) [] {\small $x_1$};
\node at (6.7,0.5) [] {\small $x_2$};
\node at (7.7,0.5) [] {\small $x_3$};
  \node at (8.2,0) [] {$\scriptstyle{x_1x_3}$};
  \node at (8.2,-0.5) [] {$\scriptstyle{x_2x_3}$};
  \node at (8.2,-1) [] {$\scriptstyle{x_3^2}$};

\node at (8.7,0.5) [] {\small $x_4$};
\node at (9.7,0.5) [] {\small $x_1x_4$};

\node at (10.2,0) [] { $\scriptstyle{x_1^2x_4}$};
\node at (10.2,-0.5) [] { $\scriptstyle{x_2x_4}$};
\node at (10.2,-1) [] {  $\scriptstyle{x_3x_4}$};
\node at (10.2,-1.5) [] { $\scriptstyle{x_4^2}$};

% \node at (7.7,0.5) [] {\small $x_2$};
% \node at (8.7,0.5) [] {\small $x_3$};
% \node at (9.7,0.5) [] {\small $x_1x_3$};
\draw [thick] (4.5,0)--(5,0);
\draw [thick] (5.5,0)--(6,0);
\draw [thick] (6.5,0)--(7,0);

\draw [thick] (7.5,0)--(8,0);
\draw [thick] (8.5,0)--(9,0);
\draw [thick] (9.5,0)--(10,0);

  \node at (6.2,0) [] {$\scriptstyle{x_1^2}$};
   \node at (7.3,0) [] {$\scriptstyle{x_1x_2}$};
     \node at (7.3,-0.5) [] {$\scriptstyle{x_2^2}$};
%   \node at (7.3,-1) [] {$\scriptstyle{x_3}$};

% \draw [thick] (7.5,0)--(8,0);
% % \node at (8.2,0) [] {$\scriptstyle{x_1x_2}$};
% \draw [thick] (8.5,0)--(9,0);
% \draw [thick] (9.5,0)--(10,0);
% \node at (10.3,0.1) [] {$\scriptstyle{x_1^2x_3}$};

\draw [thick] (4.5,-0.5)--(6,-0.5);
\draw [thick] (6.5,-0.5)--(7,-0.5);
\draw [thick] (7.5,-0.5)--(8,-0.5);
\draw [thick] (8.5,-0.5)--(10,-0.5);

% \node at (8.2,-0.5) [] {$\scriptstyle{x_2^2}$};

% \node at (10.2,-0.5) [] {$\scriptstyle{x_2x_3}$};
\draw [thick] (4.5,-1)--(7,-1);
\draw [thick] (7.5,-1)--(8,-1);
\draw [thick] (8.5,-1)--(10,-1);
% \draw [thick] (8.5,-1)--(10,-1);
% % \node at (10.2,-1) [] {$\scriptstyle{x_3^2}$};
\draw [thick] (4.5,-1.5)--(8,-1.5);
\draw [thick] (8.5,-1.5)--(10,-1.5);
\end{tikzpicture}
\end{center}
The star set is 
$\kF(I)=\{x_1^2,x_1x_2,x_2^2,x_1x_3,x_2x_3,x_3^2,x_1^2x_4,x_2x_4,
x_3x_4,x_4^2\}$ and we have $\kF(I)\supsetneq \cG(I)$,
so $I$ is not stable\footnote{We can also prove that $I$ is not stable using 
the definition, indeed we have $x_1^2 \in I $ but $x_1x_4 \notin I$.}.
\\
We can observe that $\cB$ satisfies conditions a) b)
of Proposition \ref{barredecr}. Indeed:\\
a) $2=l_{3}(\cB^{(4)}_1)>1=l_{3}(\cB^{(4)}_2)$;\\
b) $2=l_{1}(\cB^{(2)}_1)>1=l_{1}(\cB^{(2)}_2)$; 
 $2=l_{2}(\cB^{(3)}_1)>1=l_{2}(\cB^{(3)}_2)$.
\end{example}

In the following two examples, we show that the result of Proposition \ref{barredecr} is only \emph{local}, even if we consider strongly stable ideals, then strengthening the hypothesis of Proposition \ref{barredecr}.

This means that in general, fixed a row $2 \leq 
i< n$ of the Bar Code $\cB$ associated to a (even strongly) stable monomial ideal 
$I$,
 it does not hold
 $$l_{(i-1)}(\cB^{(i)}_1)>...>l_{(i-1)}(\cB^{(i)}_{\mu(i)}),$$
in particular, the $(i-1)$-length could even be completely unordered.

 \begin{example}\label{7Hilb6}
The Bar Code $\cB$, associated to the (strongly) stable monomial ideal \\
$I=(x_1^3,x_1x_2,x_2^2,x_1x_3,x_2x_3,x_3^2, 
x_1x_4,x_2x_4,x_3x_4,x_4^2)\triangleleft \mathbf{k}[x_1,x_2,x_3,x_4]$, 
 is:
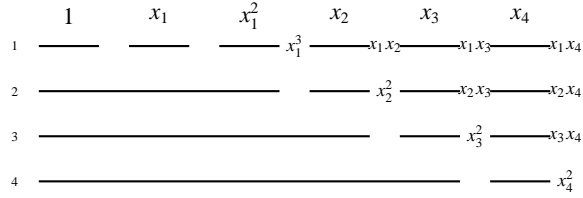
\begin{center}
\begin{tikzpicture}[scale=0.4]
% % % % \node at (-0.8,5.5) [] {${\scriptscriptstyle 0}$};
\node at (-0.8,4.5) [] {${\scriptscriptstyle 1}$};
\node at (-0.8,3) [] {${\scriptscriptstyle 2}$};
\node at (-0.8,1.5) [] {${\scriptscriptstyle 3}$};

\node at (-0.8,0) [] {${\scriptscriptstyle 4}$};
\draw [thick] (0,0) -- (14,0);
\draw [thick] (15,0) -- (17,0);
\node at (17.5,0) [] {$\scriptstyle{x_4^2}$};
\draw [thick] (0,1.5) -- (11,1.5);
\draw [thick] (12,1.5) -- (14,1.5);
\node at (14.5,1.5) [] {$\scriptstyle{x_3^2}$};
\draw [thick] (15,1.5) -- (17,1.5);
\node at (17.5,1.5) [] {$\scriptstyle{x_3x_4}$};
\draw [thick] (0,3.0) -- (8,3.0);
\draw [thick] (9,3.0) -- (11,3.0);
\node at (11.5,3.0) [] {$\scriptstyle{x_2^2}$};
\draw [thick] (12,3.0) -- (14,3.0);
\node at (14.5,3.0) [] {$\scriptstyle{x_2x_3}$};
\draw [thick] (15,3.0) -- (17,3.0);
\node at (17.5,3.0) [] {$\scriptstyle{x_2x_4}$};
\draw [thick] (0,4.5) -- (2,4.5);
\draw [thick] (3,4.5) -- (5,4.5);
\draw [thick] (6,4.5) -- (8,4.5);
\node at (8.5,4.5) [] {$\scriptstyle{x_1^3}$};
\draw [thick] (9,4.5) -- (11,4.5);
\node at (11.5,4.5) [] {$\scriptstyle{x_1x_2}$};
\draw [thick] (12,4.5) -- (14,4.5);
\node at (14.5,4.5) [] {$\scriptstyle{x_1x_3}$};
\draw [thick] (15,4.5) -- (17,4.5);
\node at (17.5,4.5) [] {$\scriptstyle{x_1x_4}$};
\node at (1,5.5) [] {\small $1$};
\node at (4,5.5) [] {\small $x_1$};
\node at (7,5.5) [] {\small $x_1^{2}$};
\node at (10,5.5) [] {\small $x_2$};
\node at (13,5.5) [] {\small $x_3$};
\node at (16,5.5) [] {\small $x_4$};
\end{tikzpicture}
\end{center}
and it holds 
$$2=l_2(\cB_1^{(3)})>l_2(\cB_2^{(3)})=l_2(\cB_3^{(3)}) =1.$$
\end{example}

\begin{example}\label{16Hilb9}
The (strongly) stable monomial ideal 
$I= (x_1^3,x_1^2x_2,x_1x_2^2,x_2^3,x_1^2x_3,x_1x_2x_3,x_2^2x_3, 
x_3^2)\triangleleft \mathbf{k}[x_1,x_2,x_3]$ is
associated 
to the Bar Code displayed below
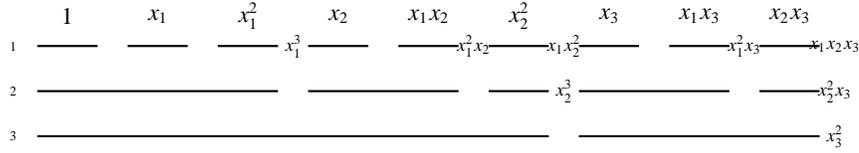
\begin{center}
\begin{tikzpicture}[scale=0.4]
% % % \node at (-0.8,4) [] {${\scriptscriptstyle 0}$};
\node at (-0.8,3) [] {${\scriptscriptstyle 1}$};
\node at (-0.8,1.5) [] {${\scriptscriptstyle 2}$};
\node at (-0.8,0) [] {${\scriptscriptstyle 3}$};

 \draw [thick] (0,0) -- (17,0);
 \draw [thick] (18,0) -- (26,0);
 \node at (26.5,0) [] {$\scriptstyle{x_3^2}$};
 \draw [thick] (0,1.5) -- (8,1.5);
 \draw [thick] (9,1.5) -- (14,1.5);
 \draw [thick] (15,1.5) -- (17,1.5);
 \node at (17.5,1.5) [] {$\scriptstyle{x_2^3}$};
 \draw [thick] (18,1.5) -- (23,1.5);
 \draw [thick] (24,1.5) -- (26,1.5);
 \node at (26.5,1.5) [] {$\scriptstyle{x_2^2x_3}$};
 \draw [thick] (0,3.0) -- (2,3.0);
 \draw [thick] (3,3.0) -- (5,3.0);
 \draw [thick] (6,3.0) -- (8,3.0);
 \node at (8.5,3.0) [] {$\scriptstyle{x_1^3}$};
 \draw [thick] (9,3.0) -- (11,3.0);
 \draw [thick] (12,3.0) -- (14,3.0);
 \node at (14.5,3.0) [] {$\scriptstyle{x_1^2x_2}$};
 \draw [thick] (15,3.0) -- (17,3.0);
 \node at (17.5,3.0) [] {$\scriptstyle{x_1x_2^2}$};
 \draw [thick] (18,3.0) -- (20,3.0);
 \draw [thick] (21,3.0) -- (23,3.0);
 \node at (23.5,3.0) [] {$\scriptstyle{x_1^2x_3}$};
 \draw [thick] (24,3.0) -- (26,3.0);
 \node at (26.5,3.0) [] {$\scriptstyle{x_1x_2x_3}$};
 \node at (1,4.0) [] {\small $1$};
 \node at (4,4.0) [] {\small $x_1$};
 \node at (7,4.0) [] {\small $x_1^{2}$};
 \node at (10,4.0) [] {\small $x_2$};
 \node at (13,4.0) [] {\small $x_1x_2$};
 \node at (16,4.0) [] {\small $x_2^{2}$};
 \node at (19,4.0) [] {\small $x_3$};
 \node at (22,4.0) [] {\small $x_1x_3$};
 \node at (25,4.0) [] {\small $x_2x_3$};
\end{tikzpicture}
\end{center}
This monomial ideal is strongly stable, but
\begin{center}
$l_1(\cB_1^{(2)})=3$, $l_1(\cB^{(2)})=2$, $l_1(\cB_3^{(2)})=1$, 
$l_1(\cB_4^{(2)})=2$ 
and $l_1(\cB_5^{(2)})=1$,
\end{center}
 so in this case the $1$-lengths are unordered.
\end{example}
\smallskip

The proposition below gives a way to count zerodimensional 
stable ideals in two variables, once known their affine Hilbert polynomial.
 
 \begin{Proposition}\label{CorrispPartInt}
The number of Bar Codes $\cB \subset \kB_2$
with bar list $(p,h)$ and 
such that  $\eta(B)=\cN \subset \mathbf{k}[x_1,x_2]$ 
is the  Groebner escalier of a  
stable ideal $J \triangleleft \mathbf{k}[x_1,x_2]$ 
equals the number of
integer partitions of $p$ into $h$ distinct parts.
\end{Proposition}
\begin{proof}
Consider the set 
$$\kB_{(p,h)}:=\{ \cB \in \kA_2, \textrm{ s.t. } 
\cL_\cB =(p,h) \textrm{ and }  \eta(\cB)=\cN(J),\, J \textrm{  
stable}\}$$
and the set of integer partitions of $p$ into $h$ distinct parts, i.e.
% % $$I_{(p,h)}=\{(\alpha_1,...,\alpha_h)\in \NN^h,\, 
% % \alpha_1>...>\alpha_h  \textrm{ and }
% % \sum_{j=1}^h \alpha_j=p \}.$$

$$I_{(p,h)}=\left\{(\alpha_1,...,\alpha_h)\in \NN^h,\,
\alpha_1>...>\alpha_h  \textrm{ and }
\sum_{j=1}^h \alpha_j=p \right\}.$$

We define  
$$\Xi:\kB_{(p,h)} \longrightarrow \NN^h$$
$${\sf B} \mapsto (l_1(\cB_1^{(2)}),...,l_1(\cB_h^{(2)}))$$
and we prove that $\Xi$ defines a biunivocal correspondence 
between $\kB_{(p,h)}$ and $I_{(p,h)} \subset \NN^h$.

Let $\cB \in \kB_{p,h}$. We have $\eta(\cB)=\cN(J),\, J\triangleleft 
\ck[x_1,x_2]$  
stable. \\
For each $1 \leq j \leq h$ set $\alpha_j= l_1(\cB_j^{(2)})$.
By Proposition \ref{barredecr} a), we have $\alpha_1>...>\alpha_h$ and
by definition of Bar Code (see Definition \ref{BCdef1})
 $p=\sum_{i=1}^{p} l_1(\cB_{i}^{(1)})=\sum_{j=1}^{h} 
l_1(\cB_{j}^{(2)})=\sum_{j=1}^{h} \alpha_{j}$, 
so we can desume that
 $(l_1(\cB_1^{(2)}),...,l_1(\cB_h^{(2)}))= 
 (\alpha_1,...,\alpha_h) \in 
I_{(p,h)}$, so $\Xi(\kB_{(p,h)})\subseteq I_{(p,h)}$.
The map is injective by definition of $1$-length of a bar.
\\ 
Now, let us consider $(\alpha_1,...,\alpha_h)\in I_{(p,h)}$ and 
construct a Bar Code $\cB\subset \kB_2$ with $h$ $2$-bars 
$\cB^{(2)}_1,...,\cB^{(2)}_h$ and s.t.
for each $1 \leq j \leq h$ there are $\alpha_j$ $1$-bars lying over 
$\cB^{(2)}_j$.
 \begin{center}
\begin{tikzpicture}[scale=0.4]
\node at (-0.5,1.5) [] {${\scriptscriptstyle 2}$};
\node at (-0.5,3) [] {${\scriptscriptstyle 1}$};

   \node at (2.5, 2) [] {\small $ \cB^{(2)}_1$};
 \draw [thick] (0,1.5) -- (4.9,1.5);
 \draw [dotted] (4.9,1.5) -- (7.9,1.5);
  \draw [dotted] (7.9,1.5) -- (8.9,1.5);
     \node at (9, 2) [] {\small $ \cB^{(2)}_h$};

  \draw [thick] (6.9,1.5) -- (10.9,1.5);
 
 \draw [thick] (0,3.0) -- (1.9,3.0);
  \draw [dotted] (1.9,3.0) -- (3,3.0);
%   \node at (2.5,4.0) [] {$\ldots$};
 \draw [thick] (3,3.0) -- (4.9,3.0);
 
 \draw [dotted] (4.9,3.0) -- (9,3.0);
 
 \draw [thick] (9,3.0) -- (10.9,3.0);
 
 \node at (1,4.0) [] {\small $\cB^{(1)}_1$};
 \node at (4,4.0) [] {\small $\cB^{(1)}_{\alpha_1}$};
 \node at (10,4.0) [] {\small $\cB^{(1)}_{\alpha_h}$};
\end{tikzpicture}
\end{center}
 
Clearly:
\begin{itemize}
 \item $\cB$ is univocally determined by $(\alpha_1,...,\alpha_h)\in I_{(p,h)}$
 \item for each $1 \leq j \leq h$,  $l_1(\cB^{(2)}_j)=\alpha_j$.
\end{itemize}

We prove that $\cB \in \kA_2$, i.e. that $\cB$ is admissible. Let $\cB_i^{(1)}$ be a $1$-bar, 
$1 \leq i \leq p$ and let $e(\cB_i^{(1)})=(b_{i,1},b_{i,2})$ be its 
e-list. If $b_{i,1}=b_{i,2}=0$ there is nothing to prove. If 
$b_{i,1}>0$ trivially 
there is a $1$-bar with e-list $(b_{i,1}-1,b_{i,2})$; if 
$b_{i,2}>0$, the assumption $\alpha_1>...>\alpha_h$
 proves that there is a $1$-bar with e-list $(b_{i,1},b_{i,2}-1)$.

Finally, we prove that the order ideal $\cN=\eta(\cB)$ is the
 Groebner escalier $\cN=\cN(J)$ of a  stable ideal $J$.\\
Let us take $\sigma \in \kF(J)$; it can be constructed
 from a) or b) of Definition \ref{StarSet}:
 
 \begin{itemize}
  \item  If $\sigma$ comes from a), $\sigma =x_iP_{x_i}(\tau_i)$, $i=1,2$. For $i=2$, there is nothing to prove.
 \\
 We prove  then  the case $i=1$,
%  first we notice that we can take, without loss of generality, $\tau_1=\tau_2=\tau$, the term labelling $\cB_{\mu(1)}^{(1)}$, 
 so we write $\sigma = x_1P_{x_1}(\tau_1)$, where $\tau_1$ labels  $\cB_{\mu(1)}^{(1)}$, and we prove that $\frac{\sigma x_2}{x_1}=x_2 P_{x_1}(\tau_1)$ belongs to $J$.
 \\
 Since $P_{x_2}(\tau_1)\mid P_{x_1}(\tau_1) $, $x_2P_{x_2}(\tau_1)\mid x_2P_{x_1}(\tau_1) $.
  Now, $\tau_1$ labels a $1$-bar over $\cB_{\mu(2)}^{(2)}$, so $x_2P_{x_2}(\tau_1) \in \kF(J)$ and so we are done.
   \item Suppose now $\sigma$ coming from b), so  $\sigma=x_1P_{x_1}(\tau^{(1)}_j)$, where  $\tau^{(1)}_j$
   is the term labelling a bar $\cB_{j}^{(1)}$, $1 \leq j \leq \mu(1)-1$,  and 
    $\cB_{j}^{(1)}$ and $\cB_{j+1}^{(1)}$ are two consecutive $1$-bars not lying over the same $2$-bar; in particular, we say that $\cB_{j}^{(1)}$ lies over $\cB^{(2)}_{j_1}$ and $\cB_{j+1}^{(1)}$ lies over $\cB^{(2)}_{j_1+1}$.
   \\
  We have to prove that $x_2P_{x_1}(\tau^{(1)}_j)$ belongs to $J$.
   \\
  Denoted $\tau^{(1)}_{\overline{j}}$ the term labelling the rightmost $1$-bar over  $\cB^{(2)}_{j_1+1}$, we have
  $\deg_2(\tau^{(1)}_{\overline{j}})=\deg_2(\tau^{(1)}_j)+1$ and 
  $\deg_1(\tau^{(1)}_{\overline{j}})<   \deg_1(\tau^{(1)}_j)  $,
   so 
   $\deg_1(x_1P_{x_1}(\tau^{(1)}_{\overline{j}}) ) \leq \deg_1(x_2P_{x_1}(\tau^{(1)}_j))$
  and 
  $\deg_2(x_1P_{x_1}(\tau^{(1)}_{\overline{j}})) = \deg_2(x_2P_{x_1}(\tau^{(1)}_j))$, whence 
  $x_1P_{x_1}(\tau^{(1)}_{\overline{j}}) \mid x_2P_{x_1}(\tau^{(1)}_j)$ and since $x_1P_{x_1}(\tau^{(1)}_{\overline{j}}) \in J$ we are done.
 \end{itemize}

% % %  \\
% % %  Let $\sigma =x_1^{\beta_1}x_2^{\beta_2} \in \cN(J)$.
% % %  It labels the $1$-bar $\cB^{(1)}_{\overline{j}}$, lying over
% % %  $\cB^{(2)}_j$, 
% % %  for some $\overline{j} \in \{1,...,p\}$, $j \in \{1,...,h\}$.\\
% % % If $\beta_2=0$ there is 
% % % nothing to prove.
% % % If $\beta_2 >0 $, we prove that $\sigma'=\frac{\sigma 
% % % x_1}{x_2}=x_1^{\beta_1+1}x_2^{\beta_2-1} \in \cN(J)$.
% % % We clearly have $j>1$, so we can consider the $2$-bar 
% % % $\cB^{(2)}_{j-1}$ and the $1$-bars 
% % % $\cB^{(1)}_{j_1},...,\cB^{(1)}_{j_1+\alpha_{j-1}-1}$ lying 
% % % over it. 
% % % 
% % % Since $0 \leq \beta_1 \leq \alpha_j-1$, it holds 
% % % $1 \leq \beta_1+1 \leq \alpha_j < \alpha_{j-1}$ so there is a $1$-bar 
% % % s.t. the corresponding term is exactly $\sigma'$ (see Remark 
% % % \ref{ElistExp}).

\end{proof}
% % % %  
% % % % In this section, we generalize what done so far for strongly stable ideals to 
% % % % the case of stable ideals (see Definition \ref{Stab}).
% % % % 
% % % % The following Lemma is enough to deal with the case of two variables.
% % % % 
% % % % \begin{Lemma}\label{StabEqStrStab}
% % % % An ideal $I \triangleleft \ck[x_1,x_2]$ is stable if and only if it is strongly 
% % % % stable.
% % % % \end{Lemma}
% % % % \begin{proof}
% % % %  A strongly stable ideal is trivially stable, so we only need to
% % % %  prove the  converse, namely, given a stable ideal  $I$,  we have to show that for each  
% % % % for every term $\tau\in I$ and pair of variables
% % % % $x_i,\ x_j$ such that $x_i\vert \tau$ and $x_i<x_j$,
% % % % then also $
% % % % \frac{\tau x_j}{x_i} $ belongs to $I$.
% % % % The only pair of variables of the above type is $x_1<x_2$ and $x_1$ is the 
% % % % smallest variable in the polynomial ring $\ck[x_1,x_2]$ so, if
% % % % $x_1 \mid \tau \in I$, then $x_1 =\min(\tau)$ and $\frac{\tau x_2}{x_1}\in I$ by 
% % % % definition of stable ideal, whereas if $x_1 \nmid \tau$ there is nothing to do.
% % % % This proves the claimed equivalence.
% % % % \end{proof}
% % % % 
% % % % By the above Lemma and by Proposition \ref{NumBorel2V}, we can conclude that 
% % % % the number of  stable ideals $J\triangleleft\ck[x_1,x_2]$ with $H_{\_}(t,J)=p$ is $\sum_{i=1}^h 
% % % % Q(p,i),$ where $h=\lfloor \frac{-1+\sqrt{1+8p}}{2} \rfloor$ and $Q(p,i)$ is the number of integer partitions of $p$ into $i$ distinct parts.

With the Proposition below, we prove which is the maximal value that $h$ can assume.

\begin{Proposition}\label{MaxH}
 Denoting by $\cB$ a Bar Code associated to a  stable ideal $I\triangleleft \ck[x_1,x_2]$ with affine Hilbert polynomial $H_I(d)=p \in \NN$ and by $\cL_\cB=(p,h)$ its bar list, the maximal value that $h$ can assume is 
% %  $$h := \lfloor \frac{-1+\sqrt{1+8p}}{2} \rfloor$$
% %  
 $$h := \left\lfloor \frac{-1+\sqrt{1+8p}}{2} \right\rfloor$$
\end{Proposition}
\begin{proof}
 By Proposition \ref{CorrispPartInt}, the Bar Codes associated to stable ideals 
 s.t. the associated bar list is $(p,i)$ are in bijection with the integer partitions of $p$ with $i$ distinct parts.
 \\
 An integer partition of $p$ with $i$ distinct parts is a partition $ (\alpha_1,...,\alpha_i)\in \NN^i,\, 
\alpha_1>...>\alpha_i, \sum_{j=1}^i \alpha_j=p $. Since the minimal value we can give to $\alpha_j, 1\leq j \leq i$, so that $\alpha_1>...>\alpha_i$, is 
 $\alpha_j=i-j+1$ and $\sum_{j=1}^i (i-j+1)= \frac{i(i+1)}{2}$, we have that $\frac{i(i+1)}{2}$ is the minimal sum of $i$ positive distinct 
 integer numbers. If $\frac{i(i+1)}{2} >p$, there cannot exist any partition of $p$ with $i$ distinct parts; if $\frac{i(i+1)}{2}=p$, the $i$-tuple
  $ (\alpha_1,...,\alpha_i)\in \NN^i$ is such a partition and if  $\frac{i(i+1)}{2}\leq p$, it is possible to find a  partition of $p$ with $i$ distinct parts starting from   $ (\alpha_1,...,\alpha_i)\in \NN^i$, for example by increasing the value of $\alpha_1$, until $\sum_{j=1}^i \alpha_j=p $.
\\
Then, we have proved that the maximal number $h$ of distinct parts in a partition of $p$ is
 $h := \max_{i \in \NN}\Big\{\frac{i(i+1)}{2} \leq p \Big\}$.  Since $\frac{i(i+1)}{2} \leq p$ for $ \frac{-1-\sqrt{1+8p}}{2} \leq i \leq  \frac{-1+\sqrt{1+8p}}{2}$, then 
%  $$h=\lfloor \frac{-1+\sqrt{1+8p}}{2} \rfloor .$$

$$h := \left\lfloor \frac{-1+\sqrt{1+8p}}{2} \right\rfloor$$

\end{proof}

 \begin{example}\label{BListMonId}
 Applying proposition \ref{MaxH}, we get that for $p=1,2$, we have $h=1$, so the only (strongly) stable
  monomial ideals of $\ck[x_1,x_2],$ with constant affine Hilbert polynomial $p=1,2$ are the ideals  
  $I_1=(x_1,x_2)$ and $I_2=(x_1^2,x_2)$ (see Remark \ref{SegmLex}).
\\
 For the affine Hilbert polynomial $p=3$
we have $h=2$, so we have two (strongly) stable monomial ideals, $J_1=(x_1^3,x_2)$ and $J_2=(x_1^2,x_1x_2,x_2^2)$.	\\
The Bar Code $\cB_1$ associated to $J_1$ is
\begin{center}
\begin{tikzpicture}

\node at (3.8,-0.5) [] {${\scriptscriptstyle 1}$};
\node at (3.8,-1) [] {${\scriptscriptstyle 2}$};

\node at (4.2,0) [] {${\small 1}$};
\draw [thick] (4,-0.5) --(4.5,-0.5);

\node at (5.2,0) [] {${\small x_1}$};
\draw [thick] (5,-0.5) --(5.5,-0.5);

\node at (6.2,0) [] {${\small x_1^2}$};
\draw [thick] (6,-0.5) --(6.5,-0.5);

\draw [thick] (4,-1)--(6.5,-1);

\node at (6.8,-0.5) [] {${\scriptscriptstyle
x_1^3}$};

\node at (6.8,-1) [] {${\scriptscriptstyle x_2}$};
\end{tikzpicture}
\end{center}
whose bar list is $\cL_{\cB_1}=(3,1)$.
\\
The Bar Code associated $\cB_2$ to $J_2$ is
\begin{center}
\begin{tikzpicture}

\node at (3.8,-0.5) [] {${\scriptscriptstyle 1}$};
\node at (3.8,-1) [] {${\scriptscriptstyle 2}$};

\node at (4.2,0) [] {${\small 1}$};
\draw [thick] (4,-0.5) --(4.5,-0.5);

\node at (5.2,0) [] {${\small x_1}$};
\draw [thick] (5,-0.5) --(5.5,-0.5);

\node at (6.2,0) [] {${\small x_2}$};
\draw [thick] (6,-0.5) --(6.5,-0.5);

\draw [thick] (4,-1)--(5.5,-1);
\draw [thick] (6,-1)--(6.5,-1);

\node at (5.8,-0.5) [] {${\scriptscriptstyle
x_1^2}$};
\node at (6.8,-0.5) [] {${\scriptscriptstyle
x_1x_2}$};

\node at (6.8,-1) [] {${\scriptscriptstyle
x_2^2}$};
\end{tikzpicture}
\end{center}
and its bar list is $\cL_{\cB_2}=(3,2)$.
\end{example}
\medskip

In order to deal with  stable ideals $J \triangleleft
\mathbf{k}[x_1,...,x_n]$ for $n>2$, the following corollary will be rather
useful.
\begin{Corollary}\label{ph11}
The number of Bar Codes associated to  stable ideals in $\mathbf{k}[x_1,...,x_n]$,
$n>2$, whose bar list is $(p,h,\underbrace{1,...,1}_{3,...,n})$, $p,h \in \NN$, $p \geq h$ equals the
number of integer partitions of $p$ in $h$ distinct parts, namely
$$p=\alpha_1+...+\alpha_h,\, \alpha_1>...>\alpha_h>0.$$
Moreover, the maximal value that $h$ can assume in the  bar list $(p,h,1,...,1)$  is
$$h := \left\lfloor \frac{-1+\sqrt{1+8p}}{2} \right\rfloor .$$

% %  $$h = \lfloor \frac{-1+\sqrt{1+8p}}{2} \rfloor.$$
\end{Corollary}
\begin{proof}
It is a straightforward consequence of Propositions \ref{CorrispPartInt} and 
\ref{MaxH}, noticing that, if $\mu(3)=...=\mu(n)=1$, 
$x_3,...,x_n$ do not appear in any term of $M_{\sf B}$ with nonzero exponent.
\end{proof}
 
 The following proposition is a consequence of \ref{CorrispPartInt} and \ref{MaxH} and completely solves the problem of counting 
  stable monomial ideals in two variables.
\begin{Proposition}\label{NumBorel2V}
The number of   stable ideals $J\triangleleft\ck[x_1,x_2]$ with $H_{\_}(t,J)=p$ is $$\sum_{i=1}^h 
Q(p,i),$$
where $h := \left\lfloor \frac{-1+\sqrt{1+8p}}{2} \right\rfloor$
% % 
% % $h=\lfloor \frac{-1+\sqrt{1+8p}}{2} \rfloor$ 
and $Q(p,i)$ is the number of integer partitions of $p$ into $i$ distinct parts.
\end{Proposition}

\begin{Remark}\label{SegmLex}
Let $I\triangleleft \mathbf{k}[x_1,x_2]$ be a strongly stable monomial ideal with affine Hilbert polynomial
 $H_{I}(t)=p$, $\cB$ be the corresponding Bar Code and suppose that $\cL_\cB=(p,1)$. In this case, we can easily deduce that 
 $I=(x_1^p,x_2)$ so $I$ is a \emph{lex-segment ideal}, i.e., for each
degree $i\in \NN$, $I$ is $\ck$-spanned by the first
$H_{I}(i)$ terms w.r.t. Lex.
\end{Remark}

By Remark \ref{SegmLex}, for each $p\in \NN$, there exists a (strongly) stable monomial ideal $I\triangleleft \mathbf{k}[x_1,x_2]$ with affine Hilbert polynomial  $H_{I}(t)=p$ and s.t. the  corresponding Bar Code $\cB$ has $\cL_\cB=(p,1)$, so the minimal value that $h$ can assume is $1$.

We summarize in the following table the possible bar lists
for  stable ideals corresponding to some small values of $p$, together with the corresponding ideals.
\begin{center}
 $\begin{array}{|c|c|c|}
\hline
H_{\_}(t)=p&\textrm{Bar lists}&\textrm{Ideals} \cr
\hline
 1&(1,1)&(x_1,x_2) \cr
\hline
 2&(2,1)& (x_1^2,x_2) \cr
\hline
3&(3,1),(3,2)&
(x_1^3,x_2),(x_1^2,x_1x_2,x_2^2)\cr
\hline
4&(4,1),(4,2)&
(x_1^4,x_2),(x_1^3,x_1x_2,x_2^2)\cr
\hline
5&(5,1),(5,2),(5,2)& 
(x_1^5,x_2),(x_1^4,x_1x_2,x_2^2),(x_1^3,x_1^2x_2,
x_2^2) \cr
\hline
6&(6,1),(6,2),(6,2),(6,3)& 
(x_1^6,x_2),(x_1^5,x_1x_2,x_2^2),(x_1^4,x_1^2x_2,
x_2^2),(x_1^3, 
x_1^2x_2,x_1x_2^2,x_2) \cr
\hline
\end{array}$
\end{center}
We notice that the above ideals are also strongly stable.
% \end{example}

\bigskip

% % % In order to deal with strongly stable ideals $J \triangleleft
% % % \mathbf{k}[x_1,...,x_n]$ for $n>2$, the following corollary will be rather
% % % useful.
% % % \begin{Corollary}\label{ph11}
% % % The number of Bar Codes associated to strongly stable ideals in $\mathbf{k}[x_1,...,x_n]$,
% % % $n>2$, whose bar list is $(p,h,\underbrace{1,...,1}_{2,...,n})$, $p,h \in \NN$, $p \geq h$ equals the
% % % number of integer partitions of $p$ in $h$ distinct parts, namely
% % % $$p=\alpha_1+...+\alpha_h,\, \alpha_1>...>\alpha_h>0.$$
% % % Moreover, the maximal value that $h$ can assume in the  bar list $(p,h,1,...,1)$  is
% % %  $$h = \lfloor \frac{-1+\sqrt{1+8p}}{2} \rfloor.$$
% % % \end{Corollary}
% % % \begin{proof}
% % % It is a straightforward consequence of propositions \ref{CorrispPartInt} and 
% % % \ref{MaxH}, noticing that, if $\mu(3)=...=\mu(n)=1$, 
% % % $x_3,...,x_n$ do not appear in any term of $M_{\sf B}$ with nonzero exponent.
% % % \end{proof}
% % %  
% % %  The following proposition is a consequence of \ref{CorrispPartInt} and \ref{MaxH} and completely solves the problem of counting 
% % %  strongly stable monomial ideals in two variables.
% % % \begin{Proposition}\label{NumBorel2V}
% % % The number of  strongly stable ideals $J\triangleleft\ck[x_1,x_2]$ with $H_{\_}(t,J)=p$ is $$\sum_{i=1}^h 
% % % Q(p,i),$$
% % % where $h=\lfloor \frac{-1+\sqrt{1+8p}}{2} \rfloor$ and $Q(p,i)$ is the number of integer partitions of $p$ into $i$ distinct parts.
% % % \end{Proposition}
\begin{example}\label{Spiegato}
For the polynomial ring $\mathbf{k}[x_1,x_2]$,
consider $H_{\_}(t)=p=10$.\\
In this case, we have $h=4$, so we have to compute
the sum $$Q(10,1)+Q(10,2)+Q(10,3)+Q(10,4).$$ We
have:
\\
$Q(10,1)=1$;\\
$Q(10,2)=P(9,2)=P(8,1)+P(7,2)=1+P(7,2)=1+P(6,
1)+P(5,2)=2+P(5,2)=2+P(4,1)+P(3,2)=3+P(2,1)=4$\\
$Q(10,3)=P(7,3)=P(6,2)+P(4,3)=1+P(4,2)+P(3,
2)=1+P(3,1)+P(2,2)+P(2,1)=1+1+1+1=4$\\
$Q(10,4)=P(4,4)=1$.\\
Then, we have exactly $10$ strongly stable
monomial ideals with  
$H_{\_}(t)=10$.
\\
More precisely, they are:
\begin{itemize}
\item[$\star$] $J_1 = (x_1^{10},x_2)$;
\item[$\star$] $J_2 = (x_1^9,x_1x_2,x_2^2)$;
\item[$\star$] $J_3 = (x_1^8,x_1^2x_2,x_2^2)$;
\item[$\star$] $J_4 = (x_1^7,x_1^3x_2,x_2^2)$;
\item[$\star$] $J_5 = (x_1^7,x_1x_2^2,x_2x_1^2,x_2^3)$;
\item[$\star$] $J_6 = (x_1^6,x_1^4x_2,x_2^2)$;
\item[$\star$] $ J_7 = (x_1^6,x_1x_2^2,x_1^3x_2,x_2^3)$;
\item[$\star$] $ J_8 = (x_1^5,x_2^2x_1,x_2x_1^4,x_2^3)$;
\item[$\star$] $ J_9 = (x_1^5,x_2^2x_1^2,x_2x_1^3,x_2^3)$;
\item[$\star$] $ J_{10} = (x_1^4,x_2^3x_1,x_2^2x_1^2,x_2x_1^3,x_2^4)$.
\end{itemize}
\end{example}
\medskip
\begin{example}\label{Numerone}
Employing the same formula (all the computation has been performed using Singular \cite{DGPS}), we can get that the strongly stable monomial ideals with  $H_{\_}(t)=100$ are exactly $444793$.
\end{example}
\bigskip

Now we start studying the case of three variables; in this case we need to consider the bar lists of the form $(p,h,k)$.  By Corollary \ref{ph11}, we can use the formulas
for two variables in order to 
count the stable monomial ideals  in three variables, associated  to bar lists of the form 
$(p,h,1)$. This means that we only have to deal with the bar lists of the form 
$(p,h,k), $ such that $k>1$.\\
In order to handle these new bar lists, we
define the concept of \emph{minimal sum} of a
list of positive integers.
\begin{definition}\label{MinSumPartNoEq}
The \emph{minimal sum} of a given list of 
positive integers $[\alpha_1,...,\alpha_g]$ is the
integer
$$\Sm([\alpha_1,...,\alpha_g]):=\sum_{i=1}^g \frac{\alpha_i(\alpha_i +1)}{2}.$$
\end{definition}

 \begin{Lemma}\label{minhmaxh}
With the previous notation, it holds:
\begin{enumerate}
\item $k \in \{1,...,l\},$ where $l:=\max_{i \in \NN}\{i^3+3i^2+2i\leq 6p\};$
% \textbf{ricontrolla conti curva:}
% $i^3+3i^2+2i\leq 6p$
% $i^3+3i^2+2i=i(i+1)(i+2)$ da traslare di $6p$.
% Derivo per Cercare max $3i^2+6i+2$ e risolvo cercando soluzione negativa $=-(i+\alpha)$ t.c. $\alpha^2=1/3$
% Sostituisco in $i^3+3i^2+2i-6p<0$ e trovo $p>\alpha/9$.
\item $h \in \{\frac{k(k+1)}{2},...,m\}$, where  $\displaystyle m=\max_{ r\geq  \frac{k(k+1)}{2}}\{r \,\vert\, \exists
\lambda \in I_{(r,k)},\Sm(\lambda )\leq p\}.$
\end{enumerate}
\end{Lemma}
\begin{proof}
 By Corollary \ref{ph11} the minimal value for $k$ is $1$.
 \\ 
 Now, in order to construct a Bar Code $\cB$ associated to a stable ideal, we should at least meet the requirements of  Proposition  \ref{barredecr}, so, given $k$, for each  $3$-bar $\cB^{(3)}_j$ 
 there should be at least $(k-j+1)$ $2$-bars lying over it, so that $h \geq \frac{k(k+1)}{2}$.
  \\
  Now, select a $3$-bar $\cB^{(3)}_{\overline{j}}$,
  $1 \leq \overline{j} \leq k$ and let  $\cB^{(2)}_{j_1},...,\cB^{(2)}_{j_1 +t-1}$, $t \geq k-\overline{j}$ 
  %$0 \leq t \leq k-\overline{j}+1$ 
  be the $2$-bars over  $\cB^{(3)}_{\overline{j}}$. 
  Now, with an analogous argument w.r.t. the one for $2$-bars, 
  we can say that for $\cB^{(2)}_{j_1 +j-1}, 1 \leq j \leq t$,
  we must have at least  $t-j+1$  $1$-bars, so that their total number will be 
  $\Sm([1,2,...,k])=\sum_{i=1}^k \frac{i(i+1)}{2}$. Since the number of elements in $\eta(\cB)$ equals the Hilbert polynomial $p$, we must have  $\Sm([1,2,...,k])=\sum_{i=1}^k \frac{i(i+1)}{2}\leq p$.
  \\
  Now  $\sum_{i=1}^k \frac{i(i+1)}{2}= \sum_{i=1}^k {i+1 \choose 2} = {k+2 \choose 3}\leq p$,
   so $k^3+3k^2+2k\leq 6p$ and we are done.
%   $\sum_{i=1}^k \frac{i(i+1)}{2}= \frac{1}{2}\sum_{i=1}^k (i^2+i)=\frac{1}{2}(\sum_{i=1}^k i^2+\sum_{i=1}^k i)=\frac{1}{2} \Big(\frac{k(k+1)(2k+1)}{6}\Big)+ \frac{1}{2} \Big(\frac{k(k+1)}{2}\Big)\leq p$.
%  Then $\frac{2k^3+6k^2+4k}{6}\leq 2p$ so $k^3+3k^2+2k\leq 6p$.
 \\
 As regards the maximal value that $h$ can assume, from anologous arguments, to meet the requirements of  Proposition  \ref{barredecr}, it is enough to be able to find a partition $\lambda\in I_{(h,k)}$ with $\Sm(\lambda )\leq p$. 
\end{proof}
 
Thanks to the previous Lemma \ref{minhmaxh}, now we know which are the bar lists 
we have to take into account in order to count the  stable ideals with affine Hilbert polynomial
$H_{\_}(t)=p$. \\
Next step then, is to find out how many   stable ideals with 
$H_{\_}(t)=p$ and such that their Bar Code $\cB$ has bar list $(p,h,k)$ are there.\\
% % % % % Fixed a bar list $(p,h,k)$, we first compute the integer partitions of $h$ in $k$ distinct parts.
% % % % % Each partition $(\alpha_1,...,\alpha_k) \in \NN^k$, $\alpha_1>...>\alpha_k$, $\sum_{i=1}^k \alpha_i =h$ represents
% % % % % a precise structure for the $2$-bars and the $3$-bars: for each $1 \leq i \leq k$ there are exactly $\alpha_i$
% % % % %  $2$-bars over $\cB^{(3)}_i$. 
% % % % % \\
% % % % % 
% % % % % 
% % % % % \medskip
% % % % % 
% % % % % 
% % % % % Let us examine now the case of stable ideals in $\ck[x_1,x_2,x_3]$.\\
% % % % % From Lemma \ref{StabEqStrStab}, we can desume that Corollary \ref{ph11}
% % % % %  holds true for stable ideals. Moreover, Lemma \ref{minhmaxh} only relies on 
% % % % % Proposition \ref{barredecr}, which holds also for stable ideals, so the 
% % % % % computation of the bar lists is the same as done for strongly stable ideals.

Take then a bar list $(p,h,k)$ and let $\overline{\beta} \in I_{(h,k)}$, so 
$\overline{\beta_1}>...>\overline{\beta_k}$ and $\sum_{i=1}^k \overline{\beta_i} = h$.  
\\
We can construct plane partitions $\rho$ of the form

$$\rho=(\rho_{i,j})=\left( \begin{matrix} 
\rho_{1,1}& \rho_{1,2}&...& ...& ...& ...& ... & ... & \rho_{1,\overline{\beta_1}}\\
\rho_{2,1}& ...& ...& ...& ...& ...& \rho_{2,\overline{\beta_2}} & 0  &...  \\
...&...  & ...& ...& ...&...& ...& ... &... \\
\rho_{k,1}& ...&... & ... & ...& \rho_{k, \overline{\beta_k}} & 0 & ...  &...  
\\
\end{matrix} \right)$$
s.t. 
\begin{enumerate}
\item $\rho_{i,j}>0$, $1 \leq i \leq k$, $1\leq j \leq \overline{\beta_i}$; 
\item  $\rho_{i,j}>\rho_{i,j+1}$, $1 \leq i \leq k$, $1\leq j \leq \overline{\beta_i}-1$; 
\item $\rho_{i,j}> \rho_{i+1,j}$  $1 \leq i \leq k-1$,  $ 1\leq j \leq 
\overline{\beta_{i+1}}$;
\item $n(\rho)=\sum_{i=1}^{k}\sum_{j=1}^{\overline{\beta_i} }\rho_{i,j}=p$.
\end{enumerate}
%%%%%%%%%%%%%%%%%%%%%%%%%%%%%%%%%%%%%%%%%
These plane partitions are exactly of the form defined in \ref{PlanePartNoShift}, with shape $\overline{\beta}$, $c=1$ and $d =1$,
 so they are row-strict and column-strict plane partitions of shape $\overline{\beta}$.
\\
% % In Remark  \ref{ContenSStab}, we will highlight the 
% % relation between these partitions and the ones defined in the previous
% % section \ref{StStCount}.
% % \\
%%%%%%%%%%%%%%%%%%%%%%%%%%%%%%%%%%%%%%%%%
Fixed $\overline{\beta} \in I_{(h,k)}$, we denote by $\kP_{(p,h,k), \overline{\beta}}$ the set of all partitions defined as above 
and $\kP_{(p,h,k)}=\bigcup_{\overline{\beta} \in I_{(h,k)}}\kP_{(p,h,k), \overline{\beta}}$. In 
other words, 
$$\kP_{(p,h,k), \overline{\beta}}=\{\rho \in \kP_{\overline{\beta}}(1,1) 
\textrm{ s.t } n(\rho)=p\}$$

$$\kP_{(p,h,k)}=\{\rho \in \kP_{\overline{\beta}}(1,1)  \textrm{ for some } \overline{\beta} \in 
I_{(h,k)} \textrm{ and s.t. }   n(\rho)=p\}.$$ 

Each plane partition $\rho \in \kP_{(p,h,k)}$ uniquely identifies a Bar 
Code $\cB$:

\begin{itemize}
 \item[(a)] each row $i$ represents a $3$-bar $\cB^{(3)}_i$, $1 \leq i \leq k$;
 \item[(b)] for each row $i$, $1 \leq i \leq k$, $l_2(\cB^{(3)}_i)=\overline{\beta_i}$; 
 the  $\overline{\beta_i}$ nonzero entries 
represent the $\overline{\beta_i}$  $2$-bars  over $\cB^{(3)}_i$, i.e the $j$-th entry of row $i$, 
$1 \leq j \leq \overline{\beta_i}$, represents the $2$-bar $\cB^{(2)}_{t}$, 
where $t=(\sum_{l=1}^{i-1}\overline{\beta_l} )+j$;
 \item[(c)] for each  $1 \leq i \leq k$, and  each $1 \leq j \leq \overline{\beta_i}$, 
the 
number $\rho_{i,j}$ represents the number of $1$-bars over  $\cB^{(2)}_t$,  
$t=(\sum_{l=1}^{i-1}\overline{\beta_l} )+j$, the $j$-th $2$-bar lying over 
$\cB^{(3)}_i$. In other words, $l_1(\cB^{(2)}_t)=\rho_{i,j}$.
\end{itemize}
In conclusion, for each  $1 \leq i \leq k$, and  each $1 \leq j \leq 
\overline{\beta_i}$, the number $\rho_{i,j}$ means that in $\cB$ there are $1$-bars 
labelled by 
$(0,j-1,i-1),(1,j-1,i-1),...,(\rho_{i,j}-1,j-1,i-1)$, but there is no $1$-bar 
labelled by $(\rho_{i,j},j-1,i-1)$, that is also equivalent to say that 
$x_1^0x_2^{j-1}x_3^{i-1},x_1x_2^{j-1}x_3^{i-1},...,x_1^{\rho_{i,j}-1}x_2^{j-1}
x_3^{i-1} $ belong to the set of terms associated to $\cB$ via Bbc1 and Bbc2, 
but $x_1^{\rho_{i,j}}x_2^{j-1}
x_3^{i-1} $ does not belong to the aforementioned set\footnote{Actually, we will see that $x_1^{\rho_{i,j}}x_2^{j-1}
x_3^{i-1} $ will belong to the star set associated to the Bar Code $\cB$, after proving that it is admissible.}.
\begin{example}\label{NnonNStab}
Taken the plane partition
$$\rho=\left( \begin{matrix} 
4& 3& 2&1 \\
3& \mathbf{2}& 1&  0 \\
1& 0 & 0& 0\\
\end{matrix} \right).$$

Let us examine the position in bold, i.e. $\rho_{2,2}=2$.

The Bar Code $\cB$  associated to $\rho $ is 

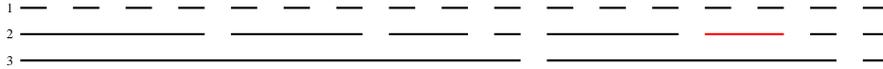
\begin{center}
\begin{tikzpicture}[scale=0.7]
\node at (3.8,-0.5) [] {${\scriptscriptstyle 1}$};
\node at (3.8,-1) [] {${\scriptscriptstyle 2}$};
\node at (3.8,-1.5) [] {${\scriptscriptstyle 3}$};

\draw [thick] (4,-0.5) --(4.5,-0.5);
\draw [thick] (5,-0.5) --(5.5,-0.5);
\draw [thick] (6,-0.5) --(6.5,-0.5);
\draw [thick] (7,-0.5) --(7.5,-0.5);

\draw [thick] (8,-0.5) --(8.5,-0.5);
\draw [thick] (9,-0.5) --(9.5,-0.5);
\draw [thick] (10,-0.5) --(10.5,-0.5);

\draw [thick] (11,-0.5) --(11.5,-0.5);
\draw [thick] (12,-0.5) --(12.5,-0.5);

\draw [thick] (13,-0.5) --(13.5,-0.5);

\draw [thick] (14,-0.5) --(14.5,-0.5);
\draw [thick] (15,-0.5) --(15.5,-0.5);
\draw [thick] (16,-0.5) --(16.5,-0.5);

\draw [ thick] (17,-0.5) --(17.5,-0.5);
\draw [thick] (18,-0.5) --(18.5,-0.5);

\draw [thick] (19,-0.5) --(19.5,-0.5);

\draw [thick] (20,-0.5) --(20.5,-0.5);

\draw [thick] (4,-1)--(7.5,-1);
\draw [thick] (8,-1) --(10.5,-1);
\draw [thick] (11,-1) --(12.5,-1);
\draw [thick] (13,-1) --(13.5,-1);

\draw [thick] (14,-1) --(16.5,-1);
\draw [color=red, thick] (17,-1) --(18.5,-1);
\draw [thick] (19,-1) --(19.5,-1);

\draw [thick] (20,-1) --(20.5,-1);

\draw [thick] (4,-1.5)--(13.5,-1.5);
\draw [thick] (14,-1.5) --(19.5,-1.5);
\draw [thick] (20,-1.5) --(20.5,-1.5);

\end{tikzpicture}
\end{center}

We have $t=\overline{\beta_1}+2=6$, so  $2=\rho_{2,2}=l_1(\cB^{(2)}_6)$ (we have marked
$\cB^{(2)}_6$ in red in the picture).
Applying  Bbc1 and Bbc2 we can see,  absolutely in agreement, with the above comments, that $x_2x_3,x_1x_2x_3$ are in the set 
 of terms associated to $\cB$, whereas 
 $x_1^2x_2x_3$ does not.
\end{example}
\begin{Remark}\label{Lb}
 The Bar Code $\cB$, uniquely identified by $\rho$, has bar list 
$\cL_\cB=(p,h,k)$. The relation $\mu(3)=k$ comes from (a), $\mu(2)=h$ comes 
from (b), since $\beta \in I_{(h,k)}$, so $\sum_{i=1}^k \beta_i=h$, whereas
$\mu(1)=p$ is an easy consequence of (c).
\end{Remark}

\medskip

In the following Lemma, we prove that a Bar Code $\cB$, defined as above, 
is admissible.

\begin{Lemma}\label{admiss2}
Fixed $(p,h,k)$ and $\beta\in I_{(h,k)}$, 
% % $\beta=(\beta_1,...,\beta_k) \in 
% % \NN^k$, $\beta_1>...>\beta_k$, $\sum_{i=1}^k \beta_i=h$, 
let $\rho$ be a 
partition in $ \kP_{(p,h,k), \beta}$. 
\\
The Bar Code $\cB$, uniquely identified by $\rho$, is admissible.
\end{Lemma}
\begin{proof}
By Remark \ref{Lb}, $\cL_\cB=(p,h,k)$, so consider a $1$-bar $\cB^{(1)}_l$, 
$1\leq l 
\leq p$ and  its e-list that we denote
 $e(\cB^{(1)}_l)=(b_{l,1},b_{l,2},b_{l,3})$. 
 From the construction of $\cB$ from $\rho$, we desume 
that   $\rho_{b_{l,3}+1,b_{l,2}+1}\geq b_{l,1}+1$; 
moreover  $(m, 
b_{l,2},b_{l,3})$, $0\leq m\leq \rho_{b_{l,3}+1,b_{l,2}+1}-1$ 
are e-lists for 
some bars of $\cB$, so, if $b_{l,1}\geq 1$,  
$(b_{l,1}-1, b_{l,2},b_{l,3})$ is an
e-list labelling a $1$-bar of $\cB$.
 \\
 For $\cB$ being admissible, we also need two other conditions:
 \begin{itemize}
  \item[a.] if $b_{l,2}>0$, then $(b_{l,1}, b_{l,2}-1,b_{l,3})$ 
  labels a $1$-bar of $\cB$;
  \item[b.] if $b_{l,3}>0$, then $(b_{l,1}, 
b_{l,2},b_{l,3}-1)$   labels a $1$-bar of $\cB$.
 \end{itemize}
Let us prove them: 
 \begin{itemize}
 \item[a.] suppose $b_{l,2}>0$; for $(b_{l,1}, b_{l,2}-1,b_{l,3})$
 labelling a $1$-bar of $\cB$,  we would 
need $\rho_{b_{l_3}+1, 
b_{l_2} }\geq b_{l_1}+1$, but 
 since $\rho_{b_{l_3}+1, b_{l_2} }> 
 \rho_{b_{l_3}+1, b_{l_2}+1}\geq b_{l_1}+1$ we 
are done
  \item[b.] suppose $b_{l,3}>0$; for $(b_{l,1}, b_{l,2},b_{l,3}-1)$  labelling 
a $1$-bar of $\cB$, we would need $\rho_{b_{l_3}, 
b_{l_2}+1}\geq b_{l_1}+1$, but 
 since $\rho_{b_{l_3}, b_{l_2}+1 } >  \rho_{b_{l_3}+1, b_{l_2}+1 }  
 \geq b_{l_1}+1$  we are done again and 
$\cB$ turns out to be admissible.
 \end{itemize}
\end{proof}

\begin{Lemma}\label{FJNoto}
 Let $\rho \in \kP_{(p,h,k)}$ be a strict plane partition and 
 $\cB$ be the Bar Code uniquely determined by 
$\rho$. Denoted by $J$  the monomial ideal s.t. $\eta(\cB)=\cN(J)$
and by $A$ the set $$A:=\{x_3^k, x_2^{\beta_i}x_3^{i-1}, 
x_1^{\rho_{i,j}}x_2^{j-1}x_3^{i-1}, 
\, 1 \leq i \leq k,\, 1\leq j\leq \beta_i \},$$
then $\kF(J)=A$.
\end{Lemma}
\begin{proof}

Let us first prove $\kF(J)\supseteq A$.\\
Neither $x_3^k,$ nor $ x_2^{\beta_i}x_3^{i-1}, $ nor $x_1^{\rho_{i,j}}x_2^{j-1}x_3^{i-1}$
belong to $\cN(J)$ by the definition of $\eta$ and by the construction of $\cB$ from $\rho$.
\\
Consider $x_3^k$; clearly, being $k>0$, $\min(x_3^k)=x_3$, so we prove 
that $x_3^{k-1} \in 
\cN(J)$. Since $k=\mu(3)$, 
there are exactly $k$ $3$-bars. 
By BbC1, the $k$-th $3$-bar of $\cB$ is labelled by 
$l_1(\cB^{(3)}_k)$ copies of 
$x_3^{k-1} $, so the $1$-bars over $\cB^{(3)}_k$ are labelled by 
terms which are multiple of $x_3^{k-1} $. 
The Bar Code $\cB$ is
admissible, then also 
$x_3^{k-1} \in \cN(J)$\footnote{Actually, by  BbC1,  $x_3^{k-1} $ labels the 
first $1$-bar over $\cB^{(3)}_k$.}.
\\
As regards $x_2^{\beta_i}x_3^{i-1}$, $ 1 \leq i \leq k$, $\beta_i>0$, 
whence $\min(x_2^{\beta_i}x_3^{i-1})=x_2$, so we have to prove 
that 
$x_2^{\beta_i-1}x_3^{i-1}\in \cN(J)$. 
\\ 
We take the $i$-th $3$-bar $\cB^{(3)}_i$; it is labelled by $l_1(\cB^{(3)}_i)$ copies of $x_3^{i-1}$.
Now, over $\cB^{(3)}_i$ there are exactly $\beta_i$ $2$-bars and, by  BbC2, the  $\beta_i$-th  $2$-bar  over  $\cB^{(3)}_i$ (i.e. $\cB^{(2)}_t,\, t=\sum_{l=1}^i \beta_i$) is labelled by
$l_1(\cB^{(2)}_t)$ copies of $x_2^{\beta_i-1}x_3^{i-1}$, so the $1$-bars over 
$\cB^{(3)}_i$ are labelled by terms which are multiple of $x_2^{\beta_i-1}x_3^{i-1}$; by the 
admissibility of $\cB$, we get $x_2^{\beta_i-1}x_3^{i-1}\in \cN(J)$\footnote{Actually, by  BbC1,  $x_2^{\beta_i-1}x_3^{i-1}$ labels the first $1$-bar over $\cB^{(2)}_t$.}.
 \\
Take then $x_1^{\rho_{i,j}}x_2^{j-1}x_3^{i-1}$, 
$ 1 \leq i \leq k,\, 1\leq j\leq \beta_i$; since $\rho_{i,j}>0$, 
$\min(x_1^{\rho_{i,j}}x_2^{j-1}x_3^{i-1})=x_1$ and so we 
have to prove 
that $x_1^{\rho_{i,j}-1}x_2^{j-1}x_3^{i-1}\in \cN(J)$, 
but this is trivial by the construction of $\cB$ from
  $\rho.$
\\
\medskip 

We prove now that  $\kF(J)\subseteq  A$.\\
Let $\tau \in \kF(J)$; we have to show that it belongs to 
$A$.\\
If $\min(\tau)=x_3$, then $\tau =x_3^{h_3}$ for some $h_3 \in \NN$; 
we show that necessarily $h_3=k$ and so $\tau=x_3^k \in A$.

By the construction of $\cB$ from $\rho$ we have $\mu(3)=k$,
i.e. $\cB$ has exactly $k$ $3$-bars; by Definition 
 \ref{StarSet} a), with $i=n=3$, $x_3P_{x_3}(\tau_3)\in \kF(J)$, where 
 $\tau_3$ is a term  labelling a $1$-bar over $\cB^{(3)}_k$.
 Now, by BbC1, each $\tau_3 \in \kT$ labelling a $1$-bar over $\cB^{(3)}_k$ 
  is s.t. $P_{x_3}(\tau_3)=x_3^{k-1}$, so 
 $x_3P_{x_3}(\tau_3)=x_3^k\in \kF(J)$. \\
No other pure powers of $x_3$ can 
occur in $\kF(J)$ by Definition \ref{StarSet}, indeed,
$x_3^k$ is the only term with minimal variable $x_3$ 
derived by a) and 
there cannot be terms  derived by b), since each term $\sigma$ coming
 from b) has $\min(\sigma)\leq x_2$.
\\
 We can conclude that the only pure power of $x_3$ in $\kF(J)$ is  
$\tau=x_3^k$, which is also an element of $A$.
\\
\smallskip

\noindent Let now be $\min(\tau)=x_2$, so $\tau = x_2^{h_2}x_3^{h_3}$, for some 
$h_2,h_3 \in \NN.$
This term may be derived either from a) or from b) of Definition \ref{StarSet}; we have to prove that, in 
any case, it belongs to $A$.
\begin{itemize}
 \item[a)] In this case, $\tau=x_2P_{x_2}(\tau_2)$, where $\tau_2$ is a term labelling a $1$-bar over 
 $\cB^{(2)}_{\mu(2)}$. But $\mu(2)=h$; since 
$\cB^{(2)}_{\mu(2)}=\cB^{(2)}_{h}$ is the rightmost $2$-bar, it lies 
over 
 $\cB^{(3)}_k$, where $k=\mu(3)$ and, in particular it is the $\beta_k$-th bar 
over  $\cB^{(3)}_k$.
 Now, by BbC1 and BbC2, we can desume that $h_3=k-1$ and $h_2=\beta_k-1$, so 
$\tau_2 = x_2^{\beta_k-1}x_3^{k-1}$ and so $\tau = x_2^{\beta_k}x_3^{k-1} \in 
A$.
  \item[b)] In this case, for $1 \leq l \leq h-1$, we consider two 
  consecutive $2$-bars 
  $\cB^{(2)}_l,\,\cB^{(2)}_{l+1}$ not lying over the same $3$-bar, 
  i.e. lying 
over two consecutive $3$-bars $\cB^{(3)}_{l_1},\,\cB^{(3)}_{l_1+1}$,
$1 \leq 
l_1 <k$; let $\tau^{(2)}_l$ a term labelling a $1$-bar over $\cB^{(2)}_l$.
%%, s.t. 
%%%$\tau = x_2P_{x_2}(\tau^{(2)}_l)$.
  \\
  Since $\tau^{(2)}_l$ labels a $2$-bar lying over $\cB^{(3)}_{l_1}$, 
  $1 \leq l_1 
<k$, it holds  
  $x_3^{l_1-1}\mid \tau^{(2)}_l$ and $x_3^{l_1}\nmid \tau^{(2)}_l$. 
  \\
  Now, over $\cB^{(3)}_{l_1}$ there are $\beta_{l_1}$ $2$-bars and since 
$\cB^{(2)}_{l+1}$ lies over
   $\cB^{(3)}_{l_1+1}$, then $\cB^{(2)}_{l}$ lies over the
   $\beta_{l_1}$-th 
$2$-bar over 
  $\cB^{(3)}_{l_1}$, so $x_2^{\beta_{l_1}-1} \mid \tau^{(2)}_l $  and 
  $x_2^{\beta_{l_1}} \nmid \tau^{(2)}_l $. 
  This implies that $\tau=x_2P_{x_2}(\tau^{(2)}_l)= 
x_2^{\beta_{l_1}}x_3^{l_1-1} \in A$,   $1 \leq l_1 <k$.  
\end{itemize}
  \medskip
  
\noindent  Finally, let $\min(\tau)=x_1$; as for the above case, we have to 
examine a) and b) separately:
  \begin{itemize}
  \item[a)] In this case, $\tau=x_1P_{x_1}(\tau_1)$, 
  where $\tau_1$ labels  $\cB^{(1)}_{\mu(1)}= \cB^{(1)}_{p}$.
  Now, $\cB^{(1)}_{p}$ is the rightmost $1$-bar, so it lies over 
  $\cB^{(2)}_h$, which,  in turn, lies over $\cB^{(3)}_k$.
  By BbC1 and BbC2, $x_3^{k-1} \mid \tau_1$, $x_3^k \nmid \tau_1$,
  $x_2^{\beta_k-1} \mid \tau_1$, $x_2^{\beta_k} \nmid \tau_1$
  From $l_1(\cB^{(2)}_h)=\rho_{k,\beta_k}$ we desume that  
$\tau=x_1P_{x_1}(\tau_1)=x_1^{\rho_{k,\beta_k}}x_2^{\beta_k-1}x_3^{k-1}\in 
A$.
  \item[b)] In this case, for $1 \leq l_1 \leq 
\mu(1)-1=p-1$ we consider two consecutive $1$-bars 
   $\cB^{(1)}_{l_1}$ and $\cB^{(1)}_{l_1+1}$, 
   lying over two consecutive 
$2$-bars    $\cB^{(2)}_{l_2},\,\cB^{(2)}_{l_2+1} $, $1\leq l_2< h $ and we 
denote  
  $\cB^{(3)}_{l_3}$, $1 \leq l_3 
\leq k$, the $3$-bar underlying\footnote{We remark that 
$\cB^{(2)}_{l_2+1} $ may lie over $\cB^{(3)}_{l_3}$ or - if it exists - to 
its consecutive $2$-bar, but we do not care about it, since it has no 
influence
 on $\tau$. Remember also that, by construction, $l_2=\sum_{r=1}^{l_3-1}\beta_r  + \overline{j}$ with $1 \leq \overline{j} \leq \beta_{l_3}$.} $\cB^{(2)}_{l_2}$.
\\ Let $\tau^{(1)}_{l_1}$ be the term 
labelling $\cB^{(1)}_{l_1}$; by BbC1 and BbC2 $x_3^{l_3-1} \mid \tau^{(1)}_{l_1}$, 
$x_3^{l_3} \nmid 
\tau^{(1)}_{l_1}$,
  $x_2^{u-1} \mid \tau^{(1)}_{l_1}$, $x_2^{u} \nmid \tau^{(1)}_{l_1}$, $u=l_2-\sum_{r=1}^{l_3-1}\beta_r \leq \beta_{l_3}$
  and  $x_1^{\rho_{l_3, u}-1} \mid \tau^{(1)}_{l_1}$, $x_1^{\rho_{l_3, u}} \nmid \tau^{(1)}_{l_1}$,
so
we have $\tau = 
x_1P_{x_1}(\tau^{(1)}_{l_1})=x_1^{\rho_{l_3, 
u}}x_2^{u-1}x_3^{l_3-1} \in A$.
% % % ,  
% % % $1 \leq l_3 \leq k$, $1 \leq u \leq 
% % % \beta_{l_3}$, belonging to our set.
  \end{itemize}
\end{proof}

\begin{Theorem}\label{bijez3varStab}
There is a biunivocal correspondence between $\kP_{(p,h,k)}$ and the set \\
$\cB^{(S)}_{(p,h,k)}=\{\cB \in \kA_3 \textrm{ s.t. }   \cL_\cB=(p,h,k),\, \eta(\cB)=\cN(J),\, J \textrm{  stable}     \}.$
\end{Theorem}
\begin{proof}
Let $\cB \in \cB^{(S)}_{(p,h,k)}$; we construct a plane partition  
$$\rho=(\rho_{i,j})=\left( \begin{matrix} 
\rho_{1,1}& \rho_{1,2}&...& ...& ...& ...& ... & ... & \rho_{1,\beta_1}\\
\rho_{2,1}& ...& ...& ...& ...& ...& \rho_{2,\beta_2} & 0  &...  \\
...&...  & ...& ...& ...&...& ...& ... &... \\
\rho_{k,1}& ...&... & ... & ...& \rho_{k, \beta_k} & 0... & ...  &...  
\\
\end{matrix} \right)$$
with $k$ rows and $l_2(\cB^{(3)}_1)=\beta_1$ columns.

Chosen $1 \leq i\leq k$ as row index and $1\leq j\leq \beta_1$ as column index and set 
$\beta_i = l_2(\cB^{(3)}_i)$, we define

$$\rho_{i,j}=
\left\{
        \begin{array}{l}
          l_1(\cB^{(2)}_ {t}) \quad \textrm{ with } t=(\sum_{l=1}^{i-1}\beta_l)+j, \textrm{ for } 1 \leq i \leq k,\; 1 \leq j\leq \beta_i,\\
           0 \quad \quad \quad \;\textrm{ if } 1 \leq i \leq k,\; \beta_i<j\leq \beta_1,\\
        \end{array}
    \right.
$$
% % % % 
% % % % 
% % % % % % % 
% % % % % % % 
% % % % % % % $\rho_{i,j}=0$ if $1 \leq i \leq k$ ,  $\beta_i<j\leq \beta_1$   and 
% % % % 
% % % % 
% % % % $\rho_{i,j}=l_1(\cB^{(2)}_ {t})$ with $t=(\sum_{l=1}^{i-1}\beta_l)+j$ otherwise
so $\beta$ is the shape of $\rho$.
\\
We notice that the partition $\rho$ is uniquely determined by $\cB$ and that 
 $\beta \in I_{(h,k)}$; indeed $\sum_{i=1}^{k}\beta_i=h=\mu(2)$ and,  by Proposition \ref{barredecr} a), $\beta_1>...>\beta_n$.
\\
% % 
% % \textbf{c'e' un pezzo non copiato: vedere se serve dimostrare davvero Ihk 
% % o se posso lasciarlo in questa forma}
% % 
Now, we prove that $\rho \in \kP_{(p,h,k)}$. 

The nonzero parts of $\rho$ are positive by definition of length of a bar.

Clearly $\rho_{i,j}>\rho_{i,j+1}$, $1 \leq i \leq k$, $1\leq j < \beta_i$, indeed, this can be stated 
as $l_1(\cB^{(2)}_t)>l_1(\cB^{(2)}_{t+1})$,  $t=(\sum_{l=1}^{i-1}\beta_l)+j$, with $\cB^{(2)}_t$ and $\cB^{(2)}_{t+1}$ lying over the same $3$-bar $\cB^{(3)}_{i}$. This statement follows from Proposition \ref{barredecr} b).

Moreover, $\rho_{i,j}>\rho_{i+1,j}$  $1 \leq i \leq k-1$,  $1\leq j \leq \beta_{i+1}$.
\\ 
% % % % % % % % % % \textbf{controllare questa sopra che ho aggiustato a naso; forse meglio unificare come leq}
% % % % % % % % % % Indeed, if $\rho_{i,j}= \rho_{i+1,j}$ then it would happen that
% % % % % % % % % % $\tau_1=x_1^{\rho_{i+1,j}-1}x_2^{j-1}x_3^{i} \in \cN(J)$ and 
% % % % % % % % % % $\tau_2=x_1^{\rho_{i,j}}x_2^{j-1}x_3^{i-1} \in J$,
% % % % % % % % % % but since $\frac{x_3\tau_2}{\min(\tau_2)}= \tau_1 \in \cN(J)$, this contradicts 
% % % % % % % % % %  the stability of  $J$. 
% % % % % % % % % %  
% % % % % % % % % % On the other hand, if $\rho_{i,j}< \rho_{i+1,j}$ then it would happen that 
% % % % % % % % % %  $\tau_1=x_1^{\rho_{i,j}}x_2^{j-1}x_3^{i-1} \in J$ and 
% % % % % % % % % % $\tau_2=x_1^{\rho_{i+1,j}-1}x_2^{j-1}x_3^{i} \in \cN(J)$, but since 
% % % % % % % % % % $\frac{x_3\tau_2}{\min(\tau_2)}\mid \tau_1 \in \cN(J)$,this contradicts 
% % % % % % % % % %  the stability of  $J$. 
% % Indeed, if $\rho_{i,j}\leq \rho_{i+1,j}$  $1 \leq i \leq k-1$,  $1\leq j \leq \beta_{i+1}$, 
% % $\sigma := x_1^{\rho_{i,j}}x_2^{j-1}x_3^{i-1} \in J$; being $\rho_{i,j}>0$, $\min(\sigma)=x_1<x_3$, so $\frac{\sigma x_3}{x_1}=  x_1^{\rho_{i,j}-1}x_2^{j-1}x_3^{i}$ should belong to the stable ideal $J$, but this is impossibile  since $\widetilde{\sigma}:= x_1^{\rho_{i+1,j}-1}x_2^{j-1}x_3^{i}\in \cN(J)$ and 
% % $\frac{\sigma x_3}{x_1} \mid \widetilde{\sigma}$, so the stability of $J$ is contradicted.

Indeed, for $1 \leq i \leq k-1$,  $1\leq j \leq \beta_{i+1}$,
$\sigma := x_1^{\rho_{i,j}}x_2^{j-1}x_3^{i-1} \in J$; being $\rho_{i,j}>0$,
$\min(\sigma)=x_1<x_3$, so $\frac{\sigma x_3}{x_1}=
x_1^{\rho_{i,j}-1}x_2^{j-1}x_3^{i}$ should belong to the stable ideal $J$.

But this implies $\rho_{i,j}> \rho_{i+1,j}$ since $\rho_{i,j}\leq \rho_{i+1,j}$
implies $\widetilde{\sigma}:= x_1^{\rho_{i+1,j}-1}x_2^{j-1}x_3^{i}\in \cN(J)$
and
$\frac{\sigma x_3}{x_1} \mid \widetilde{\sigma}$, contradicting   the stability
of $J$.

Finally, $n(\rho)=p$ by definition of $1$-length.

Then, we can define a map 
$$\Xi: \kB^{(S)}_{(p,h,k)} \rightarrow \kP_{(p,h,k)}$$
$$\cB \mapsto \rho, $$
where $\rho $ is constructed from $\cB$ as described above.
We prove that $\Xi$ is a bijection.

It is clearly an injection by definition of lenght of a bar: two different Bar 
Codes have at least one bar with different length.

Now, we have to prove the surjectivity of $\Xi$, so let us take $\rho \in 
\kP_{(p,h,k)}$. We know that it uniquely identifies a Bar Code $\cB$ and by 
Lemma \ref{admiss2} that $\cB$ is admissible, so we only have to prove 
that $\cL_\cB =(p,h,k)$ and that 
$\eta(B)=\cN(J)$, $J$ stable.
\\
The statement  $\cL_\cB =(p,h,k)$ is trivial, since 
\begin{enumerate}
 \item there are $k$ $3$-bars,
 \item  for each $1 \leq i \leq k$, 
$l_2(\cB^{(3)}_i)=\beta_i$ and $\sum_{i=1}^k \beta_i =h$, 
 \item  for each  $1 \leq i \leq k$,  $1 \leq j \leq \beta_i$, 
 $l_1(\cB^{(2)}_t)=\rho_{i,j}$,   $t=(\sum_{l=1}^{i-1}\beta_l )+j$ and $n(\rho)=p$. 
\end{enumerate}

% % % % % % % % % % % Now, let $\cB^{(1)}_l$  $l \in \{1,...,p\}$ be a $1$-bar labelled by
% % % % % % % % % % % $e(\cB^{(1)}_l)=(b_{l,1},b_{l,2},b_{l,3})$, so 
% % % % % % % % % % %   $\pi_{b_{l,3}+1,b_{l,2}+b_{l_3}+1}\geq b_{l,1}+1$.
% % % % % % % % % % %   
% % % % % % % % % % % To prove that $J$ is strongly stable, we have to prove that 
% % % % % % % % % % % $(b_{l,1}+1,b_{l,2},b_{l,3}-1)$, $(b_{l,1},b_{l,2}+1,b_{l,3}-1)$ and
% % % % % % % % % % % $(b_{l,1}+1,b_{l,2}-1,b_{l,3})$ are the e-lists of some $1$-bars of 
% % % % % % % % % % % $\cB$:
% % % % % % % % % % % \begin{itemize}
% % % % % % % % % % %  \item $(b_{l,1}+1,b_{l,2},b_{l,3}-1)$: we have to prove that $\pi_{b_{l_3}, b_{l_2}+b_{l_3}}\geq b_{l_1}+2$.
% % % % % % % % % % %  Since $\pi_{b_{l_3}, b_{l_2}+b_{l_3}}> \pi_{b_{l_3}, b_{l_2}+b_{l_3}+1}\geq \pi_{b_{l_3}+1, b_{l_2}+b_{l_3}+1}\geq b_{l,1}+1$ we are done.
% % % % % % % % % % %  \item $(b_{l,1},b_{l,2}+1,b_{l,3}-1)$:  
% % % % % % % % % % %  we have to prove that $\pi_{b_{l_3}, b_{l_2}+b_{l_3}+1}\geq b_{l_1}+1$.
% % % % % % % % % % %  Since $\pi_{b_{l_3}, b_{l_2}+b_{l_3}+1} \geq \pi_{b_{l_3}+1, b_{l_2}+b_{l_3}+1}\geq  b_{l,1}+1$ we are done.
% % % % % % % % % % %  
% % % % % % % % % % %  \item $(b_{l,1}+1,b_{l,2}-1,b_{l,3})$: 
% % % % % % % % % % %  we have to prove that $\pi_{b_{l_3}+1, b_{l_2}+b_{l_3}}\geq b_{l_1}+2$.
% % % % % % % % % % %  Since $\pi_{b_{l_3}+1, b_{l_2}+b_{l_3}}> \pi_{b_{l_3}+1, b_{l_2}+b_{l_3}+1}\geq b_{l,1}+1$ we are done.
% % % % % % % % % % %  \end{itemize}

 A monomial ideal $J$ is stable if and only if $\kF(J)=\cG(J)$; by Lemma 
\ref{FJNoto} $\kF(J)=A=\{x_3^k, x_2^{\beta_i}x_3^{i-1}, x_1^{\rho_{i,j}}x_2^{j-1}x_3^{i-1}, 
\, 1 \leq i \leq k,\, 1\leq j\leq \beta_i \} $, so we only 
have to prove that $A \subset \cG(J)$, i.e. that, for each element in the star set, 
all the predecessors belong to the Groebner escalier. 
\\ 
We have already proved that $x_3^{k-1} \in \cN(J)$, since  $\min(x_3^k)=x_3$ and
$x_3^k \in \kF(J)$.

Let us take $x_2^{\beta_i}x_3^{i-1}$, $1 \leq i \leq k$; since it belongs to the star set,
 $x_2^{\beta_i-1}x_3^{i-1}\in \cN(J)$, so we only have to prove   that  $x_2^{\beta_i}x_3^{i-2}\in \cN(J)$, $2 \leq i \leq k$.
\\ The bar $\cB^{(3)}_{i-1}$ is labelled by $x_3^{i-2}$ and, over $\cB^{(3)}_{i-1}$ , there are $\beta_{i-1}>\beta_i$ $2$-bars. The $(\beta_i+1)$-th $2$-bar over $\cB^{(3)}_{i-1}$, i.e.
$\cB^{(2)}_t,\, t=\sum_{l=1}^{i-2}\beta_l +(\beta_i +1)$, is labelled by $x_2^{\beta_i}x_3^{i-2}$, so all the terms labelling the $1$-bars over $\cB^{(2)}_t$ are 
multiples of $x_2^{\beta_i}x_3^{i-2}$ and since the Bar Code is admissible, we can desume that 
$x_2^{\beta_i}x_3^{i-2} \in \cN(J)$.
\\
Let us finally take $ x_1^{\rho_{i,j}}x_2^{j-1}x_3^{i-1}, 
\, 1 \leq i \leq k,\, 1\leq j\leq \beta_i$; we need to prove that 
$x_1^{\rho_{i,j}}x_2^{j-2}x_3^{i-1}$ and $x_1^{\rho_{i,j}}x_2^{j-1}x_3^{i-2}$, when they are defined,
 belong to $\cN(J)$.

 \begin{itemize}
  \item $x_1^{\rho_{i,j}}x_2^{j-2}x_3^{i-1} \in \cN(J)$: we take 
  $\cB^{(2)}_{t}$, $t=\sum_{l=1}^{i-1}\beta_l + (j-1)$, i.e.    
  the $(j-1)$-th $2$-bar over $\cB^{(3)}_i$;   since $\rho_{i,j-1}>\rho_{i,j}$ the $(\rho_{i,j}+1)$-th $1$-bar over $\cB^{(2)}_{t}$ is labelled by $x_1^{\rho_{i,j}}x_2^{j-2}x_3^{i-1}$, so belonging to $ \cN(J)$;
  \item $x_1^{\rho_{i,j}}x_2^{j-1}x_3^{i-2}\in \cN(J)$: analogously as above, it comes from the inequality
  $\rho_{i-1,j}> \rho_{i,j}$.
 \end{itemize}

% % Now, let $\cB^{(1)}_l$  $l \in \{1,...,p\}$ be a $1$-bar labelled by
% % $e(\cB^{(1)}_l)=(b_{l,1},b_{l,2},b_{l,3})$, so 
% %   $\pi_{b_{l,3}+1,b_{l,2}+b_{l_3}+1}\geq b_{l,1}+1$.
% %   
% % To prove that $J$ is strongly stable, we have to prove that 
% % $(b_{l,1}+1,b_{l,2},b_{l,3}-1)$, $(b_{l,1},b_{l,2}+1,b_{l,3}-1)$ and
% % $(b_{l,1}+1,b_{l,2}-1,b_{l,3})$ are the e-lists of some $1$-bars of 
% % $\cB$:
% % \begin{itemize}
% %  \item $(b_{l,1}+1,b_{l,2},b_{l,3}-1)$: we have to prove that $\pi_{b_{l_3}, b_{l_2}+b_{l_3}}\geq b_{l_1}+2$.
% %  Since $\pi_{b_{l_3}, b_{l_2}+b_{l_3}}> \pi_{b_{l_3}, b_{l_2}+b_{l_3}+1}\geq \pi_{b_{l_3}+1, b_{l_2}+b_{l_3}+1}\geq b_{l,1}+1$ we are done.
% %  \item $(b_{l,1},b_{l,2}+1,b_{l,3}-1)$:  
% %  we have to prove that $\pi_{b_{l_3}, b_{l_2}+b_{l_3}+1}\geq b_{l_1}+1$.
% %  Since $\pi_{b_{l_3}, b_{l_2}+b_{l_3}+1} \geq \pi_{b_{l_3}+1, b_{l_2}+b_{l_3}+1}\geq  b_{l,1}+1$ we are done.
% %  
% %  \item $(b_{l,1}+1,b_{l,2}-1,b_{l,3})$: 
% %  we have to prove that $\pi_{b_{l_3}+1, b_{l_2}+b_{l_3}}\geq b_{l_1}+2$.
% %  Since $\pi_{b_{l_3}+1, b_{l_2}+b_{l_3}}> \pi_{b_{l_3}+1, b_{l_2}+b_{l_3}+1}\geq b_{l,1}+1$ we are done.
% %  \end{itemize}
This proves the stability of $J$, concluding our proof.
\end{proof}

Now, by Theorem \ref{bijez3varStab}, counting  stable ideals in three 
variables becomes an application of Theorem \ref{ContoKratStab} (see \cite{Krat}).

Fix a  constant Hilbert polynomial $p$. Lemma  \ref{minhmaxh} 
 allows to enumerate all bar lists. Fix then a bar list $(p,h,k)$ and construct the plane partitions
 $\rho $ as explained above, denoting by $(\beta_1,...,\beta_k)$ their shape.
 Finally, denote by $b=(1,...,1)$ and  $a=(a_1,...,a_k)$ such that
 \begin{center}
\begin{equation}\label{Ai}
\left\{
        \begin{array}{l}
            a_1=p-\frac{\beta_1(\beta_1-1)}{2}-\sum_{i=2}^k \frac{\beta_i(\beta_i+1)}{2}\quad \\
            a_{i}=a_{i-1}-1,   \; 2\leq i \leq k\quad \\
        \end{array}
    \right.
\end{equation}
\end{center}
 the vectors  of Theorem \ref{ContoKratStab}. We can compute 
the number of  stable ideals
by exploiting the formula in the aforementioned Theorem  (see appendix \ref{SIapp}).

We remark that our choice for $a$ and $b$ meets the required inequalities of Theorem 
\ref{ContoKratStab}, remembering that $\mu=0$ and $\lambda_i>\lambda_{i+1}$ for each $i=1,...,k-1$.
Indeed, $a_i=a_{i+1}+1$ so $a_i \geq a_{i+1}$ and 
$b_i+(\lambda_i-\lambda_{i+1})= 1+(\lambda_i-\lambda_{i+1}) \geq 1=b_{i+1}$.

\section{Counting strongly stable ideals}\label{StStCount}

In this section, we extensively deal with  strongly stable ideals (see Definition \ref{StronglyStab}).

An asymptotical estimation of the number of strongly stable ideals with a fixed constant Hilbert
 polynomial has been given by Onn-Sturmfels in \cite{SO};
 in the aforementioned paper, ${\NN^2 \choose n}_{\textrm{stair}}$ 
 denotes the size-$n$ subsets of $\NN^2$ that are also staircases.

 \begin{Proposition}\label{SturmStrSt}
  The number of Borel-fixed staircases in ${\NN^2 \choose n}_{\textrm{stair}}$ is $2^{\Omega(\sqrt{n})}$.
 \end{Proposition}

The following Lemma is enough to deal with the case of two variables.

\begin{Lemma}\label{StabEqStrStab}
An ideal $I \triangleleft \ck[x_1,x_2]$ is stable if and only if it is strongly 
stable.
\end{Lemma}
\begin{proof}
 A strongly stable ideal is trivially stable, so we only need to
 prove the  converse, namely, given a stable ideal  $I$,  we have to show that for each  
for every term $\tau\in I$ and pair of variables
$x_i,\ x_j$ such that $x_i\vert \tau$ and $x_i<x_j$,
then also $
\frac{\tau x_j}{x_i} $ belongs to $I$.
The only pair of variables of the above type is $x_1<x_2$ and $x_1$ is the 
smallest variable in the polynomial ring $\ck[x_1,x_2]$ so, if
$x_1 \mid \tau \in I$, then $x_1 =\min(\tau)$ and $\frac{\tau x_2}{x_1}\in I$ by 
definition of stable ideal, whereas if $x_1 \nmid \tau$ there is nothing to do.
This proves the claimed equivalence.
\end{proof}

By the above Lemma and by Proposition \ref{NumBorel2V}, we can conclude that 
the number of strongly stable ideals $J\triangleleft\ck[x_1,x_2]$ with $H_{\_}(t,J)=p$ is $\sum_{i=1}^h 
Q(p,i),$ where $h := \left\lfloor \frac{-1+\sqrt{1+8p}}{2} \right\rfloor$
% % 
% % $h=\lfloor \frac{-1+\sqrt{1+8p}}{2} \rfloor$ 
and $Q(p,i)$ is the number of integer partitions of $p$ into $i$ distinct parts.

Let us examine now the case of strongly ideals in $\ck[x_1,x_2,x_3]$.\\
% % % From Lemma \ref{StabEqStrStab}, we can desume that Corollary \ref{ph11}
% % %  holds true for strongly stable ideals. Moreover, Lemma \ref{minhmaxh} only relies on 
% % % Proposition \ref{barredecr}, which holds also for strongly stable ideals, 
Strongly stable ideals are also stable, so all the propositions proved for stable ideals also hold here; then  the computation of the bar lists is the same as done for  stable ideals.
Fixed a bar list $(p,h,k)$, we first compute the integer partitions of $h$ in $k$ distinct parts.
Each partition $(\alpha_1,...,\alpha_k) \in \NN^k$, $\alpha_1>...>\alpha_k$, $\sum_{i=1}^k \alpha_i =h$ represents
a precise structure for the $2$-bars and the $3$-bars: for each $1 \leq i \leq k$ there are exactly $\alpha_i$
 $2$-bars over $\cB^{(3)}_i$. 
\\

\medskip
Now, fix a partition  $\overline{\alpha}\in I_{(h,k)}$, $\overline{\alpha}=(\overline{\alpha_1},...,\overline{\alpha_k}) \in \NN^k$, $\overline{\alpha_1}>...>\overline{\alpha_k}$, $\sum_{i=1}^k \overline{\alpha_i}=h$. We can construct the plane partitions $\pi$ of the form

$$\pi=(\pi_{i,j})=\left( \begin{matrix} 
\pi_{1,1}& \pi_{1,2}&...& ...& ...& ...& ... & ... & \pi_{1,\overline{\alpha_1}}\\
0...& \pi_{2,2}& ...& ...& ...& ...& ...&  \pi_{2,2+\overline{\alpha_2}-1} &0...  \\
0...&...  & ...& ...& ...&...& ...& ... &... \\
0...& ...&... & \pi_{k,k}& ...& ...& \pi_{k,k+\overline{\alpha_k}-1} &0...  &...  \\
\end{matrix} \right)$$
s.t. 
\begin{enumerate}
\item $\pi_{i,j}>0$, $1 \leq i \leq k$, $i\leq j \leq i+\overline{\alpha_i}-1$; 
\item  $\pi_{i,j}>\pi_{i,j+1}$, $1 \leq i \leq k$, $i\leq j < i+\overline{\alpha_i}-1$;
\item $\pi_{i,j}\geq \pi_{i+1,j}$  $1 \leq i \leq k-1$,  $i+1\leq j \leq i+\overline{\alpha_{i+1}}-1$;
\item $n(\pi)=\sum_{i=1}^{k}\sum_{j=i}^{i+\overline{\alpha_i}-1}\pi_{i,j}=p$.
\end{enumerate}
These plane partitions are exactly of the form of Definition \ref{PlanePartShaped}, with $\lambda_i = i + \overline{\alpha_i} -1 \geq i$, $1 \leq i \leq k$, $c=1$ and $d =0$.
\\
In Remark  \ref{ContenSStab}, we will highlight the 
relation between these partitions and the ones defined in the previous section \ref{COUNTSTAB}.
\\
We
denote by $\kS_{(p,h,k),\overline{ \alpha}}$ the set of all partitions defined above and $\kS_{(p,h,k)}=\bigcup_{\overline{\alpha} \in I_{(h,k)}}\kS_{(p,h,k),\overline{ \alpha}}$. In other words, 

$$\kS_{(p,h,k),\overline{ \alpha}}=\{\pi \in \kS_{\lambda}(1,0),\, n(\pi)=p,\, \lambda_i=i+ \overline{\alpha_i}-1,\, 1\leq i \leq k\}$$

$$\kS_{(p,h,k)}=\{\pi \in \kS_{\lambda}(1,0),\, n(\pi)=p,\, \lambda_i=i+\overline{\alpha_i}-1,\, 1\leq i \leq k, \textrm{ for some } \overline{\alpha} \in I_{(h,k)}\}$$ 
\begin{Remark}\label{ContenSStab}
 We remark that the set of the shifted plane partitions defined here for strongly stable ideals
 can be easily viewed as a subset of the strict plane partitions defined in the previous section  for counting stable ideals.
 \\
With the notation above, let us take a shifted plane partition $\pi:=(\pi_{i,j})$, $1 \leq i \leq k$, $i \leq j \leq i+\alpha_i-1$. There are exactly $\alpha_i$ elements in the $i$-th row and the values in row $i$ 
is shifted to the right by $i-1$ positions.
We define then a non-shifted plane partition $\rho:=(\rho_{i,m})$ of shape $\alpha=(\alpha_1,...,\alpha_k)$, by 
$\rho_{i,m} = \pi_{i,m+i-1}$  $1 \leq i \leq k$, $1 \leq m \leq  \alpha_i$.
We prove that  $\rho \in \kP_{(p,h,k), \alpha}$:
\begin{itemize}
 \item $\rho_{i,m}>0$, $1 \leq i \leq k$, $1\leq m \leq \alpha_i$
 holds true since $\pi_{i,j}>0$, $1 \leq i \leq k$  $i\leq j \leq i+\alpha_i-1$.
% % %  
% % %  indeed  
% % %  $\rho_{i,m} = \pi_{i,j}$ with $j=m+i-1$; but $i\leq j \leq i+\alpha_i-1\}$, so it comes 
% % %  from the analogous property holding for $\pi$. 
\item  $\rho_{i,m}>\rho_{i,m+1}$, $1 \leq i \leq k$, $1\leq m \leq \alpha_i-1$ is trivially true 
since  $\pi_{i,m+i-1}>\pi_{i, m+i}$.
\item $\rho_{i,m}> \rho_{i+1,m}$  $1 \leq i \leq k-1$,  $ 1\leq j \leq 
\alpha_{i+1}$ comes from 
$\pi_{i, m+i-1}>\pi_{i, m+i}\geq \pi_{i+1, m+i}$.
\item
%%$n(\rho)=\sum_{i=1}^{k}\sum_{m=1}^{\alpha_i}\rho_{i,j}=p$: it holds from 
$n(\rho)=\sum_{i=1}^{k}\sum_{m=1}^{\alpha_i}\rho_{i,j}= \sum_{i=1}^{k}\sum_{j=i}^{\alpha_i + i-1}\pi_{i,j}  =p$.
 \end{itemize}

 On the other hand, we have to point out that there are some strict plane partitions that cannot be 
 brought back to any shifted plane partition. For example, if we shift 
 $$\rho=\left( \begin{matrix} 
4& 2&1 \\
3&  0&0 \\
\end{matrix} \right)$$
 we get 
  $$\pi=\left( \begin{matrix} 
4& 2&1 \\
0 &3 &0 \\
\end{matrix} \right),$$
which is not of the type defined here and cannot be associated to any strongly stable monomial ideal.
\end{Remark}

Each plane partition $\pi \in \kS_{(p,h,k)}$ uniquely identifies a Bar 
Code $\cB$:  
\begin{itemize}
 \item[(a)] each row $i$ represents a $3$-bar $\cB^{(3)}_i$, $1 \leq i \leq k$;
 \item[(b)] for each row $i$, $1 \leq i \leq k$, $l_2(\cB^{(3)}_i)=\overline{\alpha_i}$; the  $\overline{\alpha_i}$ nonzero entries represent the $\overline{\alpha_i}$  $2$-bars  over $\cB^{(3)}_i$, i.e $\cB^{(2)}_{t}$, where $t=(\sum_{l=1}^{i-1}\overline{\alpha_l} )+j-i+1$, $i \leq j \leq i+\overline{\alpha_i}-1$;
 \item[(c)] for each  $1 \leq i \leq k$, and  each $i \leq j \leq i+\overline{\alpha_i}-1$, the number $\pi_{i,j}$ represents the number of $1$-bars over  $\cB^{(2)}_t$,  $t=(\sum_{l=1}^{i-1}\overline{\alpha_l} )+j-i+1$, namely the $j-i+1$-th $2$-bar lying over $\cB^{(3)}_i$. In other words, $l_1(\cB^{(2)}_t)=\pi_{i,j}$.
\end{itemize}
In conclusion, for each  $1 \leq i \leq k$, and  each $i \leq j \leq 
i+\overline{\alpha_i}-1$, the number $\pi_{i,j}$ means that in $\cB$ there are $1$-bars 
labelled by 
$(0,j-i,i-1),(1,j-i,i-1),...,(\pi_{i,j}-1,j-i,i-1)$, but there is no $1$-bar 
labelled by $(\pi_{i,j},j-i,i-1)$, that is also equivalent to say that 
$x_1^0x_2^{j-i}x_3^{i-1},x_1x_2^{j-i}x_3^{i-1},...,x_1^{\pi_{i,j}-1}x_2^{j-i}
x_3^{i-1} $ belong to the set of terms associated to $\cB$ via Bbc1 and Bbc2, 
but $x_1^{\pi_{i,j}}x_2^{j-i}
x_3^{i-1} $ does not belong to the aforementioned set\footnote{Again, as for stable ideals, we will see that $\cB$ is admissible and that $x_1^{\pi_{i,j}}x_2^{j-i}
x_3^{i-1} $ belongs to the star set associated to $\cB$.}. 
\begin{example}\label{PartizDelBC}
 Let us take the bar list $(p,h,k)=(6,3,2)$, $\overline{\alpha_1}=2>\overline{\alpha_2}=1$, $\overline{\alpha_1}+\overline{\alpha_2}=3=h$.
 We have, for example 
 \[ \pi =\left( \begin{array}{cccc}
3 & 2\\
0& 1
\end{array} \right)\] 
 and it holds 
 \begin{enumerate}
\item  $\pi_{i,j}>\pi_{i,j+1}$, $1 \leq i \leq 2$, $i\leq j < i+\overline{\alpha_i}-1$, i.e. $\pi_{1,1}>\pi_{1,2}$ ;
\item $\pi_{i,j}\geq \pi_{i+1,j}$  $i= 1$,  $ j =2$, i.e. $\pi_{1,2}\geq \pi_{2,2}$;
\item $n(\pi)=\sum_{i=1}^{2}\sum_{j=i}^{i+\overline{\alpha_i}-1}\pi_{i,j}=6$.
\end{enumerate}
With the notation of \cite{Krat}, $\lambda_1=\lambda_2=2$.
\\
The partition $\pi$ uniquely identifies the Bar Code $\cB$ below:

\begin{center}
\begin{tikzpicture}

\node at (3.8,-0.5) [] {${\scriptscriptstyle 1}$};
\node at (3.8,-1) [] {${\scriptscriptstyle 2}$};
\node at (3.8,-1.5) [] {${\scriptscriptstyle 3}$};

\node at (4.2,0) [] {${\small 1}$};
\draw [thick] (4,-0.5) --(4.5,-0.5);

\node at (5.2,0) [] {${\small x_1}$};
\draw [thick] (5,-0.5) --(5.5,-0.5);

\node at (6.2,0) [] {${\small x_1^2}$};
\draw [thick] (6,-0.5) --(6.5,-0.5);

\node at (7.2,0) [] {${\small x_2}$};
\node at (8.2,0) [] {${\small x_1x_2}$};

\draw [thick] (7,-0.5) --(7.5,-0.5);

\draw [thick] (8,-0.5) --(8.5,-0.5);

\node at (9.2,0) [] {${\small x_3}$};

\draw [thick] (4,-1.5) --(8.5,-1.5);
\draw [thick] (9,-0.5) --(9.5,-0.5);

\draw [thick] (9,-1.5) --(9.5,-1.5);

\draw [thick] (4,-1)--(6.5,-1);
\draw [thick] (7,-1)--(8.5,-1);
\draw [thick] (9,-1)--(9.5,-1);

\end{tikzpicture}
\end{center}
with $k=2$ $3$-bars $B^{(3)}_1,\,\cB^{(3)}_2$, $l_2(B^{(3)}_1)=2$, $l_2(B^{(3)}_2)=1$. The bars 
$\cB^{(2)}_1$ and $\cB^{(2)}_2$ lie over $B^{(3)}_1$, whereas  $\cB^{(2)}_3$ lie over $B^{(3)}_2$.
As regards $1$-lengths, we have $l_1(\cB^{(2)}_1)=\pi_{1,1}=3$, $l_1(\cB^{(2)}_2)=\pi_{1,2}=2$
and $l_1(\cB^{(2)}_3)=\pi_{2,2}=1$.
The associated set of terms, via BbC1 and BbC2 is 
$\cN=\{1,x_1,x_1^2,x_2,x_1x_2,x_3\}$ 
and it is an order ideal.

\end{example}

\begin{Remark}\label{Lb2}
 The Bar Code $\cB$, uniquely identified by $\pi$, has bar list 
$\cL_\cB=(p,h,k)$. The relation $\mu(3)=k$ comes from (a), $\mu(2)=h$ comes 
from (b), since $\alpha \in I_{(h,k)}$, so $\sum_{i=1}^k \alpha_i=h$, whereas
$\mu(1)=p$ is an easy consequence of (c).
\end{Remark}

\begin{Lemma}\label{admiss}
Fixed $(p,h,k)$ and $\alpha\in I_{(h,k)}$, $\alpha=(\alpha_1,...,\alpha_k) \in 
\NN^k$, $\alpha_1>...>\alpha_k$, $\sum_{i=1}^k \alpha_i=h$, let $\pi$ be a 
partition in $ \kS_{(p,h,k), \alpha}$. The Bar 
Code $\cB$, uniquely identified by $\pi$, is admissible.
\end{Lemma}
\begin{proof}
 By Remark \ref{Lb2}, $\cL_\cB=(p,h,k)$. Consider a $1$-bar $\cB^{(1)}_l$, $1\leq l \leq p$ and let its e-list be
 $e(\cB^{(1)}_l)=(b_{l,1},b_{l,2},b_{l,3})$. 
 From the construction of $\cB$ from $\pi$, we desume that
  $\pi_{b_{l,3}+1,b_{l,2}+b_{l_3}+1}\geq b_{l,1}+1$; moreover, we know that $(m, 
b_{l,2},b_{l,3})$, $0\leq m\leq \pi_{b_{l,3}+1,b_{l,2}+b_{l,3}+1}-1$ are 
e-lists for some bars of $\cB$, so, if $b_{l,1}\geq 1$,  $(b_{l,1}-1, 
b_{l,2},b_{l,3})$ is a bar list labelling a $1$-bar of $\cB$.
 \\
 For $\cB$ being admissible, we also need two other conditions:
 \begin{itemize}
  \item if $b_{l,2}>0$, $(b_{l,1}, b_{l,2}-1,b_{l,3})$ labels a $1$-bar of $\cB$;
  \item if $b_{l,3}>0$,  $(b_{l,1}, b_{l,2},b_{l,3}-1)$  labels a $1$-bar of $\cB$.
 \end{itemize}
 Let us prove them:
 \begin{itemize}
 \item suppose $b_{l,2}>0$; for $(b_{l,1}, b_{l,2}-1,b_{l,3})$ labelling a $1$-bar of $\cB$, we would need $\pi_{b_{l_3}+1, b_{l_2}+b_{l_3} }\geq b_{l_1}+1$, but 
 since $\pi_{b_{l_3}+1, b_{l_2}+b_{l_3} }> \pi_{b_{l_3}+1, b_{l_2}+b_{l_3} +1}\geq b_{l_1}+1$ we are done
  \item suppose $b_{l,3}>0$; for $(b_{l,1}, b_{l,2},b_{l,3}-1)$ labelling a $1$-bar of $\cB$, we would need $\pi_{b_{l_3}, b_{l_2}+b_{l_3} }\geq b_{l_1}+1$, but 
 since $\pi_{b_{l_3}, b_{l_2}+b_{l_3} } >  \pi_{b_{l_3}, b_{l_2}+b_{l_3}+1 }  \geq \pi_{b_{l_3}+1, b_{l_2}+b_{l_3}+1 } \geq b_{l_1}+1$  we are done again and $\cB$ turns out to be admissible.
 \end{itemize}
\end{proof}
\begin{example}\label{AdmBCPart}
 The set of terms associated to the Bar Code constructed in 
example \ref{PartizDelBC} is an order ideal, so the Bar Code is admissible.
\end{example}

\begin{Theorem}\label{bijez3var}
There is a biunivocal correspondence between $\kS_{(p,h,k)}$ and the set \\
$\cB_{(p,h,k)}=\{\cB \in \kA_3 \textrm{ s.t. }   \cL_\cB=(p,h,k),\, \eta(\cB)=\cN(J),\, J \textrm{ strongly stable}     \}.$
\end{Theorem}
\begin{proof}
Let $\cB \in \cB_{(p,h,k)}$. We construct a plane partition  
$$\pi=(\pi_{i,j})=\left( \begin{matrix} 
\pi_{1,1}& \pi_{1,2}&...& ...& ...& ...& ... & ... & \pi_{1,\alpha_1}\\
0...& \pi_{2,2}& ...& ...& ...& ...& ...&  \pi_{2,2+\alpha_2-1} &0...  \\
0...&...  & ...& ...& ...&...& ...& ... &... \\
0...& ...&... & \pi_{k,k}& ...& ...& \pi_{k,k+\alpha_k-1} &0...  &...  \\
\end{matrix} \right)$$
with $k$ rows and $l_2(\cB^{(3)}_1)$ columns.
Fixed the index $i$ for the rows and the index  $j$ for the columns, 
we define $\pi_{i,j}=0$ if $j<i$ or $i+\alpha_i-1<j\leq l_2(\cB^{(3)}_1)$ and
$\pi_{i,j}=l_1(\cB^{(2)}_ {t})$ with $t=(\sum_{l=1}^{i-1}\alpha_l)+j-i+1$
otherwise, where $\alpha_i = l_2(\cB^{(3)}_i)$, $1 \leq i \leq k$. 
\\
We observe that the partition $\pi$ is uniquely determined by $\cB$ and 
that, by Proposition \ref{barredecr}, $\alpha \in I_{(h,k)}$;
we have to  prove that $\pi \in \kS_{(p,h,k)}$. 
\\
The nonzero parts of $\pi$ are positive by definition of length of a bar.

Clearly $\pi_{i,j}>\pi_{i,j+1}$, $1 \leq i \leq k$, $i\leq j < i+\alpha_i-1$, indeed, this can be stated 
as $l_1(\cB^{(2)}_t)>l_1(\cB^{(2)}_{t+1})$ with $\cB^{(2)}_t$ and $\cB^{(2)}_{t+1}$ lying over the same $3$-bar $\cB^{(3)}_{i}$. This statement follows from Proposition \ref{barredecr} b) with $i=1$.

Moreover, $\pi_{i,j}\geq \pi_{i+1,j}$  $1 \leq i \leq k-1$,  $i+1\leq j \leq 
i+\alpha_{i+1}$.
\\
Indeed, if $\pi_{i,j}< \pi_{i+1,j}$ then it would happen that
$x_1^{\pi_{i+1,j}-1}x_2^{j-i-1}x_3^{i} \in \cN(J)$, but 
$x_1^{\pi_{i+1,j}-1}x_2^{j-i}x_3^{i-1} \notin \cN(J)$, contradicting the strongly 
stable property of $J$. By construction, the shape of $\pi$ is 
$\lambda=(\lambda_1,...,\lambda_k)$ with $\lambda_i=i+\alpha_i-1$, $1 \leq i 
\leq k$, so $\pi \in \kS_\lambda (1,0)$.
Moreover, $n(\pi)=p$ by definitions of bar list and $1$-length.

Then, we can define a map 
$$\Xi: \kB_{(p,h,k)} \rightarrow \kS_{(p,h,k)}$$
$$\cB \mapsto \pi, $$
where $\pi $ is constructed from $\cB$ as described above.
We prove that $\Xi$ is a bijection.

It is clearly an injection by definition of lenght of a bar: two different Bar 
Codes have at least one bar with different length.

Now, we have to prove the surjectivity of $\Xi$, so let us take $\pi \in 
\kS_{(p,h,k)}$. We know that it uniquely identifies a Bar Code $\cB$ and by 
Lemma \ref{admiss} that $\cB$ is admissible, so we only have to prove 
that $\cB \in \kB_{(p,h,k)}$.

More precisely, we have to prove that $\cL_\cB =(p,h,k)$ and that 
$\eta(B)=\cN(J)$, $J$ strongly stable.

Since 
\begin{enumerate}
 \item there are $k$ $3$-bars,
 \item  for each row $i$, $1 \leq i \leq k$, 
$l_2(\cB^{(3)}_i)=\alpha_i$ and $\sum \alpha_i =h$, 
 \item  for each  $1 \leq i \leq k$, and  each $i \leq j \leq i+\alpha_i-1$, 
 $l_1(\cB^{(2)}_t)=\pi_{i,j}$,   $t=(\sum_{l=1}^{i-1}\alpha_l )+j-i+1$ and $n(\pi)=p$, 
\end{enumerate}
 then 
$\cL_\cB =(p,h,k)$. 
%\textbf{Ovvio ma non so se spenderci altre parole} 

Now, let $\cB^{(1)}_l$  $l \in \{1,...,p\}$ be a $1$-bar labelled by
$e(\cB^{(1)}_l)=(b_{l,1},b_{l,2},b_{l,3})$, so 
  $\pi_{b_{l,3}+1,b_{l,2}+b_{l_3}+1}\geq b_{l,1}+1$.
  
To prove that $J$ is strongly stable, we have to prove that 
\begin{itemize}
 \item if $b_{l,3}>0$, $(b_{l,1}+1,b_{l,2},b_{l,3}-1)$ and $(b_{l,1},b_{l,2}+1,b_{l,3}-1)$ are the e-lists of some $1$-bars of 
$\cB$
 \item $b_{l,2}>0$, $(b_{l,1}+1,b_{l,2}-1,b_{l,3})$ is the e-list of
 a $1$-bar of 
$\cB$.
\end{itemize}
Let us prove these statements .

\begin{itemize}
 \item suppose that $b_{l,3}>0$ and consider  $(b_{l,1}+1,b_{l,2},b_{l,3}-1)$: we have to prove that $\pi_{b_{l_3}, b_{l_2}+b_{l_3}}\geq b_{l_1}+2$.
 Since $\pi_{b_{l_3}, b_{l_2}+b_{l_3}}> \pi_{b_{l_3}, b_{l_2}+b_{l_3}+1}\geq \pi_{b_{l_3}+1, b_{l_2}+b_{l_3}+1}\geq b_{l,1}+1$ we are done.
 \item  suppose that $b_{l,3}>0$ and consider  $(b_{l,1},b_{l,2}+1,b_{l,3}-1)$:  
 we have to prove that $\pi_{b_{l_3}, b_{l_2}+b_{l_3}+1}\geq b_{l_1}+1$.
 Since $\pi_{b_{l_3}, b_{l_2}+b_{l_3}+1} \geq \pi_{b_{l_3}+1, b_{l_2}+b_{l_3}+1}\geq  b_{l,1}+1$ we are done.
 
 \item   suppose that $b_{l,2}>0$ and consider  $(b_{l,1}+1,b_{l,2}-1,b_{l,3})$: 
 we have to prove that $\pi_{b_{l_3}+1, b_{l_2}+b_{l_3}}\geq b_{l_1}+2$.
 Since $\pi_{b_{l_3}+1, b_{l_2}+b_{l_3}}> \pi_{b_{l_3}+1, b_{l_2}+b_{l_3}+1}\geq b_{l,1}+1$ we are done.
 \end{itemize}
This concludes our proof.
\end{proof}

Now, by Theorem \ref{bijez3var}, counting strongly stable ideals in three 
variables becomes an application of Theorem \ref{ContoKrat} (\cite{Krat2}).

Fix a  constant Hilbert polynomial $p$. Lemma  \ref{minhmaxh} 
 allows to compute all bar lists. Fix then a bar list $(p,h,k)$ and their shape $\lambda$.
 Finally, denote by $b=(1,...,1)$ and  $a=(a_1,...,a_r)$ such that
 \begin{center}
\begin{equation}\label{Bi}
\left\{
        \begin{array}{l}
            a_r=\lambda_r-r+1,...,\mathbf{M}-r+1 \quad\\
            a_{i}=a_{i+1}+1,...,\mathbf{M}-i+1,   \; 1\leq i \leq r-1 \quad \\
        \end{array}
    \right.
\end{equation}
\end{center}
 $\mathbf{M}:=p-\sum_{i=1}^r\frac{c_i(c_i+1)}{2}$, $c_1=\lambda_1-1$ and $c_j=\lambda_j-j+1$,
$j=2,...,r$,  the vectors  of Theorem \ref{ContoKrat}. We can compute 
the number of strongly stable ideals
by exploiting the formula in the aforementioned Theorem (see appendix \ref{SSIapp}).

\bigskip

There is a simple case of shifted $(1,0)$-plane
partition for which a closed formula can be easily
computed.
\begin{Proposition}\label{PartizFacile}
Let $p \in \NN\setminus\{0\}$. Then there is a biunivocal correspondence between 
the sets $\kS_\lambda(1,0)$ with $\lambda=(2,2)$ and 
% $P_{3,p-1}$ where
% $$P_{M_p}:=\{\pi \textrm{ shifted  } (1,0)\textrm{-plane partition of shape } 
% \lambda=(2,2) \textrm{ of  } p\} $$
% and
$P_{3,p-1}:=\{\lambda' \textrm{ partition of } p-1 \textrm{ in } 3 \textrm{ non 
necessarily distinct parts }\}. $
\end{Proposition}
\begin{proof}
Let $\pi \in \kS_\lambda(1,0),\, \lambda=(2,2)$, then $\pi$ is of the form
   \[\left(              
\begin{array}{cc}                  
\pi_{1,1} &\pi_{1,2}\\
0&\pi_{2,2}
\end{array}       
\right)\]
with $\pi_{1,1}>\pi_{1,2}$, $\pi_{1,2}\geq \pi_{2,2}$, and 
$\pi_{1,1}+\pi_{1,2}+\pi_{2,2}=p$.\\
Consider the $3$-uple $\pi'=(\pi_{1,1}-1,\pi_{1,2},\pi_{2,2})$, whose sum is 
$\pi_{1,1}-1+\pi_{1,2}+\pi_{2,2}=p-1$. Since 
$\pi_{1,1}-1 \geq \pi_{1,2}\geq \pi_{2,2}$ then $\pi'$ is a partition of $p-1$ 
in three non necessarily distinct parts.\\
Conversely, let us consider a partition $\pi'=(\pi'_1, \pi'_2, \pi'_3)\in 
P_{3,p-1}$ of $p-1$ in three non necessarily distinct parts. Then $\pi'_1 \geq 
\pi'_2 \geq \pi'_3$. Take $\pi'':=(\pi'_1+1, \pi'_2, \pi'_3)$: $\pi'_1+1>\pi'_2$, 
$\pi'_2 \geq \pi'_3$ and $\pi'_1+1+\pi'_2+\pi'_3=p$ so, putting it in the plane 
as 
   \[\left(              
\begin{array}{cc}                  
\pi'_1+1 &\pi'_2\\
0&\pi'_3
\end{array}       
\right)\]
we get a shifted $(1,0)$-plane partition of shape $(2,2)$ of $p$.
\end{proof}

The closed formula for the partitions of Proposition \ref{PartizFacile} is well 
known in literature.
\begin{Proposition}[Hardy-Wright,\cite{HW,OEIS}]
The partitions of the set $P_{3,p-1}$ are $\lfloor\frac{(p-1)^2+6}{12}\rfloor .$
\end{Proposition}

In general, finding closed formulas for plane
partitions is
 rather difficult and most of them are still
unknown.

\section{Future work and generalizations}

In this section, we present a conjecture on 
the relation between (strongly) stable ideals in $\ck[x_1,...,x_n]$, $n>3$ and 
integer partitions.
\\
We start setting an ordering on $n$-tuples of natural numbers, that we will need to define the required partitions.
\begin{Definition}\label{Ordering}
Let $(i_1,...,i_n), (j_1,...,j_n)\in \NN^n$; we say that $(i_1,...,i_n)< (j_1,...,j_n)$ if $i_1 \leq j_1,...,i_n \leq j_n$ but $(i_1,...,i_n)\neq (j_1,...,j_n)$.
\end{Definition}

We can now define \emph{strict solid partitions} (so partitions of dimension $n=3$) and then, inductively \emph{strict $n$-partitions}, for $n\geq 4$; they are the natural generalization for the partitions of Definition  \ref{PlanePartNoShift} and they will be necessary in order  to state our conjecture for stable ideals.

\begin{Definition}\label{SolidStrictNoShift}
 Let $\rho=(\rho_{i,j})_{i \in \{1,...,r\}, j \in \{1,...,\beta_i\}}$ be a $(1, 1)$-plane partition of shape $\beta=(\beta_1,...,\beta_r)$, $\beta_1>...>\beta_r$ (see Definition \ref{PlanePartNoShift}). A \emph{strict solid partition} (or \emph{strict $3$-partition}) of shape $\rho$ is a $3$-dimensional array
 $\gamma=(\gamma_{i_1,i_2,i_3})$, $ 1\leq i_1 \leq \beta_{i_3},\, 1 \leq i_2 \leq \rho_{i_3,i_1}, \,1 \leq i_3 \leq r$, s.t.
 \begin{itemize}
  \item for each $1 \leq l \leq r$, the $2$-dimensional array $\gamma_l:=(\gamma_{i_1,i_2,l})$ is a $(1,1)$-plane partition of shape 
  $\rho_l=(\rho_{l,1},...,\rho_{l,\beta_l})$.
  \item $\gamma_{i_1,i_2,i_3}> \gamma_{j_1,j_2,j_3}$, for
%   $i_1\leq j_1$,
%   $i_2\leq j_2$, $i_3\leq j_3$,  but 
  $(i_1,i_2,i_3)< (j_1,j_2,j_3)$.
 \end{itemize}

\end{Definition}

We denote by $\kP_{\rho}(1,1,1)$ the set of strict $3$-partitions
 of shape $\rho$.

% % \textbf{We denote by $3\kP()_{\rho}$... legame con la bar list. 
% % I piani sono le potenze di t}

\begin{Definition}\label{nStrictNoShift}
 For $n\geq 4$, consider a strict $(n-1)$-partition $\rho=(\rho_{\overline{i}_1,...,\overline{i}_{n-1}})$ 
 with $1 \leq \overline{i}_{n-1} \leq h$, for some $h>0$.\\
 A \emph{strict $n$-partition} of shape $\rho$ is a $n$-dimensional array
  $\gamma=(\gamma_{i_1,...,i_n})$  s.t.
  \begin{itemize}
   \item for each $1 \leq l \leq h$, $\gamma_l:=(\gamma_{i_1,...,i_{n-1},l})$ is a strict $(n-1)$-partition of shape $\rho_l=(\rho_{\overline{i}_1,...,\overline{i}_{n-2},l})$
   \item $\gamma_{i_1,...,i_n}>\gamma_{j_1,...,j_n}$, for 
$(i_1,...,i_n)< (j_1,...,j_n)$.
   %     $i_1\leq j_1,...,i_n\leq j_n$.
  \end{itemize}

\end{Definition}

We denote by $\kP_{\rho}(\underbrace{1,1,...,1}_n)$ the set of strict $n$-partitions
 of shape $\rho$.

\begin{example}\label{strict3noshift}
  Let us consider the $(1,1)$-plane partition
\[\rho=\left( \begin{array}{ccc}
4 & 2& 1\\
2&1&0\\
1&0&0
\end{array} \right)\]
  of shape $\beta=(3,2,1)$.
\\
An example of strict solid partition of shape $\rho$
  is is the following $\gamma$, formed by three $(1,1)$-plane partitions $\gamma_1,\gamma_2,\gamma_3$:
\[\gamma_1=\left( \begin{array}{cccc}
 \mathbf{\gamma_{1,1,1}}&\mathbf{\gamma_{1,2,1}}&\gamma_{1,3,1}&\gamma_{1,4,1}\\
 \mathbf{\gamma_{2,1,1}}&\gamma_{2,2,1}&0&0\\
 \gamma_{3,1,1}&0&0&0\end{array} \right)
=\left( \begin{array}{cccc}
 \mathbf{4} &\mathbf{3}&2&1\\
 \mathbf{3}&1&0&0\\
1&0&0&0
\end{array} \right)\]  
  
  \[\gamma_2
  =\left( \begin{array}{ccc}
 \mathbf{\gamma_{1,1,2}}& \gamma_{1,2,2}&0\\
\gamma_{2,1,2}&0&0
\end{array} \right)  =\left( \begin{array}{ccc}
\mathbf{2}&1&0\\
1&0&0
\end{array} \right)\]  
  
  \[\gamma_3=\left( \begin{array}{ccc}
 \gamma_{1,1,3}&0&0
\end{array} \right)   =\left( \begin{array}{ccc}
1&0&0
\end{array} \right)\]  
 
 where we mark in bold the elements of $\gamma_i$ over which those of $\gamma_{i+1}$ are posed, for $i=1,2$.  
  \end{example}

\begin{example}\label{nstrictpartnoshift}
  Let us consider the following very simple strict solid partition $\rho$:
  
   \[\rho_1=\left( \begin{array}{cc}
\mathbf{2}&1\\
1&0
\end{array} \right) \quad 
\rho_2=\left( \begin{array}{cc}
1&0\\
\end{array} \right)
\]  
  
An example of strict   $4$-partition of shape $\rho$ is 
   \[\gamma_1=\left( \begin{array}{cc}
\mathbf{\gamma_{1,1,1,1}}&\gamma_{1,2,1,1}\\
\gamma_{2,1,1,1}&0
\end{array} \right) \quad 
\left( \begin{array}{cc}
\gamma_{1,1,2,1}&0\\
\end{array} \right)
   =\left( \begin{array}{cc}
\mathbf{4}&2\\
2&0
\end{array} \right) \quad 
\left( \begin{array}{cc}
1&0\\
\end{array} \right)
\]  
  
 \[ \gamma_2=\left( \begin{array}{cc}
\gamma_{1,1,1,2}&0\\
\end{array} \right)
 =\left( \begin{array}{cc}
1&0\\
\end{array} \right)
\]

\end{example}

It is possible to generalize Lemma \ref{minhmaxh} to the 
case of $n$ variables, with some cumbersome computation, so 
that it is possible to compute the bar lists in order to count stable 
ideals in $\ck[x_1,...,x_n]$.
\\
Fixed a bar list $(p_1,...,p_n) \in \NN^n,\, p_1,...,p_n\neq 0$ and a strict $(n-2)$-partition $\rho$ of shape $(p_2,...,p_n)$, we define the 
following sets
$$\kP_{\rho}(p_1,...,p_n):= \{\gamma \in \kP_{\rho}(\underbrace{1,...,1}_{n-1}), \, n(\gamma)=p_1\} $$
 and
$$\kP(p_1,...,p_n):= \{\gamma \in \kP_{\rho}(\underbrace{1,...,1}_{n-1}), 
\textrm{ for some } \rho \in \kP(p_2,...,p_n), \textrm{ s.t. }
n(\gamma)=p_1\}, $$
where $\kP_{\rho}(\underbrace{1,1,...,1}_{n-1})$ is the set of strict $(n-1)$-partitions
 of shape $\rho$.

We can then state our conjecture for stable ideals.

\begin{Conjecture}\label{ConjStab}
There is a biunivocal correspondence between the set 
$\kP_{\rho}(p_1,...,p_n)$ and the set 
$\kB_{(p_1,...,p_n)}:=\{\cB \in \kA_n \textrm{ s.t. }  
\cL_\cB=(p_1,...,p_n),\, \eta(\cB)=\cN(J),\, J \textrm{  stable}     \}$.
\end{Conjecture}

In an analogous (but a bit more cumbersome) way, we handle now the case of strongly stable ideals, giving the necessary generalizations of Definition  \ref{PlanePartShaped} and stating our conjecture.

\begin{Definition}\label{SolidShift}
 Let $\pi=(\pi_{i,j})_{i \in \{1,...,r\}, j \in \{1,...,\alpha_i\}}$ 
 be a shifted $(1, 0)$-plane partition of shape 
 $\alpha=(\alpha_1,...,\alpha_r)$, $\alpha_1\geq...\geq\alpha_r\geq r$  
 (see Definition \ref{PlanePartShaped}). 
 A \emph{shifted solid partition} (or \emph{shifted $3$-partition}) 
 of shape $\pi$ is a $3$-dimensional array
 $\gamma=(\gamma_{i_1,i_2,i_3})$, $ i_3\leq i_1 \leq \alpha_{i_3},
 \, i_1 \leq i_2 \leq \pi_{i_3,i_1}+i_1-1, 
 \,1 \leq i_3 \leq r$, s.t.
 \begin{itemize}
  \item for each $1 \leq l \leq r$, the $2$-dimensional array $\gamma_l:=(\gamma_{i_1,i_2,l})$ is a shifted $(1,0)$-plane partition of
   shape 
    $\widetilde{\pi}_l=(\pi_{l,l}+l-1,\pi_{l,l+1}+l,...,\pi_{l,\alpha_l}+\alpha_l-1)$.
% %   $\widetilde{\pi}_l=(\pi_{l,l},\pi_{l,l+1}+1,...,\pi_{l,\alpha_l}+\alpha_l-l-1)$.  
  \item $\gamma_{i_1,i_2,i_3}\geq \gamma_{i_1,i_2,i_3+1}$.
  %, for
%   $i_1\leq j_1$,
%   $i_2\leq j_2$, $i_3\leq j_3$,  but 
  %$(i_1,i_2,i_3)< (j_1,j_2,j_3)$.
 \end{itemize}

\end{Definition}

We denote by $\kS_{\pi}(1,1,1)$ the set of shifted $3$-partitions
 of shape $\pi$.

\begin{Definition}\label{nShift}
 For $n\geq 4$, consider a shifted $(n-1)$-partition $\pi=(\pi_{\overline{i}_1,...,\overline{i}_{n-1}})$ 
 with $1 \leq \overline{i}_{n-1} \leq h$, for some $h>0$.\\
 A \emph{shifted $n$-partition} of shape $\pi$ is a $n$-dimensional array
  $\gamma=(\gamma_{i_1,...,i_n})$  s.t.
  \begin{itemize}
   \item for each $1 \leq l \leq h$, $\gamma_l:=(\gamma_{i_1,...,i_{n-1},l})$ is a shifted $(n-1)$-partition with shape given by the $(n-2)$-partition  $\widetilde{\pi}_l=(\pi_{\overline{i}_1,...,\overline{i}_{n-2},l}+i_m-1)$, where $m$ is the maximal index s.t. 
   $i_m>1$, and such that, w.r.t. the ordering defined in 
   Definition \ref{Ordering},  $(l,l,...,l)$ is the minimal $(i_1,...,i_{n-1},l)$ for which $\gamma_{i_1,...,i_{n-1},l}\neq 0$;
   \item $\gamma_{i_1,...,i_n} \geq \gamma_{i_1,...,i_n+1}$.
  \end{itemize}

\end{Definition}
% % % % \begin{Definition}\label{nShift}
% % % %  For $n\geq 4$, consider a shifted $(n-1)$-partition $\pi=(\pi_{\overline{i}_1,...,\overline{i}_{n-1}})$ 
% % % %  with $1 \leq \overline{i}_{n-1} \leq h$, for some $h>0$.\\
% % % %  A \emph{shifted $n$-partition} of shape $\pi$ is a $n$-dimensional array
% % % %   $\gamma=(\gamma_{i_1,...,i_n})$  s.t.
% % % %   \begin{itemize}
% % % %    \item for each $1 \leq l \leq h$, $\gamma_l:=(\gamma_{i_1,...,i_{n-1},l})$ is a shifted $(n-1)$-partition with shape given by the $(n-2)$-partition  $\widetilde{\pi}_l=(\pi_{\overline{i}_1,...,\overline{i}_{n-2},l}+i_m-1)$, where $m$ is the maximal index s.t. 
% % % %    $i_m>1$ and such that, w.r.t. the ordering defined in 
% % % %    Definition \ref{Ordering}, the minimal $(i_1,...,i_{n-1},i_n)$ for which $\gamma_{i_1,...,i_{n-1},i_n}\neq 0$ s.t.
% % % %     $i_n=l$ is $(l,l,...,l)$;   
% % % %    \item $\gamma_{i_1,...,i_n} \geq \gamma_{i,...,i_n+1}$.
% % % %   \end{itemize}
% % % % 
% % % % \end{Definition}
We denote by $\kS_{\pi}(\underbrace{1,1,...,1}_n)$ the set of shifted $n$-partitions
 of shape $\pi$.

 \begin{example}\label{strict3noshift}
  Let us consider the shifted $(1,0)$-plane partition
\[\pi=\left( \begin{array}{ccc}
3 & 2& 1\\
0&2&0\\ 
\end{array} \right)\]
  of shape $\alpha=(3,2)$.
\\
An example of strict solid partition of shape $\pi$
  is the following $\gamma$, formed by two shifted $(1,0)$-plane partitions $\gamma_1,\gamma_2$:
\[\gamma_1=\left( \begin{array}{ccc}
  \gamma_{1,1,1} &\gamma_{1,2,1}&\gamma_{1,3,1}\\
 0 &\mathbf{\gamma_{2,2,1}}& \mathbf{\gamma_{2,3,1}}\\
 0&0&\gamma_{3,3,1}\end{array} \right)
=\left( \begin{array}{ccc}
 3 &2&1\\
0&\mathbf{2}&\mathbf{1}\\
0&0&1
\end{array} \right)\]  
  
  \[\gamma_2
  =\left( \begin{array}{ccc}
0&0&0\\
0&\gamma_{2,2,2}&\gamma_{2,3,2}
\end{array} \right)  =\left( \begin{array}{ccc}
0&0&0\\
0&2&1
\end{array} \right)\]  
 
 where we mark in bold the elements of $\gamma_1$ over which those of $\gamma_{2}$ are posed.  
  \end{example}

\begin{example}\label{nstrictpartnoshift}
  Let us consider the following very simple shifted solid partition $\pi$:
  
   \[\pi_1=\left( \begin{array}{cc}
2&1\\
0&\mathbf{1}
\end{array} \right) \quad 
\pi_2=\left( \begin{array}{cc}
0&1\\
\end{array} \right)
\]  
  
An example of strict   $4$-partition of shape $\pi$ is\footnote{According to the $3$-partition shape definition $\gamma_{2,2,2,1}\geq \gamma_{2,2,1,1}$.
}
   \[\gamma_1=\left( \begin{array}{cc}
\gamma_{1,1,1,1}&\gamma_{1,2,1,1}\\
0&\mathbf{\gamma_{2,2,1,1}}
\end{array} \right) \quad 
\left( \begin{array}{cc}
0& 0 \\
0&\gamma_{2,2,2,1}\\
\end{array} \right)
   =\left( \begin{array}{cc}
3&2\\
0&\mathbf{2}
\end{array} \right) \quad 
\left( \begin{array}{cc}
0&0\\
0&1
\end{array} \right)
\]  
  
 \[ \gamma_2=\left( \begin{array}{cc}
0&0\\
 0&0\\
\end{array} \right) \quad \left( \begin{array}{cc}
0&0\\
 0&\gamma_{2,2,2,2}\\
\end{array} \right)
 = \left( \begin{array}{cc}
0&0\\
 0&0\\
\end{array} \right) \quad \left( \begin{array}{cc}
0&0\\
0&1\\
\end{array} \right)
\]  
\end{example}

Fixed a bar list $(p_1,...,p_n) \in \NN^n,\, p_1,...,p_n\neq 0$ and a shifted $(n-2)$-partition $\pi$ of shape $(p_2,...,p_n+n-2)$, we define the following sets
$$\kS_{\pi}(p_1,...,p_n):= \{\gamma \in \kS_{\pi}(\underbrace{1,...,1}_{n-1}), \, n(\gamma)=p_1\} $$
 and
$$\kS(p_1,...,p_n):= \{\gamma \in \kS_{\pi}(\underbrace{1,...,1}_{n-1}), 
\textrm{ for some } \pi \in \kS(p_2,...,p_n), \textrm{ s.t. }
n(\gamma)=p_1\}, $$
where $\kS_{\pi}(\underbrace{1,1,...,1}_{n-1})$ is the set of shifted $(n-1)$-partitions
 of shape $\pi$.

We can then state our conjecture for strongly stable ideals.

\begin{Conjecture}\label{ConjStrongStab}
There is a biunivocal correspondence between the set 
$\kS_{\pi}(p_1,...,p_n)$ and the set 
$\kB_{(p_1,...,p_n)}:=\{\cB \in \kA_n \textrm{ s.t. }  
\cL_\cB=(p_1,...,p_n),\, \eta(\cB)=\cN(J),\, J \textrm{ strongly  stable}     \}$.
\end{Conjecture}

\appendix
\section{Some explicit computation}
In example \ref{Spiegato} we have counted the (strongly) stable ideals 
in $\ck[x_1,x_2]$; in the next sections, we will count the stable (section \ref{SIapp}) and strongly stable ideals (section \ref{SSIapp})
 in $\mathbf{k}[x_1,x_2,x_3]$
with constant affine Hilbert polynomial $p=10.$

\subsection{Stable ideals}\label{SIapp}
Let us count the stable ideals in $\mathbf{k}[x_1,x_2,x_3]$
with constant affine Hilbert polynomial $p=10.$\\
\medskip 

By Corollary  \ref{ph11}  and Lemma \ref{minhmaxh}, the possible bar lists $(p=10,h,k)$  are:
\begin{enumerate}
\item $(10,1,1)$;
\item $(10,2,1)$;
\item $(10,3,1)$;
\item $(10,4,1)$;
\item $(10,3,2)$;
\item $(10,4,2)$;
\item $(10,5,2)$;
\item $(10,6,3)$.
\end{enumerate}
Indeed, for $k=1$, the maximal value for $h$ is 
$h= \left\lfloor \frac{-1+\sqrt{1+80}}{2}\right\rfloor =4 $; for $k=2$, using
Lemma \ref{minhmaxh}, 2., we can deduce that $h$ is an integer between
$\frac{k(k+1)}{2}=3$ and  $5$.\\
In order to deduce the maximal value $5$, we may notice that the only partitions of $6$ in $k=2$ distinct parts are $6=5+1=4+2$ and $\Sm([5,1])=16>p=10$, $\Sm([4,2])=13>p=10$. For $k=3$, using again 
Lemma \ref{minhmaxh}, 2., we can deduce that the minimal value for $h$ is $\frac{k(k+1)}{2}=6$ and that the maximal value for $h$ is again $6$. Indeed, the only partition of $7$ in $k=3$ distinct parts is
$7=4+2+1$ for which $\Sm([4,2,1])=14>p=10$.
\\
\smallskip
For $k=1$ above, we have (see  Corollary \ref{ph11}) $Q(10,1)+Q(10,2)+Q(10,3)+Q(10,4)=10$.
\\
Consider now $(10,3,2)$; the only possible shape\footnote{It is the only  possible partition of $3$ in two distinct parts.} is $\beta =(2,1)$, so we have 
\[ \left( \begin{array}{cccc}
\rho_{1,1} & \rho_{1,2}\\
\rho_{2,1} &0
\end{array} \right)\]
We need to take $a=(8,7)$ (see (1) of section \ref{COUNTSTAB}) and $b=(1,1)$ so that the determinant to compute is 
\[ det\left( \begin{array}{cccc}
x^3\bfrac{8}{2} & x^5 \bfrac{8}{3}\\
1 &x\bfrac{7}{1}
\end{array} \right)\]
% \[ 
% det\left( \begin{array}{cccc}
% x^{15}+x^{14}+2x^{13}+2x^{12}+3x^{11}+3x^{10}+4x^9+3x^8+3x^7+2x^6+2x^5+x^4+x^3 &  x^{20}+x^{19}+2x^{18}+3x^{17}+4x^{16}+5x^{15}+6x^{14}+6x^{13}+6x^{12}+6x^{11}+5x^{10}+4x^9+3x^8+2x^7+x^6+x^5
% \\
% 1 & x^7+x^6+x^5+x^4+x^3+x^2+x
% \end{array} \right)=\]
% 
and it gives $x^{22}+2x^{21}+3x^{20}+5x^{19}+7x^{18}+9x^{17}+12x^{16}+13x^{15}+14x^{14}+14x^{13}+14x^{12}+12x^{11}
+11x^{10}+8x^9+6x^8+4x^7+3x^6+x^5+x^4$, so we have $11$ stable ideals with this bar list.

As for $(10,4,2)$ we have  $\beta =(3,1)$, so 

\[\left( \begin{array}{ccc}
\rho_{1,1} & \rho_{1,2}& \rho_{1,3}\\
\rho_{1,2}&0&0
\end{array} \right)\]
We fix $a=(6,5)$ (see (1) of section \ref{COUNTSTAB}) and, by Theorem \ref{ContoKratStab}, we have 

$x^{20}+2x^{19}+4x^{18}+6x^{17}+9x^{16}+10x^{15}+12x^{14}+11x^{13}+
10x^{12}+8x^{11}+6x^{10}+3x^{9}+2x^{8}+x^{7}$, so $6$
 plane partitions of this shape.
 
Then take $(10,5,2)$; we have the partition below\footnote{Notice 
that also $\beta'=(4,1)$ is a potential shape; anyway there are no $(1,0)$-shifted plane partitions
of $10$ with shape $\beta'$.}
\[M= \left( \begin{array}{ccc}
\rho_{1,1} & \rho_{1,2}& \rho_{1,3}\\
\rho_{2,1} & \rho_{2,2}&0
\end{array} \right)\]
with $\beta=(3,2)$. Fixing $a=(4,3)$  (see (1) of section \ref{COUNTSTAB}), we get 
$x^{14}+2x^{13}+2x^{12}+2x^{11}+x^{10}+x^{9}$, so only one 
partition with norm $10$.
\\
We conclude with $(10,6,3)$, for which
we have
\[M= \left( \begin{array}{ccc}
\rho_{1,1} & \rho_{1,2}& \rho_{1,3}\\
\rho_{2,1}&\rho_{2,2}& 0\\
\rho_{3,1}& 0 & 0
\end{array} \right)\] 
with $\beta = (3,2,1)$; fixing $a=(3,2,1)$  (see again (1) of section \ref{COUNTSTAB}), we get $x^{10}$, so again only one plane partition
with this shape.
Summing up, we get $10+11+6+1+1=29$ stable ideals in $\ck[x_1,x_2,x_3]$, with 
affine Hilbert polynomial equal to $10$.
\begin{Remark}\label{Tedious1}
We notice that a tedious computation could allow us to list all $29$ plane partitions and the corresponding stable ideals.
 To show this we limit ourselves to consider the case $(10,4,2)$, for which there are exactly $6$ plane partitions:
\begin{enumerate}
 \item The plane partition \[\left( \begin{array}{ccc}
6 & 2& 1\\
1&0&0
\end{array} \right)\]

uniquely determines the Bar Code

 
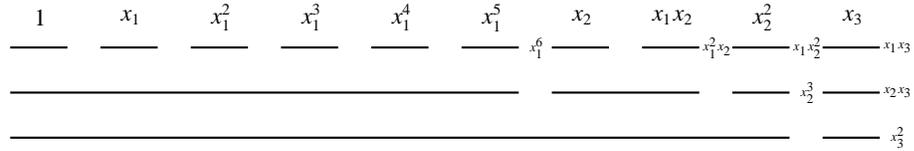
\begin{center}
\begin{tikzpicture}[scale=0.4]
 \draw [thick] (0,0) -- (25.9,0);
 \draw [thick] (27,0) -- (28.9,0);
 \node at (29.5,0) [] {${\scriptscriptstyle
x_3^2}$};
 \draw [thick] (0,1.5) -- (16.9,1.5);
 \draw [thick] (18,1.5) -- (22.9,1.5);
 \draw [thick] (24,1.5) -- (25.9,1.5);
 \draw [thick] (27,1.5) -- (28.9,1.5);
 
 \node at (26.5,1.5) [] {${\scriptscriptstyle
x_2^3}$};
 \node at (29.5,1.5) [] {${\scriptscriptstyle
x_2x_3}$};
%  \draw [thick] (9,1.5) -- (10.9,1.5);
 \draw [thick] (0,3.0) -- (1.9,3.0);
 \draw [thick] (3,3.0) -- (4.9,3.0);
 \draw [thick] (6,3.0) -- (7.9,3.0);
 \draw [thick] (9,3.0) -- (10.9,3.0);
 \draw [thick] (12,3.0) -- (13.9,3.0);
 \draw [thick] (15,3.0) -- (16.9,3.0);
 \draw [thick] (18,3.0) -- (19.9,3.0);

  \draw [thick] (21,3.0) -- (22.9,3.0);
 \draw [thick] (24,3.0) -- (25.9,3.0);
 \draw [thick] (27,3.0) -- (28.9,3.0);

\node at (17.5,3.0) [] {${\scriptscriptstyle
x_1^6}$};

\node at (23.5,3.0) [] {${\scriptscriptstyle
x_1^2x_2}$};
\node at (26.5,3.0) [] {${\scriptscriptstyle
x_1x_2^2}$};

 \node at (29.5,3.0) [] {${\scriptscriptstyle
x_1x_3}$};

\node at (1,4.0) [] {\small $1$};
 \node at (4,4.0) [] {\small $x_1$};
 \node at (7,4.0) [] {\small $x_1^2$};
 \node at (10,4.0) [] {\small $x_1^3$};
 \node at (13,4.0) [] {\small $x_1^4$};
 \node at (16,4.0) [] {\small $x_1^5$};
 \node at (19,4.0) [] {\small $x_2$};
 \node at (22,4.0) [] {\small $x_1x_2$};
 \node at (25,4.0) [] {\small $x_2^2$};
 \node at (28,4.0) [] {\small $x_3$};
\end{tikzpicture}
\end{center}

which corresponds to the stable ideal $I_1=(x_1^6,x_1^2x_2,x_1x_2^2,x_2^3,x_1x_3,x_2x_3,x_3^2)$;
 \item the plane partition \[\left( \begin{array}{ccc}
5 & 2& 1\\
2&0&0
\end{array} \right)\]

uniquely determines the Bar Code

 
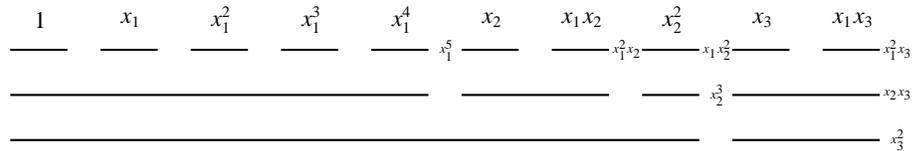
\begin{center}
\begin{tikzpicture}[scale=0.4]
 \draw [thick] (0,0) -- (22.9,0);
 \draw [thick] (24,0) -- (28.9,0);
 \node at (29.5,0) [] {${\scriptscriptstyle
x_3^2}$};
 \draw [thick] (0,1.5) -- (13.9,1.5);
 \draw [thick] (15,1.5) -- (19.9,1.5);
 \draw [thick] (21,1.5) -- (22.9,1.5);
 \draw [thick] (24,1.5) -- (28.9,1.5);
 
 \node at (23.5,1.5) [] {${\scriptscriptstyle
x_2^3}$};
 \node at (29.5,1.5) [] {${\scriptscriptstyle
x_2x_3}$};
%  \draw [thick] (9,1.5) -- (10.9,1.5);
 \draw [thick] (0,3.0) -- (1.9,3.0);
 \draw [thick] (3,3.0) -- (4.9,3.0);
 \draw [thick] (6,3.0) -- (7.9,3.0);
 \draw [thick] (9,3.0) -- (10.9,3.0);
 \draw [thick] (12,3.0) -- (13.9,3.0);
 \draw [thick] (15,3.0) -- (16.9,3.0);
 \draw [thick] (18,3.0) -- (19.9,3.0);

  \draw [thick] (21,3.0) -- (22.9,3.0);
 \draw [thick] (24,3.0) -- (25.9,3.0);
 \draw [thick] (27,3.0) -- (28.9,3.0);

\node at (14.5,3.0) [] {${\scriptscriptstyle
x_1^5}$};

\node at (20.5,3.0) [] {${\scriptscriptstyle
x_1^2x_2}$};
\node at (23.5,3.0) [] {${\scriptscriptstyle
x_1x_2^2}$};

 \node at (29.5,3.0) [] {${\scriptscriptstyle
x_1^2x_3}$};

\node at (1,4.0) [] {\small $1$};
 \node at (4,4.0) [] {\small $x_1$};
 \node at (7,4.0) [] {\small $x_1^2$};
 \node at (10,4.0) [] {\small $x_1^3$};
 \node at (13,4.0) [] {\small $x_1^4$};
 \node at (16,4.0) [] {\small $x_2$};
 \node at (19,4.0) [] {\small $x_1x_2$};
 \node at (22,4.0) [] {\small $x_2^2$};
 \node at (25,4.0) [] {\small $x_3$};
 \node at (28,4.0) [] {\small $x_1x_3$};
\end{tikzpicture}
\end{center}

which corresponds to the stable  ideal $I_2=(x_1^5,x_1^2x_2,x_1x_2^2,x_2^3,x_1^2x_3,x_2x_3,x_3^2)$;

\item the plane partition \[\left( \begin{array}{ccc}
5 &3&1\\
1&0&0
\end{array} \right)\]

uniquely determines the Bar Code

 
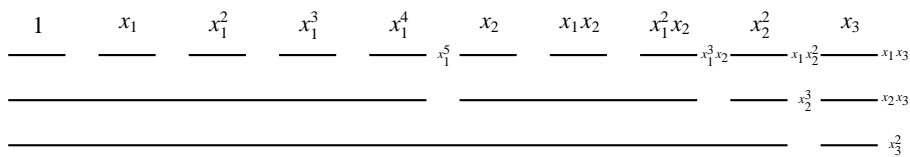
\begin{center}
\begin{tikzpicture}[scale=0.4]
 \draw [thick] (0,0) -- (25.9,0);
 \draw [thick] (27,0) -- (28.9,0);
 \node at (29.5,0) [] {${\scriptscriptstyle
x_3^2}$};
 \draw [thick] (0,1.5) -- (13.9,1.5);
 \draw [thick] (15,1.5) -- (22.9,1.5);
%  \draw [thick] (21,1.5) -- (22.9,1.5);
 \draw [thick] (24,1.5) -- (25.9,1.5);
 \draw [thick] (27,1.5) -- (28.9,1.5);

 \node at (26.5,1.5) [] {${\scriptscriptstyle
x_2^3}$};
 \node at (29.5,1.5) [] {${\scriptscriptstyle
x_2x_3}$};
%  \draw [thick] (9,1.5) -- (10.9,1.5);
 \draw [thick] (0,3.0) -- (1.9,3.0);
 \draw [thick] (3,3.0) -- (4.9,3.0);
 \draw [thick] (6,3.0) -- (7.9,3.0);
 \draw [thick] (9,3.0) -- (10.9,3.0);
 \draw [thick] (12,3.0) -- (13.9,3.0);
 \draw [thick] (15,3.0) -- (16.9,3.0);
 \draw [thick] (18,3.0) -- (19.9,3.0);

  \draw [thick] (21,3.0) -- (22.9,3.0);
 \draw [thick] (24,3.0) -- (25.9,3.0);
 \draw [thick] (27,3.0) -- (28.9,3.0);

\node at (14.5,3.0) [] {${\scriptscriptstyle
x_1^5}$};

\node at (23.5,3.0) [] {${\scriptscriptstyle
x_1^3x_2}$};
\node at (26.5,3.0) [] {${\scriptscriptstyle
x_1x_2^2}$};

 \node at (29.5,3.0) [] {${\scriptscriptstyle
x_1x_3}$};

\node at (1,4.0) [] {\small $1$};
 \node at (4,4.0) [] {\small $x_1$};
 \node at (7,4.0) [] {\small $x_1^2$};
 \node at (10,4.0) [] {\small $x_1^3$};
 \node at (13,4.0) [] {\small $x_1^4$};
 \node at (16,4.0) [] {\small $x_2$};
 \node at (19,4.0) [] {\small $x_1x_2$};
 \node at (22,4.0) [] {\small $x_1^2x_2$};
 \node at (25,4.0) [] {\small $x_2^2$};
 \node at (28,4.0) [] {\small $x_3$};
\end{tikzpicture}
\end{center}

which corresponds to the  stable ideal $I_3=(x_1^5,x_1^3x_2,x_1x_2^2,x_2^3,x_1x_3,x_2x_3,x_3^2)$;

\item the plane partition \[\left( \begin{array}{ccc}
4 & 3& 2\\
1&0&0
\end{array} \right)\]

uniquely determines the Bar Code

 
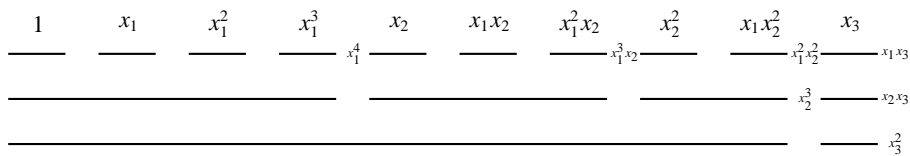
\begin{center}
\begin{tikzpicture}[scale=0.4]
 \draw [thick] (0,0) -- (25.9,0);
 \draw [thick] (27,0) -- (28.9,0);
 \node at (29.5,0) [] {${\scriptscriptstyle
x_3^2}$};
 \draw [thick] (0,1.5) -- (10.9,1.5);
 \draw [thick] (12,1.5) -- (19.9,1.5);
%  \draw [thick] (21,1.5) -- (22.9,1.5);
 \draw [thick] (21,1.5) -- (25.9,1.5);
 \draw [thick] (27,1.5) -- (28.9,1.5);

 \node at (26.5,1.5) [] {${\scriptscriptstyle
x_2^3}$};
 \node at (29.5,1.5) [] {${\scriptscriptstyle
x_2x_3}$};
%  \draw [thick] (9,1.5) -- (10.9,1.5);
 \draw [thick] (0,3.0) -- (1.9,3.0);
 \draw [thick] (3,3.0) -- (4.9,3.0);
 \draw [thick] (6,3.0) -- (7.9,3.0);
 \draw [thick] (9,3.0) -- (10.9,3.0);
 \draw [thick] (12,3.0) -- (13.9,3.0);
 \draw [thick] (15,3.0) -- (16.9,3.0);
 \draw [thick] (18,3.0) -- (19.9,3.0);

  \draw [thick] (21,3.0) -- (22.9,3.0);
 \draw [thick] (24,3.0) -- (25.9,3.0);
 \draw [thick] (27,3.0) -- (28.9,3.0);

\node at (11.5,3.0) [] {${\scriptscriptstyle
x_1^4}$};

\node at (20.5,3.0) [] {${\scriptscriptstyle
x_1^3x_2}$};
\node at (26.5,3.0) [] {${\scriptscriptstyle
x_1^2x_2^2}$};

 \node at (29.5,3.0) [] {${\scriptscriptstyle
x_1x_3}$};

\node at (1,4.0) [] {\small $1$};
 \node at (4,4.0) [] {\small $x_1$};
 \node at (7,4.0) [] {\small $x_1^2$};
 \node at (10,4.0) [] {\small $x_1^3$};
 \node at (13,4.0) [] {\small $x_2$};
 \node at (16,4.0) [] {\small $x_1x_2$};
 \node at (19,4.0) [] {\small $x_1^2x_2$};
 \node at (22,4.0) [] {\small $x_2^2$};
 \node at (25,4.0) [] {\small $x_1x_2^2$};
 \node at (28,4.0) [] {\small $x_3$};
\end{tikzpicture}
\end{center}

which corresponds to the stable  ideal $I_4=(  x_1^4,x_1^3x_2,x_1^2x_2^2,x_2^3,x_1x_3,x_2x_3,x_3^2)$;

\item the plane partition \[\left( \begin{array}{ccc}
4 & 2& 1\\
3&0&0
\end{array} \right)\]

uniquely determines the Bar Code

 
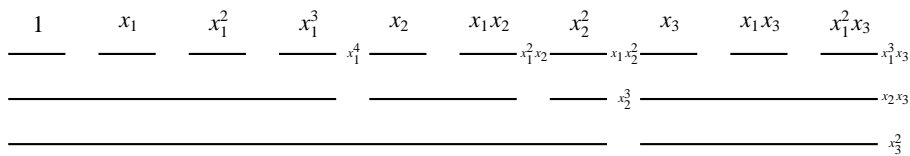
\begin{center}
\begin{tikzpicture}[scale=0.4]
 \draw [thick] (0,0) -- (19.9,0);
 \draw [thick] (21,0) -- (28.9,0);
 \node at (29.5,0) [] {${\scriptscriptstyle
x_3^2}$};
 \draw [thick] (0,1.5) -- (10.9,1.5);
 \draw [thick] (12,1.5) -- (16.9,1.5);
  \draw [thick] (18,1.5) -- (19.9,1.5);
%  \draw [thick] (21,1.5) -- (25.9,1.5);
 \draw [thick] (21,1.5) -- (28.9,1.5);

 \node at (20.5,1.5) [] {${\scriptscriptstyle
x_2^3}$};
 \node at (29.5,1.5) [] {${\scriptscriptstyle
x_2x_3}$};
%  \draw [thick] (9,1.5) -- (10.9,1.5);
 \draw [thick] (0,3.0) -- (1.9,3.0);
 \draw [thick] (3,3.0) -- (4.9,3.0);
 \draw [thick] (6,3.0) -- (7.9,3.0);
 \draw [thick] (9,3.0) -- (10.9,3.0);
 \draw [thick] (12,3.0) -- (13.9,3.0);
 \draw [thick] (15,3.0) -- (16.9,3.0);
 \draw [thick] (18,3.0) -- (19.9,3.0);

  \draw [thick] (21,3.0) -- (22.9,3.0);
 \draw [thick] (24,3.0) -- (25.9,3.0);
 \draw [thick] (27,3.0) -- (28.9,3.0);

\node at (11.5,3.0) [] {${\scriptscriptstyle
x_1^4}$};

\node at (17.5,3.0) [] {${\scriptscriptstyle
x_1^2x_2}$};
\node at (20.5,3.0) [] {${\scriptscriptstyle
x_1x_2^2}$};

 \node at (29.5,3.0) [] {${\scriptscriptstyle
x_1^3x_3}$};

\node at (1,4.0) [] {\small $1$};
 \node at (4,4.0) [] {\small $x_1$};
 \node at (7,4.0) [] {\small $x_1^2$};
 \node at (10,4.0) [] {\small $x_1^3$};
 \node at (13,4.0) [] {\small $x_2$};
 \node at (16,4.0) [] {\small $x_1x_2$};
 \node at (19,4.0) [] {\small $x_2^2$};
 \node at (22,4.0) [] {\small $x_3$};
 \node at (25,4.0) [] {\small $x_1x_3$};
 \node at (28,4.0) [] {\small $x_1^2x_3$};
\end{tikzpicture}
\end{center}

which corresponds to the  stable  ideal $I_5=( x_1^4,x_1^2x_2,x_1x_2^2,x_2^3,x_1^3x_3,x_2x_3,x_3^2)$;

\item the plane partition \[\left( \begin{array}{ccc}
4 & 3 &1\\
2&0&0
\end{array} \right)\]
\end{enumerate}

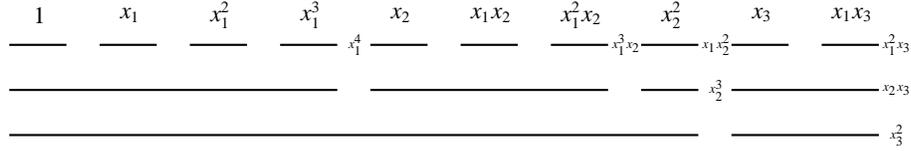
 \begin{center}
\begin{tikzpicture}[scale=0.4]
 \draw [thick] (0,0) -- (22.9,0);
 \draw [thick] (24,0) -- (28.9,0);
 \node at (29.5,0) [] {${\scriptscriptstyle
x_3^2}$};
 \draw [thick] (0,1.5) -- (10.9,1.5);
 \draw [thick] (12,1.5) -- (19.9,1.5);
%   \draw [thick] (18,1.5) -- (19.9,1.5);
%  \draw [thick] (21,1.5) -- (25.9,1.5);
 \draw [thick] (21,1.5) -- (22.9,1.5);
 \draw [thick] (24,1.5) -- (28.9,1.5);

 \node at (23.5,1.5) [] {${\scriptscriptstyle
x_2^3}$};
 \node at (29.5,1.5) [] {${\scriptscriptstyle
x_2x_3}$};
%  \draw [thick] (9,1.5) -- (10.9,1.5);
 \draw [thick] (0,3.0) -- (1.9,3.0);
 \draw [thick] (3,3.0) -- (4.9,3.0);
 \draw [thick] (6,3.0) -- (7.9,3.0);
 \draw [thick] (9,3.0) -- (10.9,3.0);
 \draw [thick] (12,3.0) -- (13.9,3.0);
 \draw [thick] (15,3.0) -- (16.9,3.0);
 \draw [thick] (18,3.0) -- (19.9,3.0);

  \draw [thick] (21,3.0) -- (22.9,3.0);
 \draw [thick] (24,3.0) -- (25.9,3.0);
 \draw [thick] (27,3.0) -- (28.9,3.0);

\node at (11.5,3.0) [] {${\scriptscriptstyle
x_1^4}$};

\node at (20.5,3.0) [] {${\scriptscriptstyle
x_1^3x_2}$};
\node at (23.5,3.0) [] {${\scriptscriptstyle
x_1x_2^2}$};

 \node at (29.5,3.0) [] {${\scriptscriptstyle
x_1^2x_3}$};

\node at (1,4.0) [] {\small $1$};
 \node at (4,4.0) [] {\small $x_1$};
 \node at (7,4.0) [] {\small $x_1^2$};
 \node at (10,4.0) [] {\small $x_1^3$};
 \node at (13,4.0) [] {\small $x_2$};
 \node at (16,4.0) [] {\small $x_1x_2$};
 \node at (19,4.0) [] {\small $x_1^2x_2$};
 \node at (22,4.0) [] {\small $x_2^2$};
 \node at (25,4.0) [] {\small $x_3$};
 \node at (28,4.0) [] {\small $x_1x_3$};
\end{tikzpicture}
\end{center}

 which corresponds to the  stable ideal $I_6=(  x_1^4,x_1^3x_2,x_1x_2^2,x_2^3,x_1^2x_3,x_2x_3,x_3^2)$;
\end{Remark}

\subsection{Strongly stable ideals}\label{SSIapp}
Let us count the strongly stable ideals in $\mathbf{k}[x_1,x_2,x_3]$
with constant affine Hilbert polynomial $p=10.$\\
\medskip 

By Corollary  \ref{ph11}  and Lemma \ref{minhmaxh}, the possible bar lists, as for the case of stable ideals,
 are:

\begin{enumerate}
\item $(10,1,1)$;
\item $(10,2,1)$;
\item $(10,3,1)$;
\item $(10,4,1)$;
\item $(10,3,2)$;
\item $(10,4,2)$;
\item $(10,5,2)$;
\item $(10,6,3)$.
\end{enumerate}

For $k=1$ above, we proceed as for stable ideals, thanks to the equivalence of Lemma \ref{StabEqStrStab}, getting $Q(10,1)+Q(10,2)+Q(10,3)+Q(10,4)=10$.
\\
Consider now $(10,3,2)$, for which we have the partition below
\[ \left( \begin{array}{cccc}
a_{1,1} & a_{1,2}\\
0& a_{2,2}
\end{array} \right)\] 
so $\lambda=(2,2)$, $r=2$, $\mathbf{M}=8$, $a_2=1,...,7$  and $a_1=a_2+1,...,8$ (see (2) in section \ref{StStCount}). We report here only the computations
giving nonzero result:
\begin{enumerate}
  \item $a=(5,1)$: $N_1=7$ and
  \[M= \left( \begin{array}{cccc}
x^3+x^2+x+1 & 0\\
1& 1
\end{array} \right)\] 
so that $x^{N_1}det(M)=x^7(x^3+x^2+x+1)$. Therefore there is one such plane 
partition.
  \item $a=(6,1)$: $N_1=8$ and
  \[M= \left( \begin{array}{cccc}
x^4+x^3+x^2+x+1 & 0\\
1& 1
\end{array} \right)\] 
so that $x^{N_1}det(M)=x^8(x^4+x^3+x^2+x+1)$. Therefore there is one such plane
partition.
  \item $a=(7,1)$: $N_1=9$ and
  \[M= \left( \begin{array}{cccc}
x^5+x^4+x^3+x^2+x+1 & 0\\
1& 1
\end{array} \right)\] 
so that $x^{N_1}det(M)=x^9(x^5+x^4+x^3+x^2+x+1)$. Therefore there is one such 
plane
partition.
  \item $a=(8,1)$: $N_1=10$ and
  \[M= \left( \begin{array}{cccc}
x^6+x^5+x^4+x^3+x^2+x+1 & 0\\
1& 1
\end{array} \right)\] 
so that $x^{N_1}det(M)=x^{10}(x^6+x^5+x^4+x^3+x^2+x+1)$. Therefore there is one 
such
plane partition.
    \item $a=(5,2)$: $N_1=8$ and
  \[M= \left( \begin{array}{cccc}
x^3+x^2+x+1 & 1\\
1& 1
\end{array} \right)\] 
so that $x^{N_1}det(M)=x^8(x^3+x^2+x)$. Therefore there is one such plane 
partition.
  
  \item $a=(6,2)$: $N_1=9$ and
  \[M= \left( \begin{array}{cccc}
x^4+x^3+x^2+x+1 & 1\\
1& 1
\end{array} \right)\] 
so that $x^{N_1}det(M)=x^9(x^4+x^3+x^2+x)$. Therefore there is one such plane 
partition.

  \item $a=(4,3)$: $N_1=8$ and
  \[M= \left( \begin{array}{cccc}
x^3+x^2+x+1 & x+1\\
1& 1
\end{array} \right)\] 
so that $x^{N_1}det(M)=x^8\cdot x^2$. Therefore there is one such plane 
partition.
\end{enumerate}
The total number we get of the partitions of type 
\[ \left( \begin{array}{cccc}
a_{1,1} & a_{1,2}\\
0& a_{2,2}
\end{array} \right)\] 
 is $7$.
\\
We will see below that the plane partitions of this shape can actually be 
counted in a simpler way.
\\
Take then $(10,4,2)$
\\
Since $4=3+1$, we only have to deal with the partitions below
\[M= \left( \begin{array}{ccc}
a_{1,1} & a_{1,2}&a_{1,3}\\
0& a_{2,2}&0
\end{array} \right),\]
 so $\lambda=(3,2)$, $r=2$, $\mathbf{M}=6$, $a_2=1,...,5$ and $a_1=a_2+1,...,6$ (see (2) in section \ref{StStCount}). We report here only the computations
giving nonzero result:
\begin{enumerate}
 \item $a=(4,1)$, $N_1=8$ and
 \[M= \left( \begin{array}{cc}
x^2+x+1 & 0\\
1& 1
\end{array} \right),\]
so that $x^8det(M)=x^8(x^2+x+1)$. Therefore there is only one such plane
partition.
 
  \item $a=(5,1)$, $N_1=9$ and
 \[M= \left( \begin{array}{cc}
(x^2+x+1)(x^2+1) & 0\\
1& 1
\end{array} \right),\]
so that $x^8det(M)=x^9(x^2+x+1)(x^2+1)$. Therefore there is only one such plane
partition.

  \item $a=(5,1)$, $N_1=10$ and
 \[M= \left( \begin{array}{cc}
(x^4+x^3+x^2+x+1)(x^2+1) & 0\\
1& 1
\end{array} \right),\]
so that $x^8det(M)=x^{10}(x^4+x^3+x^2+x+1)(x^2+1))$. Therefore there is only one
such plane
partition.

  \item $a=(4,2)$, $N_1=9$ and
 \[M= \left( \begin{array}{cc}
(x^4+x^3+x^2+x+1)(x^2+1) & 0\\
1& 1
\end{array} \right),\]
so that $x^8det(M)=x^{9}(x^4+x^3+x^2+x+1)(x^2+1))$. Therefore there is only one
such plane
partition.

  \item $a=(5,2)$, $N_1=10$ and
 \[M= \left( \begin{array}{cc}
(x^2+x+1)(x^2+1) & 0\\
1& 1
\end{array} \right),\]
so that $x^8det(M)=x^{10}(x^2+x+1)(x^2+1))$. Therefore there is only one
such plane
partition.
\end{enumerate}
The total number of the partitions of type 
\[ \left( \begin{array}{cccc}
a_{1,1} & a_{1,2}&a_{1,3}\\
0& a_{2,2}&0
\end{array} \right)\] 
is $5$.\\
Consider now $(10,5,2)$. We have the partition below
\[M= \left( \begin{array}{ccc}
a_{1,1} & a_{1,2}&a_{1,3}\\
0& a_{2,2}&a_{2,3}
\end{array} \right)\]
In this case $\lambda=(3,3)$, $r=2$, $\mathbf{M}=4$ and there is only one partition of 
this shape, coming from $a=(4,2)$ (see (2) in section \ref{StStCount}). Indeed, in this case $N_1=10$,
\[M= \left( \begin{array}{ccc}
x^2+x+1&0\\
x^2+x+1& 1
\end{array} \right)\]
and we get $x^{N_1}det(M)=x^{10}(x^2+x+1)$.
\\
We conclude with $(10,6,3)$, for which
by $6=3+2+1$. We obtain the matrix
\[M= \left( \begin{array}{ccc}
a_{1,1} & a_{1,2}&a_{1,3}\\
0& a_{2,2}&a_{2,3}\\
0&0&a_{3,3}
\end{array} \right)\]
for which $\lambda=(3,3,3)$, $r=3$, $b=(1,1,1)$ and $\mathbf{M}=3$.
It holds then $a_3=1$, $a_2=2$, $a_1=3$, i.e. there is only one vector $a$ to 
examine (see (2) in section \ref{StStCount}).
For $a=(3,2,1)$ we get $N_1=10$ and 
\[M= \left( \begin{array}{ccc}
1&0&0\\
x+1& 1&0\\
1&1&1
\end{array} \right)\]
so that $x^{10}det(M)=x^{10}$. We get only one plane partition of norm $10$ of 
this shape.\\ 
In conclusion we have exactly $24$ strongly stable ideals in $3$ variables with 
constant 
affine Hilbert polynomial $H_{\_}(t)=10$.
\begin{Remark}\label{Tedious1}
We notice that a tedious computation could allow us to list all $24$ plane partitions and the corresponding strongly stable ideals.
 To show this we limit ourselves to consider the case $(10,4,2)$, for which there are exactly $5$ plane partitions:
\begin{enumerate}
 \item The plane partition \[\left( \begin{array}{ccc}
6 & 2& 1\\
0&1&0
\end{array} \right)\]

uniquely determines the Bar Code

 
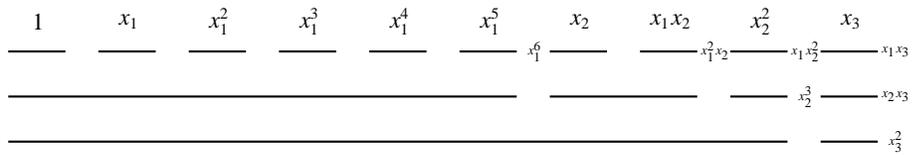
\begin{center}
\begin{tikzpicture}[scale=0.4]
 \draw [thick] (0,0) -- (25.9,0);
 \draw [thick] (27,0) -- (28.9,0);
 \node at (29.5,0) [] {${\scriptscriptstyle
x_3^2}$};
 \draw [thick] (0,1.5) -- (16.9,1.5);
 \draw [thick] (18,1.5) -- (22.9,1.5);
 \draw [thick] (24,1.5) -- (25.9,1.5);
 \draw [thick] (27,1.5) -- (28.9,1.5);
 
 \node at (26.5,1.5) [] {${\scriptscriptstyle
x_2^3}$};
 \node at (29.5,1.5) [] {${\scriptscriptstyle
x_2x_3}$};
%  \draw [thick] (9,1.5) -- (10.9,1.5);
 \draw [thick] (0,3.0) -- (1.9,3.0);
 \draw [thick] (3,3.0) -- (4.9,3.0);
 \draw [thick] (6,3.0) -- (7.9,3.0);
 \draw [thick] (9,3.0) -- (10.9,3.0);
 \draw [thick] (12,3.0) -- (13.9,3.0);
 \draw [thick] (15,3.0) -- (16.9,3.0);
 \draw [thick] (18,3.0) -- (19.9,3.0);

  \draw [thick] (21,3.0) -- (22.9,3.0);
 \draw [thick] (24,3.0) -- (25.9,3.0);
 \draw [thick] (27,3.0) -- (28.9,3.0);

\node at (17.5,3.0) [] {${\scriptscriptstyle
x_1^6}$};

\node at (23.5,3.0) [] {${\scriptscriptstyle
x_1^2x_2}$};
\node at (26.5,3.0) [] {${\scriptscriptstyle
x_1x_2^2}$};

 \node at (29.5,3.0) [] {${\scriptscriptstyle
x_1x_3}$};

\node at (1,4.0) [] {\small $1$};
 \node at (4,4.0) [] {\small $x_1$};
 \node at (7,4.0) [] {\small $x_1^2$};
 \node at (10,4.0) [] {\small $x_1^3$};
 \node at (13,4.0) [] {\small $x_1^4$};
 \node at (16,4.0) [] {\small $x_1^5$};
 \node at (19,4.0) [] {\small $x_2$};
 \node at (22,4.0) [] {\small $x_1x_2$};
 \node at (25,4.0) [] {\small $x_2^2$};
 \node at (28,4.0) [] {\small $x_3$};
\end{tikzpicture}
\end{center}

which corresponds to the stable ideal $I_1=(x_1^6,x_1^2x_2,x_1x_2^2,x_2^3,x_1x_3,x_2x_3,x_3^2)$;
 \item the plane partition \[\left( \begin{array}{ccc}
5 & 2& 1\\
0&2&0
\end{array} \right)\]

uniquely determines the Bar Code

 
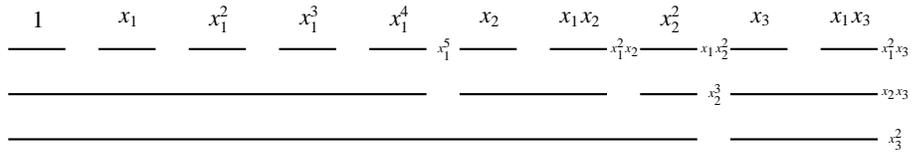
\begin{center}
\begin{tikzpicture}[scale=0.4]
 \draw [thick] (0,0) -- (22.9,0);
 \draw [thick] (24,0) -- (28.9,0);
 \node at (29.5,0) [] {${\scriptscriptstyle
x_3^2}$};
 \draw [thick] (0,1.5) -- (13.9,1.5);
 \draw [thick] (15,1.5) -- (19.9,1.5);
 \draw [thick] (21,1.5) -- (22.9,1.5);
 \draw [thick] (24,1.5) -- (28.9,1.5);
 
 \node at (23.5,1.5) [] {${\scriptscriptstyle
x_2^3}$};
 \node at (29.5,1.5) [] {${\scriptscriptstyle
x_2x_3}$};
%  \draw [thick] (9,1.5) -- (10.9,1.5);
 \draw [thick] (0,3.0) -- (1.9,3.0);
 \draw [thick] (3,3.0) -- (4.9,3.0);
 \draw [thick] (6,3.0) -- (7.9,3.0);
 \draw [thick] (9,3.0) -- (10.9,3.0);
 \draw [thick] (12,3.0) -- (13.9,3.0);
 \draw [thick] (15,3.0) -- (16.9,3.0);
 \draw [thick] (18,3.0) -- (19.9,3.0);

  \draw [thick] (21,3.0) -- (22.9,3.0);
 \draw [thick] (24,3.0) -- (25.9,3.0);
 \draw [thick] (27,3.0) -- (28.9,3.0);

\node at (14.5,3.0) [] {${\scriptscriptstyle
x_1^5}$};

\node at (20.5,3.0) [] {${\scriptscriptstyle
x_1^2x_2}$};
\node at (23.5,3.0) [] {${\scriptscriptstyle
x_1x_2^2}$};

 \node at (29.5,3.0) [] {${\scriptscriptstyle
x_1^2x_3}$};

\node at (1,4.0) [] {\small $1$};
 \node at (4,4.0) [] {\small $x_1$};
 \node at (7,4.0) [] {\small $x_1^2$};
 \node at (10,4.0) [] {\small $x_1^3$};
 \node at (13,4.0) [] {\small $x_1^4$};
 \node at (16,4.0) [] {\small $x_2$};
 \node at (19,4.0) [] {\small $x_1x_2$};
 \node at (22,4.0) [] {\small $x_2^2$};
 \node at (25,4.0) [] {\small $x_3$};
 \node at (28,4.0) [] {\small $x_1x_3$};
\end{tikzpicture}
\end{center}

which corresponds to the stable  ideal $I_2=(x_1^5,x_1^2x_2,x_1x_2^2,x_2^3,x_1^2x_3,x_2x_3,x_3^2)$;

\item the plane partition \[\left( \begin{array}{ccc}
5 &3&1\\
0&1&0
\end{array} \right)\]

uniquely determines the Bar Code

 
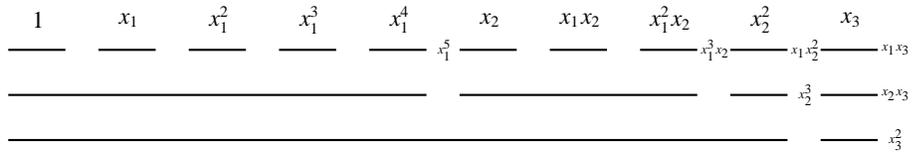
\begin{center}
\begin{tikzpicture}[scale=0.4]
 \draw [thick] (0,0) -- (25.9,0);
 \draw [thick] (27,0) -- (28.9,0);
 \node at (29.5,0) [] {${\scriptscriptstyle
x_3^2}$};
 \draw [thick] (0,1.5) -- (13.9,1.5);
 \draw [thick] (15,1.5) -- (22.9,1.5);
%  \draw [thick] (21,1.5) -- (22.9,1.5);
 \draw [thick] (24,1.5) -- (25.9,1.5);
 \draw [thick] (27,1.5) -- (28.9,1.5);

 \node at (26.5,1.5) [] {${\scriptscriptstyle
x_2^3}$};
 \node at (29.5,1.5) [] {${\scriptscriptstyle
x_2x_3}$};
%  \draw [thick] (9,1.5) -- (10.9,1.5);
 \draw [thick] (0,3.0) -- (1.9,3.0);
 \draw [thick] (3,3.0) -- (4.9,3.0);
 \draw [thick] (6,3.0) -- (7.9,3.0);
 \draw [thick] (9,3.0) -- (10.9,3.0);
 \draw [thick] (12,3.0) -- (13.9,3.0);
 \draw [thick] (15,3.0) -- (16.9,3.0);
 \draw [thick] (18,3.0) -- (19.9,3.0);

  \draw [thick] (21,3.0) -- (22.9,3.0);
 \draw [thick] (24,3.0) -- (25.9,3.0);
 \draw [thick] (27,3.0) -- (28.9,3.0);

\node at (14.5,3.0) [] {${\scriptscriptstyle
x_1^5}$};

\node at (23.5,3.0) [] {${\scriptscriptstyle
x_1^3x_2}$};
\node at (26.5,3.0) [] {${\scriptscriptstyle
x_1x_2^2}$};

 \node at (29.5,3.0) [] {${\scriptscriptstyle
x_1x_3}$};

\node at (1,4.0) [] {\small $1$};
 \node at (4,4.0) [] {\small $x_1$};
 \node at (7,4.0) [] {\small $x_1^2$};
 \node at (10,4.0) [] {\small $x_1^3$};
 \node at (13,4.0) [] {\small $x_1^4$};
 \node at (16,4.0) [] {\small $x_2$};
 \node at (19,4.0) [] {\small $x_1x_2$};
 \node at (22,4.0) [] {\small $x_1^2x_2$};
 \node at (25,4.0) [] {\small $x_2^2$};
 \node at (28,4.0) [] {\small $x_3$};
\end{tikzpicture}
\end{center}

which corresponds to the  stable ideal $I_3=(x_1^5,x_1^3x_2,x_1x_2^2,x_2^3,x_1x_3,x_2x_3,x_3^2)$;

\item the plane partition \[\left( \begin{array}{ccc}
4 & 3& 2\\
0&1&0
\end{array} \right)\]

uniquely determines the Bar Code

 
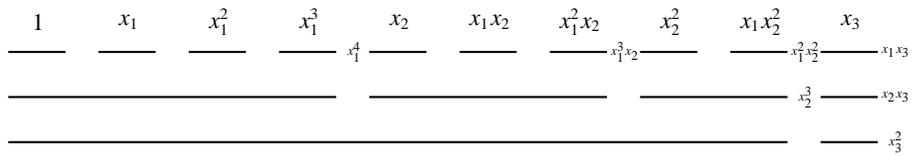
\begin{center}
\begin{tikzpicture}[scale=0.4]
 \draw [thick] (0,0) -- (25.9,0);
 \draw [thick] (27,0) -- (28.9,0);
 \node at (29.5,0) [] {${\scriptscriptstyle
x_3^2}$};
 \draw [thick] (0,1.5) -- (10.9,1.5);
 \draw [thick] (12,1.5) -- (19.9,1.5);
%  \draw [thick] (21,1.5) -- (22.9,1.5);
 \draw [thick] (21,1.5) -- (25.9,1.5);
 \draw [thick] (27,1.5) -- (28.9,1.5);

 \node at (26.5,1.5) [] {${\scriptscriptstyle
x_2^3}$};
 \node at (29.5,1.5) [] {${\scriptscriptstyle
x_2x_3}$};
%  \draw [thick] (9,1.5) -- (10.9,1.5);
 \draw [thick] (0,3.0) -- (1.9,3.0);
 \draw [thick] (3,3.0) -- (4.9,3.0);
 \draw [thick] (6,3.0) -- (7.9,3.0);
 \draw [thick] (9,3.0) -- (10.9,3.0);
 \draw [thick] (12,3.0) -- (13.9,3.0);
 \draw [thick] (15,3.0) -- (16.9,3.0);
 \draw [thick] (18,3.0) -- (19.9,3.0);

  \draw [thick] (21,3.0) -- (22.9,3.0);
 \draw [thick] (24,3.0) -- (25.9,3.0);
 \draw [thick] (27,3.0) -- (28.9,3.0);

\node at (11.5,3.0) [] {${\scriptscriptstyle
x_1^4}$};

\node at (20.5,3.0) [] {${\scriptscriptstyle
x_1^3x_2}$};
\node at (26.5,3.0) [] {${\scriptscriptstyle
x_1^2x_2^2}$};

 \node at (29.5,3.0) [] {${\scriptscriptstyle
x_1x_3}$};

\node at (1,4.0) [] {\small $1$};
 \node at (4,4.0) [] {\small $x_1$};
 \node at (7,4.0) [] {\small $x_1^2$};
 \node at (10,4.0) [] {\small $x_1^3$};
 \node at (13,4.0) [] {\small $x_2$};
 \node at (16,4.0) [] {\small $x_1x_2$};
 \node at (19,4.0) [] {\small $x_1^2x_2$};
 \node at (22,4.0) [] {\small $x_2^2$};
 \node at (25,4.0) [] {\small $x_1x_2^2$};
 \node at (28,4.0) [] {\small $x_3$};
\end{tikzpicture}
\end{center}

which corresponds to the stable  ideal $I_4=(  x_1^4,x_1^3x_2,x_1^2x_2^2,x_2^3,x_1x_3,x_2x_3,x_3^2)$;

\item the plane partition \[\left( \begin{array}{ccc}
4 & 3 &1\\
0&2&0
\end{array} \right)\]
\end{enumerate}

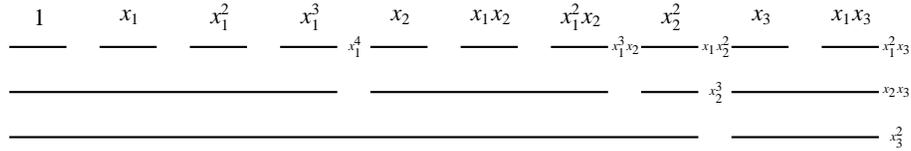
 \begin{center}
\begin{tikzpicture}[scale=0.4]
 \draw [thick] (0,0) -- (22.9,0);
 \draw [thick] (24,0) -- (28.9,0);
 \node at (29.5,0) [] {${\scriptscriptstyle
x_3^2}$};
 \draw [thick] (0,1.5) -- (10.9,1.5);
 \draw [thick] (12,1.5) -- (19.9,1.5);
%   \draw [thick] (18,1.5) -- (19.9,1.5);
%  \draw [thick] (21,1.5) -- (25.9,1.5);
 \draw [thick] (21,1.5) -- (22.9,1.5);
 \draw [thick] (24,1.5) -- (28.9,1.5);

 \node at (23.5,1.5) [] {${\scriptscriptstyle
x_2^3}$};
 \node at (29.5,1.5) [] {${\scriptscriptstyle
x_2x_3}$};
%  \draw [thick] (9,1.5) -- (10.9,1.5);
 \draw [thick] (0,3.0) -- (1.9,3.0);
 \draw [thick] (3,3.0) -- (4.9,3.0);
 \draw [thick] (6,3.0) -- (7.9,3.0);
 \draw [thick] (9,3.0) -- (10.9,3.0);
 \draw [thick] (12,3.0) -- (13.9,3.0);
 \draw [thick] (15,3.0) -- (16.9,3.0);
 \draw [thick] (18,3.0) -- (19.9,3.0);

  \draw [thick] (21,3.0) -- (22.9,3.0);
 \draw [thick] (24,3.0) -- (25.9,3.0);
 \draw [thick] (27,3.0) -- (28.9,3.0);

\node at (11.5,3.0) [] {${\scriptscriptstyle
x_1^4}$};

\node at (20.5,3.0) [] {${\scriptscriptstyle
x_1^3x_2}$};
\node at (23.5,3.0) [] {${\scriptscriptstyle
x_1x_2^2}$};

 \node at (29.5,3.0) [] {${\scriptscriptstyle
x_1^2x_3}$};

\node at (1,4.0) [] {\small $1$};
 \node at (4,4.0) [] {\small $x_1$};
 \node at (7,4.0) [] {\small $x_1^2$};
 \node at (10,4.0) [] {\small $x_1^3$};
 \node at (13,4.0) [] {\small $x_2$};
 \node at (16,4.0) [] {\small $x_1x_2$};
 \node at (19,4.0) [] {\small $x_1^2x_2$};
 \node at (22,4.0) [] {\small $x_2^2$};
 \node at (25,4.0) [] {\small $x_3$};
 \node at (28,4.0) [] {\small $x_1x_3$};
\end{tikzpicture}
\end{center}

 which corresponds to the  stable ideal $I_5=(  x_1^4,x_1^3x_2,x_1x_2^2,x_2^3,x_1^2x_3,x_2x_3,x_3^2)$;
\end{Remark}

%  \hrule 
% \bigskip
%  
%  
%  \hrule
% 
% \textbf{PIANO OPERA:}
% 
% \begin{enumerate}
%  \item Linz
%  \item CeMu-AoE JSC se risultato e' veramente come penso
%  \item char p
%  \item BC infinito rifare linz allora vedere $t,2t, 2t+1, t+1$
% \item Janet: pericolanti o simili
%  \end{enumerate}

\end{document}